\newif\ifiop
\iopfalse
\ifiop
\documentclass{iopjournal}
\else
\documentclass[10pt,parskip=half]{scrartcl}
\fi
\usepackage{amsmath,amsfonts,amsthm,latexsym,graphicx,paralist}
\usepackage[utf8]{inputenc}
\ifiop
\hypersetup{
  citecolor=teal,
  linkcolor=red,
  urlcolor=blue}
\usepackage[style=numeric-comp,sorting=none]{biblatex}
\else
\usepackage{xcolor,hyperref}
\hypersetup{
  colorlinks,
  citecolor=teal,
  linkcolor=red,
  urlcolor=blue}
\usepackage[citestyle=alphabetic,style=alphabetic]{biblatex}
\fi
\usepackage[capitalise]{cleveref}
\usepackage{derivative}
\usepackage{siunitx}

\newcommand{\prob}{{\mathbb P}}
\newcommand{\ev}[1]{ {\boldsymbol{\mathsf  E}} \left[  #1 \right]}
\newcommand{\evalt}[2]{{\boldsymbol{\mathsf  E}}_{#1} \left[ #2 \right]}

\newcommand{\vc}{\boldsymbol{c}}

\newcommand{\vx}{\boldsymbol{x}}
\newcommand{\vy}{\boldsymbol{y}}
\newcommand{\vz}{\boldsymbol{z}}
\newcommand{\vxi}{\boldsymbol{\xi}}
\newcommand{\vzeta}{\boldsymbol{\zeta}}
\newcommand{\vvarepsilon}{\boldsymbol{\varepsilon}}
\newcommand{\nhidden}{n_\text{hidden}}
\newcommand{\vV}{\boldsymbol{V}}

\newcommand{\vI}{\boldsymbol{I}}

\newcommand{\DD}{{\mathcal D}}

\newcommand{\UU}{{\mathcal U}}

\newcommand{\R}{{\mathbb R}}

\usepackage{mathtools}
\DeclarePairedDelimiter{\norm}{\lVert}{\rVert}
\DeclarePairedDelimiter\abs{\lvert}{\rvert}

\newcommand{\dkl}[2]{D_\text{KL}(#1 \| \,#2)}
\newcommand{\dtv}[2]{d_\text{TV}(#1,#2)}

\newcommand{\nkl}{n_{\text{KL}}}
\newcommand{\nobs}{n_{\text{obs}}}

\theoremstyle{plain}%
\newtheorem{thm}{Theorem}[section]

\newtheorem{prop}[thm]{Proposition}

\theoremstyle{definition}

\newtheorem{remark}{Remark}[section]

\theoremstyle{remark}

\renewcommand{\R}{\mathbb{R}}

\newcommand{\NN}{\mathcal{N}}

\usepackage{algorithm}
\usepackage{algpseudocode}
\usepackage{subcaption}
\usepackage{pdflscape}
\usepackage{booktabs}
\ifiop\else
\usepackage[a4paper,textwidth=153mm,textheight=240mm,centering]{geometry}
\setkomafont{author}{\large}
\fi

\renewcommand{\mid}{ \hspace{0.1em}\vert\hspace{0.1em}}
\renewcommand{\R}{\mathbb R}

\renewcommand{\NN}{\mathcal N}

\renewcommand{\nobs}{m}

\newcommand{\tdist}{\pi}
\newcommand{\tdens}{\pi}
\newcommand{\rdist}{\nu}
\newcommand{\rdens}{\nu}
\newcommand{\fmdist}{\pi_\theta}
\newcommand{\fmdens}{\pi_\theta}
\newcommand{\Law}{\mathcal{L}}
\newcommand{\dchisq}[2]{D_{\chi^2}(#1 \| \,#2)}

\newcommand{\papertitle}{Flow Matching for Fast Posterior Sampling in Bayesian Inverse Problems}
\newcommand{\paperkeywords}{Bayesian inverse problems, flow matching, Markov chain Monte Carlo, independent Metropolis--Hastings, amortized inference, electrical impedance tomography}
\newcommand{\datastatement}{The code and data that support the findings will be made publicly available as part of this study upon acceptance of the manuscript.}

\begin{document}

\ifiop
\articletype{Paper}

\title{\papertitle}

\author{Jan Blechschmidt$^{1,*}$\orcid{0000-0003-3606-0606}, Oliver G. Ernst$^1$\orcid{0000-0002-0176-7321}, Moritz Poguntke$^1$\orcid{0009-0007-2390-9988} and Björn Sprungk$^2$\orcid{0000-0001-8359-6863}}

\affil{$^1$Department of Mathematics, TU Chemnitz, Chemnitz, Germany}

\affil{$^2$Faculty of Mathematics and Computer Science, TU Bergakademie Freiberg, Freiberg, Germany}

\affil{$^*$Author to whom any correspondence should be addressed.}

\email{jan.blechschmidt@mathematik.tu-chemnitz.de}

\else
\title{\papertitle}
\author{Jan Blechschmidt\textsuperscript{1}\thanks{Corresponding author: \href{mailto:jan.blechschmidt@mathematik.tu-chemnitz.de}{jan.blechschmidt@mathematik.tu-chemnitz.de}} \and Oliver G. Ernst\textsuperscript{1} \and Moritz Poguntke\textsuperscript{1} \and Björn Sprungk\textsuperscript{2}}
\publishers{\normalsize
  \textsuperscript{1}Department of Mathematics, TU Chemnitz, Chemnitz, Germany\\
  \textsuperscript{2}Faculty of Mathematics and Computer Science, TU Bergakademie Freiberg, Freiberg, Germany}
\date{}
\maketitle
\fi

\begin{abstract}
Sampling from the posterior is the central task of computational Bayesian inverse problems.
The standard workhorse in Bayesian inference---Markov chain Monte Carlo (MCMC)---is sequential, yields correlated samples, and must be rerun for each new observation.
Conditional flow matching offers an amortized alternative: a transport map, trained once on joint samples of parameter and data, that yields independent approximate posterior samples for any observation at negligible online cost, without new likelihood evaluations.
We give a careful, MCMC-literate assessment of flow matching for PDE-based inverse problems with function-valued parameters.
Exploiting the flow's tractable density, we derive computable accuracy estimates for the learned posterior in total variation distance and Kullback--Leibler divergence and, moreover, propose
a hybrid sampler that is made asymptotically exact by Metropolization.
We validate the accuracy estimates and demonstrate amortization in several numerical examples, including electrical impedance tomography and a likelihood-free Lotka--Volterra model.
Flow matching reproduces posterior means and credible bands closely at a fraction of the forward solves required by MCMC, but its accuracy degrades for concentrated posteriors and for observations that are atypical under the prior.
\ifiop\else
\fi
\end{abstract}

\section{Introduction} \label{sec:introduction}

For the past three decades, Markov chain Monte Carlo (MCMC) has been the workhorse for sampling from posterior distributions in Bayesian inference and Bayesian inverse problems.
Modern variants --- dimension-robust proposals such as preconditioned Crank--Nicolson \cite{cotter2013mcmc,hairer2014spectral},
likelihood-informed \cite{CuiEtAl2016, RudolfSprungk2018} and adaptive \cite{haario2006dram,haario2001adaptive} schemes ---
are well understood theoretically and remarkably reliable in practice.
Their cost profile, however, is fixed by design: the Markov chain is sequential, its samples are correlated, every step requires at least one evaluation of the forward model, and the entire computation must be redone from scratch for each new observation.
In applications where the same inverse problem is solved repeatedly for a stream of measurements --- electrical impedance
tomography (EIT) being a prototypical example --- this per-observation cost is the dominant obstacle to online or real-time uncertainty quantification.

A different paradigm, going back to \textcite{ElMoselhyMarzouk2012} and systematically developed in \cite{parno2015transport, MarzoukEtAl2016}, is to construct a \emph{transport map} pushing a simple reference distribution forward to the posterior: once the map is
available, we can generate, in parallel, arbitrarily many independent samples from (an approximation to) the posterior.
If the transport map is of block-triangular structure and trained to generate the joint distribution of unknown parameters and observable data, then a standard disintegration argument in \cite[Sec.~7]{MarzoukEtAl2016} shows that this transport map reproduces the posterior for (almost) all possible observations.
In this sense, transport-based sampling for Bayesian inverse problems is \emph{amortized}: after an offline training phase, approximate posterior samples for any new observation are available at negligible online cost, namely one evaluation of the transport map per sample, and no further likelihood evaluations are needed.
These are substantial promises, in particular for PDE-based Bayesian inverse problems with computationally expensive likelihoods.

Several approaches have been proposed for constructing suitable transport maps.
The foundational works \cite{ElMoselhyMarzouk2012,parno2015transport,MarzoukEtAl2016} focus on monotone triangular transport maps motivated by the Knothe--Rosenblatt rearrangement.
More recently, generative machine learning has supplied more flexible and scalable approaches:
\textcite{baptista2024conditional} use monotone generative adversarial networks (GANs) \cite{Goodfellow2014}, and \textcite{wang2023efficient} employ continuous-time normalizing flows in the form of neural ordinary differential equations (ODEs) \cite{ChenEtAl2018, grathwohl2018} to construct conditional optimal transport maps.
Score-based diffusion models \cite{SongEtAl2021} have been applied to Bayesian inverse problems as well, e.g., in \cite{ChungEtAl2023, BaldassariEtAl2023}.
Closely related to these is the \emph{flow matching} approach for learning a neural ODE-based transport map \cite{LipmanEtAl2023,liu2022flow,AlbergoVandenEijnden2023,AlbergoEtAl2023}.
Its main advantage over maximum-likelihood training of continuous normalizing flows \cite{ChenEtAl2018, grathwohl2018} is that this training is replaced by a simple least-squares regression: each training step requires a single evaluation of the network instead of an ODE solve, which reduces the training cost substantially (cf.\ Section~\ref{subsec:lotka_volterra}).
\emph{Bayesian flow matching}, i.e., flow matching for sampling conditional distributions, has been derived and discussed in detail in \cite{chemseddine2025CondFM, hagemann2025conditional,SteidlWald} but has been demonstrated primarily on low-dimensional benchmarks and imaging-type applications.
Very recently, Krüger et al.\ \cite{krueger2026} compare three different machine learning approaches for transport sampling of posteriors in Bayesian inverse problems --- GANs, invertible neural networks, and Bayesian flow matching --- and find Bayesian flow matching to be the ``clear winner''.
Moreover, owing to its likelihood-free training, flow matching can also be used for simulation-based inference \cite{WildbergerEtAl2023,CranmerEtAl2020}, where the likelihood is unknown or intractable, so that likelihood-based MCMC is not applicable.

Flow matching therefore appears to be among the most promising approaches for transport-based sampling of conditional distributions.
What is missing so far in the literature, in our view, is a careful, MCMC-literate assessment of flow matching for PDE-based Bayesian inverse problems: one that compares it not with other transport-based approaches but with well-tuned standard MCMC algorithms, on problems of increasing difficulty.
Furthermore, this paper aims to provide computable estimates of lower and upper bounds of the accuracy of the learned flow matching posteriors.
These estimates rest on (Markov chain) Monte Carlo theory and on combining flow matching with the independent Metropolis–Hastings algorithm.
The combination also yields a computationally efficient, asymptotically exact sampler that outperforms standard
MCMC samplers on some challenging Bayesian inverse problems.
In this regard, we also mention previous works on the combination of transport maps and MCMC: \textcite{ParnoMarzouk2018} use monotone triangular transport maps to construct posterior-adapted proposals that accelerate MCMC;
\textcite{GabrieEtAl2022} interleave local MCMC moves with nonlocal proposals from a normalizing flow that is trained during sampling;
and \textcite{grenioux2023sampling} and \textcite{BrofosEtAl2022} study discrete-time normalizing flows as proposals within independent Metropolis–Hastings, the latter with adaptation of the flow during sampling and an accompanying convergence analysis.
In contrast, the flow matching proposal used here is trained once, without likelihood evaluations, and serves all observations.
Finally, a priori error bounds for unconditional flow matching in Wasserstein distance are derived in \cite{BentonEtAl2024}; these, however, are not computable in practice.

\paragraph{Our contributions} are summarized as follows:
\begin{enumerate}
\item
We give a self-contained account of this methodology addressed to the inverse problems and UQ communities, assembling its mathematical and computational ingredients in a unified presentation anchored in measure transport (Section~\ref{sec:fm}).

\item
Exploiting the tractable density of the flow and using classical Monte Carlo theory, we derive computable
estimates of lower and upper bounds on the accuracy of the learned posterior in total variation (TV) distance and Kullback--Leibler (KL) divergence (Section~\ref{subsec:diagnostics}); related weight-based diagnostics for amortized posteriors are used in \cite{DaxEtAl2023}.
\item
As a by-product, we obtain, as in \cite{grenioux2023sampling,BrofosEtAl2022}, an asymptotically exact \emph{hybrid sampler}, in which the bias of flow matching samples is removed via Metropolization (Section~\ref{subsubsec:imh}).
\item
We demonstrate the efficiency and accuracy of flow matching for PDE-based Bayesian inverse problems in a number of numerical experiments,\footnote{The code for all examples will be published on GitHub upon acceptance.} in which we
\begin{itemize}
    \item
    compare the estimated accuracy bounds with reference values computed by quadrature in a simple illustrative example (Section~\ref{sec:toy}),
    \item
    conduct systematic parameter studies on the accuracy of flow matching for Bayesian inversion in the stationary diffusion equation for an increasing number of observations, a decaying noise level and an increasing number of Karhunen--Loève terms (Section~\ref{sec:infinite_heat_example}), and
    \item
    illustrate the benefits of flow matching-based sampling compared to MCMC as an amortized sampler in EIT (Section~\ref{subsec:eit}), where a single offline training enables real-time posterior sampling for a stream of new measurements, and for a stochastic Lotka--Volterra model, where the likelihood is intractable and classical MCMC methods, which rely on likelihood evaluations, are not applicable (Section~\ref{subsec:lotka_volterra}).
\end{itemize}
\end{enumerate}

Once trained, the flow matching model provides a sampler for \emph{new} observations at negligible cost, in contrast to MCMC.
This advantage comes with clearly identifiable limitations, which our experimental studies map out: the accuracy of flow matching approximations can degrade for strongly concentrated posteriors (e.g., due to small observational noise) and for observations unlikely under the prior, because flow matching spreads the training data according to the prior predictive distribution whereas MCMC focuses its likelihood evaluations (after burn-in) near the center of the posterior probability mass.

This paper is deliberately computational and proof-of-concept in character.
A deeper mathematical analysis of Bayesian flow matching, in particular the development of remedies for the outlined shortcomings, is the subject of future work.

\paragraph{Problem setting and outline}
We consider the inverse problem of inferring an unknown parameter $x \in \mathcal{X}$ from a noisy observation $y \in \mathcal{Y}$ related through the observation model
\begin{equation} \label{eq:forward}
    y = F(x) + \varepsilon,
\end{equation}
where $F \colon \mathcal{X} \to \mathcal{Y}$ denotes the forward operator --- typically the composition of a PDE solution operator with an observation operator --- and $\varepsilon$ is observational noise, independent of $x$.
The parameter space $\mathcal{X}$ is in general a function space, while $\mathcal{Y} = \mathbb{R}^{\nobs}$ throughout.
The inverse problem of recovering $x$ from $y$ is typically ill-posed in the sense of Hadamard \cite{Hadamard}.
Instead of regularizing it to obtain a single estimate, e.g., variationally, we adopt the Bayesian approach \cite{Stuart2010}, which characterizes the uncertainty in $x$ given $y$ by a probability distribution: endowing $x$ with a prior distribution $\tdist_X$ and assuming a Gaussian observational noise $\varepsilon \sim \mathcal{N}(0, \Gamma)$, the posterior distribution, i.e., the conditional distribution of $x$ given the observation $y$, is given by Bayes' rule as a change of measure with respect to
the prior,
\begin{equation}
  \label{eq:rn_posterior}
  \frac{\mathrm{d}\tdist^y}{\mathrm{d}\tdist_X}(x)
  = \frac{1}{Z(y)} \exp\bigl( -\tfrac{1}{2} \bigl\| y - F(x) \bigr\|_\Gamma^2 \bigr)
\end{equation}
with the normalizing constant $Z(y) = \int_{\mathcal{X}} \exp\bigl( -\tfrac{1}{2} \bigl\| y - F(x) \bigr\|_\Gamma^2 \bigr)\, \tdist_X(\mathrm{d}x)$ and $\|v\|_\Gamma \coloneqq \|\Gamma^{-1/2} v\|_2$.

In the following, we focus mainly on the finite-dimensional case, i.e., on a discretization $\mathcal{X}_h = \mathbb{R}^d$ of $\mathcal X$ and an absolutely continuous prior $\tdist_X$ with Lebesgue density also denoted by $\tdens_X\colon \mathbb{R}^d \to [0,\infty)$.
Then, the posterior $\tdist^y$ also admits a Lebesgue density given by
\begin{equation} \label{eq:posterior}
    \tdens^y(x)
    =
    \frac{1}{Z(y)}  \exp\Bigl( -\tfrac{1}{2} \bigl\| y - F(x) \bigr\|_\Gamma^2 \Bigr)\, \tdens_X(x),
    \qquad
    x \in \mathbb R^d.
\end{equation}

The computational task is then to sample from $\tdist^y$, either for a single observation $y$ or --- the setting emphasized in this paper --- for multiple observations $y$ under the same prior and forward model.
Two features of \eqref{eq:posterior} are important for the design of sampling algorithms:
the normalizing constant $Z(y)$ is unavailable, so algorithms may access the posterior only through
ratios of the unnormalized density; and each evaluation of $F$ requires one solve of the forward model, which dominates the computational cost in the PDE-based problems of Section~\ref{sec:experiments}.

Section~\ref{sec:mcmc} recalls the Metropolis--Hastings MCMC variants used as baselines.
Section~\ref{sec:fm} presents conditional sampling via flow matching including the proposed sample-based diagnostics.
Section~\ref{sec:experiments} contains the numerical study, and Section~\ref{sec:discussion} collects our findings, limitations, and open questions.
\section{Posterior sampling via Markov chain Monte Carlo} \label{sec:mcmc}

The Markov chain Monte Carlo method constructs an ergodic Markov chain $(X^{(k)})_{k \ge 1}$ on $\mathcal{X}$, or on its discretization $\mathcal{X}_h=\mathbb{R}^d$, with $\tdist^y$ as its invariant distribution, so that the states approximately follow the posterior distribution for sufficiently large $k$.
We restrict attention to the Metropolis–Hastings (MH) class, which contains many of the most widely used sampling algorithms for Bayesian inverse problems and suffices for the comparisons in this paper; we recall the general mechanism, discuss its efficiency, and describe the variants used in this paper.

\subsection{The Metropolis--Hastings algorithm} \label{subsec:mh}

Given the current state $x^{(k)}$, an MH transition consists of two steps:
a candidate $\tilde{x}$ is drawn from a proposal density
$q(\tilde{x} \mid x^{(k)})$, and is accepted, i.e., $X^{(k+1)} = \tilde{x}$,
with probability
\begin{equation} \label{eq:mh_acceptance}
  \alpha\bigl(x^{(k)}, \tilde{x}\bigr)
  =
  \min\biggl\{ 1,\;
    \frac{\tdist^y(\tilde{x})}{\tdist^y(x^{(k)})}\,
    \frac{q(x^{(k)} \mid \tilde{x})}{q(\tilde{x} \mid x^{(k)})}
  \biggr\},
\end{equation}
otherwise the chain remains at $X^{(k+1)} = x^{(k)}$; generic pseudocode is recorded as Algorithm~\ref{alg:mh} in Appendix~\ref{app:mh_alg}.
Only ratios of the posterior density enter \eqref{eq:mh_acceptance}, so $Z(y)$ is never needed, but every iteration costs at least one evaluation of $F$ and hence one forward solve.
Under mild positivity assumptions on $q$ \cite{robert1999monte}, the resulting chain is ergodic and converges to the posterior in total variation from any initialization,
\begin{equation} \label{eq:TVdist}
  \lim_{k \to \infty}
  d_{\mathrm{TV}}\bigl( \Law(X^{(k)}),\, \tdist^y \bigr) = 0,
  \qquad
  d_{\mathrm{TV}}(\mu_1, \mu_2)
  \coloneqq
  \sup_{A \in \mathcal{B}(\mathcal{X})} |\mu_1(A) - \mu_2(A)|,
\end{equation}
where $\mathcal{L}(X^{(k)})$ denotes the law of $X^{(k)}$.
Note that for absolutely continuous $\mu_1, \mu_2$ on $\mathbb R^d$ one has $d_{\mathrm{TV}}(\mu_1,\mu_2) = \tfrac{1}{2}\int |\mu_1(x) -\mu_2(x)| \, \mathrm dx$.
Moreover, the average taken along the path of the Markov chain converges to the true posterior expectation, i.e.,
for every $f\colon \mathcal{X}_h\to\mathbb{R}$ that is integrable with respect to $\tdist^y$, we have, under mild assumptions \cite{robert1999monte}, almost surely
\begin{equation*}
  \lim_{n \to \infty} \frac 1n \sum_{k=1}^n f(X^{(k)})
  =
  \evalt{X \sim \tdist^y}{f}.
\end{equation*}
What distinguishes good from poor MH samplers, i.e., a good from a poor choice of proposal $q$, is not these asymptotic guarantees but rather the rate of convergence in \eqref{eq:TVdist}
and the asymptotic variance of the path average estimator $\frac 1n \sum_{k=1}^n f(X^{(k)})$.

A common measure for the efficiency of a Markov chain Monte Carlo algorithm is therefore its effective sample size (ESS).
For a stationary chain, i.e., $X^{(1)} \sim \tdist^y$, it is given by
\[
    \mathrm{ESS}_f(n)
    =
    \frac{n}{1+2\sum_{k=1}^\infty \mathrm{Corr}\bigl(f(X^{(1)}),f(X^{(1+k)})\bigr)}, \qquad n\in\mathbb N,
\]
and corresponds, for large $n$, to the number of independent samples from $\tdist^y$ whose empirical average of $f$ has the same variance as the path average $\frac 1n \sum_{k=1}^n f(X^{(k)})$.
It satisfies $\mathrm{ESS}_f(n) \le n$ whenever the autocorrelations sum to a nonnegative value, as is typically the case for the samplers considered here.

\paragraph{Computational aspects} \label{subsec:mcmc_cost}
MH algorithms share a common cost structure that frames the comparison in this paper.
The Markov chain is inherently sequential in that each state depends on its predecessor, so sampling typically parallelizes only across independent chains or across proposal generation.
Its output is correlated, so the effective sample size falls short of the chain length, often by orders of magnitude in the case of poor mixing.
Every iteration requires at least one forward solve, the dominant cost in PDE-based Bayesian inverse problems.
Moreover, the entire computation is tied to a single
observation $y$: a new measurement requires running a new Markov chain.
The flow matching methodology assessed in the next section targets precisely these computational bottlenecks --- independent samples, generated in parallel, at negligible per-sample cost, amortized over all observations --- while giving up the asymptotic exactness that the Metropolis--Hastings correction provides. In Sections~\ref{subsubsec:imh}
and \ref{subsec:diagnostics} we show how the two approaches can be combined to recover it.

\subsection{Common MH algorithms} \label{subsec:mcmc_variants}

In this paper three MH variants --- RWMH, pCN and DRAM --- serve as baselines against which flow matching is compared, while a fourth, IMH, forms the basis of the hybrid sampler of Section~\ref{subsubsec:imh}; we recall them in the briefest terms and refer to the cited literature for details.

\emph{Independent Metropolis--Hastings} (IMH)
uses a proposal $q$ which does not depend on the current state, i.e., $q(\tilde{x} \mid x^{(k)}) = q(\tilde{x})$, i.e.,  $\tilde{x} = \zeta$ with $\zeta \sim q$ drawn independently of $x^{(k)}$.
The performance of IMH depends crucially on the overlap between proposal and posterior: if $q$ is a good
proxy for $\tdist^y$, the chain behaves nearly like direct sampling, whereas a poor proxy produces long runs of rejections and highly correlated output.
In Section~\ref{subsubsec:imh} we use the flow matching approximation as the proposal within IMH, which yields asymptotically exact samples as well as an empirical estimate of the accuracy of the flow matching proxy in TV distance.

\emph{Random walk Metropolis} (RWMH)
\cite{metropolis1953equation,hastings1970monte} proposes a Gaussian
perturbation of the current state, $\tilde{x} = x^{(k)} + \zeta$ with
$\zeta \sim \mathcal{N}(0, \Sigma)$; the symmetric proposal density $q(\tilde{x} \mid x^{(k)}) \propto \exp(-\tfrac12 \|\tilde{x} - x^{(k)}\|^2_{\Sigma})$ cancels in \eqref{eq:mh_acceptance}. This MH variant is simple and robust but requires careful
tuning of $\Sigma$ and mixes increasingly slowly as the dimension $d$ grows.

\emph{Preconditioned Crank--Nicolson} (pCN)
\cite{cotter2013mcmc,hairer2014spectral} proposes, for a centered Gaussian prior $\tdist_X$,
\begin{equation*}
  \tilde{x} = \sqrt{1 - \beta^2}\; x^{(k)} + \beta\, \zeta,
  \qquad \zeta \sim \tdist_X, \quad \beta \in (0, 1).
\end{equation*}
The proposal is reversible with respect to the prior, so the acceptance probability reduces to the likelihood ratio
$\min\{1, \exp( \tfrac{1}{2} \bigl\| y - F(x^{(k)}) \bigr\|_\Gamma^2 - \tfrac{1}{2} \bigl\| y - F(\tilde{x}) \bigr\|_\Gamma^2)\}$.
Crucially, pCN is well defined on the function space $\mathcal{X}$ itself, i.e., before any discretization, and its convergence behavior is robust under mesh refinement
\cite{cotter2013mcmc,hairer2014spectral}, which makes it a natural reference for the function-space examples of Section~\ref{sec:experiments}.

\emph{Delayed rejection adaptive Metropolis} (DRAM) \cite{haario2006dram} combines adaptive Metropolis \cite{haario2001adaptive}, which periodically updates the proposal covariance from the chain history, with delayed rejection \cite{mira2001metropolis}, which upon rejection retries with a
rescaled, typically narrower, proposal before the current state is retained.
DRAM often attains lower autocorrelation than RWMH on correlated posteriors, at the price of
per-iteration overhead and of diminishing-adaptation conditions required for ergodicity.

\section{Posterior sampling via flow matching} \label{sec:fm}

This section describes the transport-based sampler assessed in this paper.
We first recall the structural result reducing posterior sampling to a measure transport problem on the product of parameter and observation space (Section~\ref{subsec:triangular}).
We then introduce the representation of the transport map as the flow of a neural ODE and explain why its classical, likelihood-based training is computationally prohibitive in the inverse problems context before presenting \emph{flow matching}, the regression-based training that makes the approach practical (Section~\ref{subsec:bfm}).
Finally, we discuss how flow matching combined with MCMC can yield asymptotically exact simulations of the posterior (Section~\ref{subsubsec:imh}) and how classical Monte Carlo theory yields computable
estimates of lower and upper bounds on the accuracy of the flow matching approximation (Section~\ref{subsec:diagnostics}).

\subsection{Posterior sampling as a transport problem} \label{subsec:triangular}

The measure transport approach to computational Bayesian inference  \cite{ElMoselhyMarzouk2012,MarzoukEtAl2016} seeks a map $T_X$
pushing forward a tractable reference distribution $\eta$ on the parameter space $\mathcal{X}$---in the
finite-dimensional setting below, a standard Gaussian---to the posterior, $T_X{}_\ast \eta = \tdist^y$, so that posterior sampling reduces to evaluating $T_X$ at reference samples.
For a \emph{fixed} observation $y$ such a map can be constructed variationally from the unnormalized posterior density, e.g.\ within a class of monotone triangular functions \cite{MarzoukEtAl2016}.
The approach pursued here differs in two respects.
First, the map is constructed from samples of the \emph{joint} distribution $\tdist$ of parameter and observation induced by the prior and the observation model~\eqref{eq:forward}, rather than from evaluations of the unnormalized posterior density.
Second, it is \emph{conditional}: the observation enters as a parameter, $T_X(\cdot\,; y)$.
The theoretical basis is the following disintegration result; it underlies the conditional sampling technique via triangular transports in \cite[Sec.~7, Lem.~1]{MarzoukEtAl2016}, where it is stated in finite dimensions; we state it in the Banach space formulation of \textcite{baptista2024conditional}.

Let $\mathcal{X}$, $\mathcal{Y}$ be separable Banach spaces, let
$\mathcal{S} \coloneqq \mathcal{X} \times \mathcal{Y}$ carry their product
Borel $\sigma$-algebra, and denote by $\mathcal{P}(\cdot)$ the corresponding
sets of Borel probability measures, with $\mu_X$, $\mu_Y$ denoting the marginals of
$\mu \in \mathcal{P}(\mathcal{S})$.
By disintegration we can rewrite every $\mu \in \mathcal{P}(\mathcal{S})$ as
\[
    \mu(\mathrm dx\, \mathrm dy)
    =
   \mu(\mathrm dx \mid y)\,
   \mu_Y(\mathrm dy)
\]
where the (regular version of the) conditional distribution $\mu(\cdot \mid \cdot) \colon \mathcal B(\mathcal X) \times \mathcal Y \to [0,1]$ is uniquely determined $\mu_Y$-almost everywhere.

\begin{thm} \label{thm:triangular}
  Let $\rdist = \rdist_X \otimes \rdist_Y \in \mathcal{P}(\mathcal{S})$ be a product
  reference measure and $\tdist \in \mathcal{P}(\mathcal{S})$ a target measure.
  Let $T \colon \mathcal{S} \to \mathcal{S}$ be a measurable block-triangular map of
  the form
  \begin{equation}
    \label{eq:blocktriangular}
    T(x, y) =
    \begin{pmatrix}
      T_X(x; y) \\ y
    \end{pmatrix}.
  \end{equation}
  If $T_\ast \rdist = \tdist$, then $\rdist_Y = \tdist_Y$ and
  \begin{equation*}
    T_X(\cdot\,; y)_\ast\, \rdist_X = \tdist(\cdot \mid y)
    \qquad \text{for } \tdist_Y\text{-a.e. } y \in \mathcal{Y}.
  \end{equation*}
\end{thm}

\begin{proof}
  Since $T$ leaves the second component unchanged, $T_\ast \rdist = \tdist$
  implies that the $\mathcal{Y}$-marginals coincide, hence $\rdist_Y = \tdist_Y$.
  For $B \in \mathcal{B}(\mathcal{S})$ and $y \in \mathcal{Y}$, define the slice
  $B_y \coloneqq \{x \in \mathcal{X} : (x, y) \in B\}$.
  The triangular structure of $T$ gives
  $[T^{-1}(B)]_y = T_X(\cdot\,; y)^{-1}(B_y)$, and hence, by
  $T_\ast \rdist = \tdist$, the product structure of $\rdist$, and Fubini's theorem,
  \begin{equation*}
    \tdist(B) = \rdist\bigl(T^{-1}(B)\bigr)
    =
    \int_{\mathcal{Y}} \rdist_X\bigl( T_X(\cdot\,; y)^{-1}(B_y) \bigr)
      \, \rdist_Y(\mathrm{d}y)
    =
    \int_{\mathcal{Y}} \bigl( T_X(\cdot\,; y)_\ast\, \rdist_X \bigr)(B_y)
      \, \tdist_Y(\mathrm{d}y).
  \end{equation*}
  On the other hand, the disintegration of $\tdist$ with respect to its
  $\mathcal{Y}$-marginal yields
  \[
    \tdist(B) = \int_{\mathcal{Y}} \tdist(B_y \mid y)\, \tdist_Y(\mathrm{d}y),
  \]
  and the ($\tdist_Y$-a.e.) uniqueness of the disintegration  implies
  $T_X(\cdot\,; y)_\ast\, \rdist_X = \tdist(\cdot \mid y)$ for $\tdist_Y$-a.e.\ $y$.
\end{proof}

Theorem~\ref{thm:triangular} converts the approximation of the posterior for all possible observations into a single transport problem on $\mathcal{S}$, subject to the constraint that the observation component be carried along
unchanged.

For the construction of the transport map $T_X$, we focus again on the finite-dimensional case $\mathcal{X}_h \cong \mathbb R^d$ and seek $T_X$ in a parametrized family $\{T_\theta\colon \theta \in \Theta\}$, $\Theta \subseteq \mathbb R^p$  of mappings $T_\theta\colon \mathbb R^d \times \mathcal Y \to \mathbb R^d$ (e.g., neural networks where $\theta$ consists of the weights and biases), and fix the reference measure as
\begin{equation}
  \label{eq:reference}
  \rdist = \eta \otimes \tdist_Y,
  \qquad
  \eta
  \coloneqq
  \mathcal N(0, \mathrm I_d),
\end{equation}
with $\mathrm I_d\in\mathbb R^{d\times d}$ denoting  the identity matrix.
Note that this already builds in the marginal matching $\rdist_Y = \tdist_Y$ from the start.
In the following let us denote the Lebesgue densities of $\rdist$ and $\tdist$ again by $\rdens(x, y) = \rdens_X(x)\, \rdens_Y(y)$ and
$\tdens(x, y) = \tdens(x \mid y)\, \tdens_Y(y)$, respectively, so that $\rdens_X$ is the standard Gaussian density and $\tdens^y(\cdot) = \tdens(\cdot \mid y)$ the posterior density.
Moreover, we introduce $\fmdist^y \coloneqq T_\theta(\cdot\,; y)_\ast\, \eta$ as the approximation of $\tdens^y$ resulting from the transport map $T_\theta$.
It remains to determine $\theta$.
A common approach in the machine learning literature is the minimization of the joint KL divergence between $T_\ast \rdist$ and $\tdist$, which is equivalent to maximum likelihood estimation, as we explain below.
Due to the product structure of \eqref{eq:reference}
and the triangular structure of $T(x, y) = (T_\theta(x; y), y)$ the pushforward density under the triangular map factorizes as
$[T_\ast \rdens](x, y) = [T_\theta(\cdot\,; y)_\ast\, \rdens_X](x)\, \rdens_Y(y)$, which in turn yields
\begin{equation*}
\begin{split}
    D_{\mathrm{KL}}\bigl(\tdist \,\|\, T_\ast \rdist\bigr)
    & \coloneqq
    \evalt{(x, y) \sim \tdist}{\log \frac{\tdens(x,y)}
                      {[T_\ast\rdens](x,y)}}
    =
    \evalt{(x, y) \sim \tdist}{\log \frac{\tdens^y(x)\, \tdens_Y(y)}
                      {[T_\theta(\cdot\,; y)_\ast\, \rdens_X](x)\, \rdens_Y(y)}} \\[.5em]
    &=
    \evalt{(x, y) \sim \tdist}{\log \frac{\tdens^y(x)}{[T_\theta(\cdot\,; y)_\ast\, \rdens_X](x)}}
     =
    \evalt{(x, y) \sim \tdist}{\log \tdens^y(x)}
    -\evalt{(x, y) \sim \tdist}{\log \fmdens^y(x)},
\end{split}
\end{equation*}
i.e., minimizing the joint divergence is equivalent to minimizing the expected negative log-likelihood
\begin{equation} \label{eq:nll_loss}
    L(\theta)
    =
    - \evalt{(x, y) \sim \tdist}{
      \log \fmdens^y(x)
      }
\end{equation}
over the joint samples.
In practice, this loss is estimated empirically  by Monte Carlo sampling, requiring only the ability to draw samples from $\tdist$.
Moreover, the well-known `chain rule' for the KL divergence yields
\begin{align*}  \label{eq:kl_decomposition}
    D_{\mathrm{KL}}\bigl(\tdist \,\|\, T_\ast \rdist\bigr)
    &=
    \underbrace{D_{\mathrm{KL}}\bigl(\tdist_Y \,\|\, \rdist_Y \bigr)}_{=0}
    +\;
    \evalt{y \sim \tdist_Y}{D_{\mathrm{KL}}\bigl(\tdist^y \,\|\, T_\theta(\cdot\,; y)_\ast\,\eta\bigr)}.
\end{align*}
Thus, minimizing the joint divergence also controls the expected posterior
error over observations from the prior predictive distribution.

Three features of this approach are central to the remainder of the paper.
First, the joint distribution $\tdist$ --- unlike any individual posterior --- is easy to sample: drawing $x \sim \tdist_X$ from the prior and simulating $y$ through the observation model \eqref{eq:forward} produces joint samples at the cost of one forward solve each, without a single likelihood or posterior density evaluation.
The training data for the learned transport are therefore cheap and, notably, available even when the
likelihood is intractable.
Second, maximum-likelihood estimation, i.e., minimization of the KL divergence, requires the computation of the pushforward density $\fmdens^y$ and of its gradients with respect to $\theta$.
The feasibility and computational efficiency of this may depend on the architecture of $T_\theta$.
Third, the training objective \eqref{eq:nll_loss} yields only $\tdist_Y$-\emph{almost everywhere} guarantees for the approximation of $\tdist^y$ by $\fmdist^y$.
Pointwise guarantees are available under additional assumptions \cite{Hagemann2023}.
In practice, however, training in expectation with respect to $\tdist$ means that for observations $y$ in the tail of the prior predictive distribution $\tdist_Y$, the approximation $\fmdist^y$ can be very poor, since with high probability no training data reach such regions.

In the following, we present a convenient approach for transport maps $T_\theta$ with computable pushforward density $\fmdens^y$ and explain how they can be trained efficiently via the flow matching objective instead of the maximum-likelihood objective~\eqref{eq:nll_loss}.

\subsection{Neural transport maps and their training via Bayesian flow matching}
\label{subsec:bfm}

Generative machine learning provides several flexible function classes for constructing a block-triangular map with $T_\ast \rdist \approx \tdist$.
Here, we represent the transport map as the solution operator, or \emph{flow map}, of an ODE
\begin{equation}
  \label{eq:node}
  \frac{\mathrm{d}}{\mathrm{d}t}\, x(t) = v_\theta\bigl(t, x(t); y\bigr),
  \qquad t \in [0, 1], \qquad x(0) = x_0 \sim \eta,
\end{equation}
whose right-hand side $v_\theta \colon [0,1] \times \mathbb{R}^d \times
\mathbb{R}^{\nobs} \to \mathbb{R}^d$ is a neural network with weights and biases denoted by $\theta \in \mathbb R^p$ taking as inputs time, state, and
the observation.
Equation~\eqref{eq:node} is also known as a \emph{neural ODE} \cite{ChenEtAl2018}.
If $v_\theta$ is continuous and, in the state $x$, continuously differentiable and  globally Lipschitz continuous, the Picard--Lindel\"of theorem guarantees the existence of a unique solution $x(\cdot\,;x_0,y)\colon [0,1]\to\mathbb R^d$ of \eqref{eq:node} and that the solution map
\begin{equation*}
  T_\theta(x_0; y) \coloneqq x(1; x_0, y)
\end{equation*}
is a $C^1$-diffeomorphism of $\mathbb{R}^d$ for every fixed $y$.
Setting
$T(x, y) = (T_\theta(x; y), y)$ yields exactly the block-triangular
structure \eqref{eq:blocktriangular}.
Moreover, due to the ODE representation of the transport map, the density of the pushforward
$\fmdist^y = T_\theta(\cdot\,; y)_\ast\, \eta$ is straightforward to compute.
To this end, let $p_t\colon \mathbb R^d \times \mathcal Y\to[0,\infty)$ denote the Lebesgue density of the pushforward measure $T_{\theta}(t,x_0; y)_\ast \eta$ where $T_{\theta}(t,x_0; y) \coloneqq x(t; x_0, y)$. Then, by the instantaneous
change-of-variables formula \cite{ChenEtAl2018}, we have
\begin{equation}
  \label{eq:icov}
  \frac{\partial \log p_t(x(t); y)}{\partial t}
  = - \nabla \cdot v_\theta\bigl(t, x(t); y\bigr),
\end{equation}
so that integrating \eqref{eq:node} backward from a point $x$ and
accumulating the divergence along the trajectory yields
\begin{equation}
  \label{eq:backward_likelihood}
  \log \fmdens^y(x)
  = \log \rdens_X\bigl(x(0)\bigr)
    + \int_0^1 \nabla \cdot v_\theta\bigl(t, x(t); y\bigr)\, \mathrm{d}t .
\end{equation}
The divergence $\nabla\cdot v_\theta$ is computed by automatic differentiation at the cost of $d$ matrix-vector products with the (exact) Jacobian per evaluation.
This tractable density distinguishes deterministic flows from stochastic samplers such as diffusion models simulated by their reverse-time stochastic differential equation.
It enables Metropolization as well as quantitative diagnostics (Sections~\ref{subsubsec:imh} and \ref{subsec:diagnostics}).
Moreover, it allows one to determine the parameters $\theta$ of the velocity field by maximum likelihood, i.e., by minimizing the loss~\eqref{eq:nll_loss} with respect to $\theta$, which is equivalent to minimizing the KL divergence as shown in Section~\ref{subsec:triangular}.
In fact, this is the approach for training in the  conditional continuous normalizing flows (CCNF) literature \cite{wang2023efficient}.
The computational difficulty here, however, is that every evaluation of \eqref{eq:nll_loss} and of its gradient, as needed at every step of stochastic gradient descent, requires solving the ODE \eqref{eq:node} backward via \eqref{eq:backward_likelihood} for every sample in the current batch.
The gradient must then be obtained by differentiating through the unrolled numerical solver or by the adjoint method \cite{ChenEtAl2018}, which has stability issues of its own.
In addition, the divergence $\nabla \cdot v_\theta$ must be evaluated along each trajectory; in high dimensions it is approximated by stochastic trace estimators \cite{grathwohl2018}, and its gradient with respect to $\theta$ involves mixed second derivatives of the network.
For inverse problems, where $\mathcal{X}$ is discretized into a finely resolved space $\mathcal{X}_h \cong \mathbb{R}^d$ with large $d$, these nested ODE solves and second-order terms dominate the cost and render maximum-likelihood training prohibitively expensive.

\paragraph{Bayesian flow matching} Flow matching \cite{LipmanEtAl2023,liu2022flow,AlbergoVandenEijnden2023,AlbergoEtAl2023} retains the neural ODE representation \eqref{eq:node} but replaces the maximum-likelihood objective
\eqref{eq:nll_loss} by a least-squares regression that requires no ODEs to be solved and no traces to be evaluated during training.
Its conditional extension \cite{chemseddine2025CondFM,hagemann2025conditional,SteidlWald} is referred to as \emph{guided} or \emph{Bayesian} flow matching.
To derive the flow matching objective, let $X_0 \sim \eta$ and $(X_1, Y) \sim \tdist$ be independent\footnote{Other couplings, e.g., optimal couplings with respect to conditional Wasserstein distances \cite{chemseddine2025CondFM}, are possible but not always computationally feasible.}
and connect the endpoints by linear interpolation in the parameter
component,
\begin{equation} \label{eq:interpolant}
  X_t \coloneqq (1 - t)\, X_0 + t\, X_1, \qquad t \in [0, 1],
\end{equation}
while the observation component is held fixed at $Y$.
For each fixed observation $y\in\mathcal Y$, \eqref{eq:interpolant} defines a stochastic process on the
parameter space \cite{AlbergoEtAl2023,SteidlWald} whose time marginals $p_t(\cdot \mid y) \coloneqq \Law(X_t \mid Y = y)$ connect the reference to the posterior: $p_0(\cdot \mid y) = \eta$ since $X_0$ is
independent of $Y$, and $p_1(\cdot \mid y) = \tdist^y$ since, given $Y = y$, the terminal point $X_1$ is distributed according to the posterior $\tdist^y$.
The stochastic process itself, however, is not causal --- following a trajectory requires knowing its endpoint $X_1$ in advance.
The idea of flow matching is to replace it by a neural flow, resulting from a neural ODE \eqref{eq:node} with initial state $X_0 \sim \eta$ and the \emph{same time marginals} as the original process.
Since the sample paths of \eqref{eq:interpolant} are affine in $t$, the pathwise derivative $\dot{X}_t = X_1 - X_0$ exists trivially but again depends on the terminal state $X_1$.
To obtain a causal process with the same time marginals, one conditions $\dot{X}_t$ on the current state $X_t$, which yields the velocity field
\begin{equation} \label{eq:cond_expectation}
    v^\ast(t, x; y)
    =
    \ev{\, \dot{X}_t \;\big|\; X_t = x,\; Y = y \, }
    =
    \ev{\, X_1 - X_0 \;\big|\; X_t = x,\; Y = y \,}.
\end{equation}
It can be shown under mild regularity assumptions, see \cite[Theorem 3.3]{liu2022flow} or \cite[Theorem 6.3]{SteidlWald}, that this velocity field plugged into \eqref{eq:node} yields the same time marginals $p_t(\cdot \mid y) = \Law(X_t \mid Y = y)$, since it satisfies the corresponding \emph{continuity equation}
\begin{equation} \label{eq:continuity}
  \partial_t p_t + \nabla \cdot (p_t\, v^\ast) = 0,
\end{equation}
which is the Eulerian form of \eqref{eq:icov}.
By the defining property of conditional expectation as the best approximation in the $L^2$ sense, the velocity field \eqref{eq:cond_expectation} is the unique minimizer, over all measurable vector fields $v$, of $\ev{\|\dot{X}_t - v(t, X_t; Y)\|_2^2}$.
The flow matching objective is therefore not an ad hoc regression but exactly this best
approximation problem, restricted to the parametric class $v_\theta$:
Bayesian flow matching trains the network by
\begin{equation} \label{eq:fm_loss}
  \min_\theta\;
  \ev{ \bigl\| v_\theta(t, X_t; Y) - (X_1 - X_0) \bigr\|_2^2},
\end{equation}
where the expectation is taken over $t \sim \mathcal{U}[0, 1]$, $X_0 \sim \eta$, and
$(X_1, Y) \sim \tdist$.
In practice, the expectation in~\eqref{eq:fm_loss} is estimated by plain Monte Carlo: sampling $t \sim \mathcal{U}[0, 1]$ and $X_0 \sim \eta$ is trivial, and each sample $(X_1, Y) \sim \tdist$ requires one forward solve of~$F$.
The resulting objective is minimized by stochastic gradient descent; in contrast to~\eqref{eq:nll_loss}, its gradients require only standard backpropagation through the network, with no ODE solves and no divergence evaluations.
If the trained velocity field $v_\theta$ coincides with the ideal field $v^\ast$, then, since $v^\ast$ reproduces the time marginals and $p_1(\cdot \mid y) = \tdist^y$, the flow map $T_\theta(\cdot\,; y)$ pushes $\eta$ forward to $\tdist^y$ for $\tdist_Y$-a.e.\ $y$.
Consequently, the block-triangular map $(x, y) \mapsto (T_\theta(x; y), y)$ pushes $\rdist = \eta \otimes \tdist_Y$ forward to $\tdist$ (cf.\ the proof of Theorem~\ref{thm:triangular}).
Flow matching thus trades the expensive objective \eqref{eq:nll_loss} for a cheap
regression whose minimizer provably provides the desired transport.
However, this argument does not control the error incurred by finite network capacity, finite training data, inexact optimization, and ODE discretization at sampling time.
A priori error bounds are available for the unconditional case, see, e.g., \cite{BentonEtAl2024}, while the conditional setting is the subject of ongoing analysis.
Moreover, such bounds are stated in terms of the $L^2$ distance between the learned and the ideal velocity field, which cannot be evaluated for a trained network since $v^\ast$ is unknown.
A posteriori diagnostics such as those of Section~\ref{subsec:diagnostics} are therefore highly relevant in applications.

\paragraph{Computational aspects} Applying flow matching for posterior sampling consists of two phases.
In the \emph{offline phase (training)}, one generates joint samples $(x_1^{(i)}, y^{(i)}) \sim \tdist$ by prior sampling and forward simulation, and minimizes \eqref{eq:fm_loss} by stochastic gradient descent over
minibatches of $(t, x_0, x_1, y)$.
In the \emph{online phase (sampling)}, given an observation $y$, one draws samples $x_0 \sim \eta$ and integrates \eqref{eq:node} with the trained field $v_\theta(\cdot, \cdot\,; y)$ using a standard ODE solver.
The resulting points $T_\theta(x_0; y)$ are independent draws from the approximation $\fmdist^y$ of $\tdist^y$; they can be generated in parallel, each at the cost of one ODE solve, i.e., tens to hundreds of network evaluations.
Neither forward solves nor likelihood evaluations are required.
When necessary, the density of the approximation is available through
\eqref{eq:backward_likelihood} at the cost of a further ODE solve with $d$ Jacobian–vector products per step, which dominates the cost of Metropolization (cf.\ \Cref{subsubsec:visual}).
Pseudocode for data generation, training and sampling is given in \cref{alg:fm_data,alg:fm_train,alg:fm_sample} in \cref{app:fm_alg}.

\begin{remark}[Likelihood-free inference] \label{rem:likelihood_free}
Both training and sampling require only draws from the joint distribution $\tdist$, i.e., prior samples pushed through the observation model \eqref{eq:forward}; the likelihood is never evaluated.
Bayesian flow matching therefore applies verbatim to simulation-based (likelihood-free) inverse problems in which
$\tdens(y \mid x)$ is intractable but forward simulation is feasible \cite{CranmerEtAl2020,WildbergerEtAl2023,baptista2024conditional}, such as the stochastic Lotka--Volterra model
of Section~\ref{subsec:lotka_volterra}.
In contrast, the Metropolization of Section~\ref{subsubsec:imh} and the diagnostics of  Section~\ref{subsec:diagnostics} require the unnormalized posterior density and are hence available only for tractable likelihoods.
\end{remark}

\subsection{Flow-based independent Metropolis--Hastings} \label{subsubsec:imh}

Bayesian flow matching provides an approximation to the posterior $\tdist^y$, which is easy to sample.
However, the samples generated by flow matching are biased whenever $v_\theta$ differs from the ideal field $v^\ast$ or the ODE~\eqref{eq:node} is solved inexactly, which is typically the case in practice.
Often it is hard to judge a priori how close the flow matching approximation is to the true posterior $\tdist^y$ for given data $y$.
A combination of flow matching and Metropolization can remove this bias and provide asymptotically exact samples from the posterior $\tdist^y$.
To this end, we use the flow matching approximation as the proposal distribution within IMH, i.e., $q$ is the pushforward $\fmdist^y = T_\theta(\cdot\,; y)_\ast\, \eta$ of the standard Gaussian $\eta$.
The resulting Markov chain $(X^{(k)})_{k\in\mathbb N}$ proceeds as follows: given the current state $X^{(k)} = x$, it proposes $\tilde x = T_\theta(x_0; y)$ with a fresh draw $x_0 \sim \eta$, and accepts $\tilde x$ with log acceptance probability
\begin{equation*}
    \begin{aligned}
    \log \alpha(x,\tilde x)
    &=
    \min \left\{ 0, \, \log \tdens^y(\tilde x) - \log \tdens^y(x) + \log \fmdens^y(x) - \log \fmdens^y(\tilde x)\right\}\\
    & =
    \min \left\{ 0, \, \log w(\tilde x) - \log w(x)\right\},
    \end{aligned}
\end{equation*}
where we set $w(x) \coloneqq \tdens^y(x) / \fmdens^y(x)$.
Thus, we need to compute the log flow matching density $\log \fmdens^{y}(\tilde x)$ via numerical integration of the corresponding ODE \eqref{eq:backward_likelihood}, see above.
This approach combines the generative capability of normalizing flows, which yields a flexible, high-dimensional proposal that closely approximates the posterior geometry, with the asymptotic exactness of MCMC methods.
The proposal $\tilde{x}$ is generated independently of the current state $x^{(k)}$, allowing for parallel sampling of candidates prior to the acceptance step.
Normalizing flows have been used to enhance MCMC before, e.g., in \cite{grenioux2023sampling,GabrieEtAl2022,BrofosEtAl2022}.

\subsection{Diagnostic metrics} \label{subsec:diagnostics}
We now provide two sampling-based diagnostics for the accuracy of the flow matching approximation to $\tdist^y$.
Both diagnostics exploit that the density $\fmdens^y(x)$ is computationally tractable.

The first diagnostic is based on the flow-based IMH mentioned above.
In fact, the mean acceptance rate of the (stationary) Markov chain generated by IMH is closely related to the TV distance between proposal $q$ and target distribution $\tdist^y$.

\begin{prop}[\cite{EgeaEtAl2026}] \label{prop:imh_tv}
Let the target $\tdist^y$ and the proposal $q$ have Lebesgue densities, again denoted by $\tdist^y$ and $q$, respectively, and let
\begin{equation*}
    \bar{\alpha}
    \coloneqq \evalt{X \sim \tdist^y,\, \widetilde{X} \sim q}
        {\alpha(X, \widetilde{X})}
\end{equation*}
denote the mean acceptance rate of the stationary IMH chain.
Then
\begin{equation} \label{eq:imh_sandwich}
    1 - 2\, d_{\mathrm{TV}}(q, \tdist^y)
    \;\le\; \bar{\alpha} \;\le\;
    1 - d_{\mathrm{TV}}(q, \tdist^y).
\end{equation}
\end{prop}

These bounds are based on \cite[Theorem 1]{Neklyudov2019} and formally shown in \cite{EgeaEtAl2026} with applications in multimodal sampling; for completeness, a short proof is given in \cref{app:imh_tv_proof}.
For our purposes the decisive observation is the \emph{reverse} reading of \eqref{eq:imh_sandwich}: running an IMH chain
with proposal $q = \fmdist^{y} = T_\theta(\cdot\,; y)_\ast\eta$ for $n$ iterations, resulting in a path $\{x^{(k)}\}_{k=1}^n$ based on the proposals $\{\tilde x^{(k)}\}_{k=1}^n$, we can estimate $\bar{\alpha}$ empirically by
\[
    \overline{\alpha}
    \approx
    \frac 1n \sum_{k=1}^n \alpha(x^{(k)},\tilde x^{(k)})
\]
and obtain a computable two-sided bound
\begin{equation} \label{eq:tv_from_acceptance}
  \frac{1 - \bar{\alpha}}{2}
  \;\le\; d_{\mathrm{TV}}\bigl(\fmdist^{y} ,\, \tdist^y\bigr) \;\le\;
  1 - \bar{\alpha}.
\end{equation}
Note that this also yields a lower bound on the KL divergence between $\tdist^y$ and $\fmdist^{y}$ by means of the well-known \emph{Pinsker inequality}:
\[
    \dtv{\mu_1}{\mu_2} \le \sqrt{\frac12 \dkl{\mu_1}{\mu_2}},
\qquad \text{hence} \qquad
    \dkl{\tdist^y}{ \fmdist^y }
    \geq
    \frac 12\left(1 - \overline{\alpha}\right)^2.
\]

A computable upper bound on the above KL divergence can be derived via the $\chi^2$ divergence and its relation to \emph{importance sampling}. It is well known \cite{GibbsSu2002} that
\[
    \dkl{\mu_1}{ \mu_2}
    \leq
    \log\left( 1 + \dchisq{\mu_1}{ \mu_2} \right),
\]
where the $\chi^2$ divergence is defined by
\[
    \dchisq{\mu_1}{ \mu_2}
    \coloneqq
    \evalt{x \sim \mu_2}{\left(\frac{\mathrm d\mu_1}{\mathrm d\mu_2}(x) -1\right)^2
    }
    =
    \evalt{x \sim \mu_2}{\left(\frac{\mathrm d\mu_1}{\mathrm d\mu_2}(x)\right)^2
    }-1
\]
if the Radon--Nikodym derivative $\frac{\mathrm d\mu_1}{\mathrm d\mu_2}$ exists, and $\dchisq{\mu_1}{ \mu_2} \coloneqq \infty$ otherwise.
Note that the $\chi^2$ divergence may be infinite even if the Radon–Nikodym derivative exists.
Now, considering the normalized \emph{importance weights}
\begin{equation}
    \bar{W}_i = W_i / \sum_{j=1}^n W_j,
    \qquad
    W_i \coloneqq \frac{\tdens^y({X}_i)}{\fmdens^y({X}_i)}
\end{equation}
for i.i.d.\ flow matching samples $X_i \sim  \fmdens^y$, it is also well known from importance sampling theory \cite[Section 2.3]{AgapiouEtAl2017} that by means of the law of large numbers we have almost surely
\[
    \rho
    \coloneqq
    \lim_{n\to \infty} n \sum_{i=1}^n \bar{W}_i^2
    =
    \lim_{n\to \infty} \frac{\frac 1n \sum_{i=1}^n W^2_i}{(\frac 1n\sum_{i=1}^n {W}_i)^2}
    =
    \evalt{x \sim \fmdens^y}{\left(\frac{\tdist^y(x)}{\fmdens^y(x)}\right)^2
    }
\]
given that $\dchisq{\tdist^y}{\fmdens^y} < \infty$.
Thus, $\rho \approx 1$ indicates that $\dkl{\tdist^y}{ \fmdist^y } \approx 0$, implying the flow matching approximation is highly accurate.
Conversely, $\rho \gg 1$ indicates a significant mismatch.

In practice, given a batch of flow matching samples $x_1,\ldots, x_n \in \mathcal X$ with associated importance weights $\bar{w}_1,\ldots,\bar{w}_n$ we can estimate
\begin{equation} \label{eq:kl_bounds}
    \frac 12\left(1 - \overline{\alpha}\right)^2
    \leq
    \dkl{\tdist^y}{ \fmdist^y }
    \leq \log(\rho)
    \ \approx \
    \log\left(n \sum_{i=1}^n \bar{w}_i^2 \right)
\end{equation}
for $n$ sufficiently large.

We stress that \eqref{eq:tv_from_acceptance} and \eqref{eq:kl_bounds} are Monte Carlo estimates of bounds, not guaranteed bounds.
The estimator of $\bar\alpha$ refers to the stationary chain, but it is consistent by the ergodic theorem for any initial state; in our experiments we start the IMH chain from a flow matching sample and use the entire chain, since the influence of the initial state on the average vanishes like $1/n$ whenever the chain mixes well.
More importantly, the self-normalized estimator $n\sum_{i=1}^n \bar w_i^2$ of $\rho$ is consistent but tends to underestimate $\rho$ for finite $n$, since $\rho$ is dominated by rare samples with large weights, which a finite sample is likely to miss, particularly in higher dimensions, where the weights are heavy-tailed.
Indeed, $n\sum_{i=1}^n \bar w_i^2 \le n$, so the estimated upper bound in \eqref{eq:kl_bounds} never exceeds $\log n$.
When the flow matching approximation is poor, the empirical upper bound may therefore fall below the true Kullback–Leibler divergence.
It is reliable when the weights are well behaved, as in the example of \Cref{sec:toy}, where it encloses the values obtained by quadrature (\Cref{fig:toy_acceptance}).

In summary, we derived computable estimates of
lower and upper bounds on both the TV distance and the KL divergence between the true posterior and its flow matching approximation.

\begin{remark}
The term $n \sum_{i=1}^n \bar{W}_i^2 \geq 1$ is related to the so-called \emph{effective sample size} of importance sampling.
The latter is defined by $(\sum_{i=1}^n \bar{W}_i^2)^{-1}$ and quantifies the number of independent samples from $\tdist^y$ which yield the same statistical efficiency for Monte Carlo as the importance-weighted samples drawn from the proposal distribution.
Its ratio to $n$, the reciprocal of $n\sum_{i=1}^n \bar W_i^2$, is the sample efficiency used by \textcite{DaxEtAl2023} as a diagnostic for amortized neural posterior estimates in gravitational-wave inference, where low efficiency was found to flag, among other things, out-of-distribution data.
\end{remark}

\section{Numerical experiments} \label{sec:experiments}

This section applies the conditional flow model presented in \cref{sec:fm}, trained using Bayesian flow matching (FM), to four example problems.
The first example is a two-point boundary value problem with two unknown scalar parameters for which the posterior, and hence the TV distance and KL divergence of its flow matching approximation, can be computed by quadrature.
The second example is a stationary diffusion problem in one space dimension where the goal is to reconstruct a distributed diffusion coefficient.
The third example is the Bayesian formulation of electrical impedance tomography (EIT) \cite{dunlop2016bayesian}, a well-known challenging problem in the inverse problems community.
Finally, we consider a stochastic Lotka--Volterra problem which poses the challenge that no tractable likelihood is available and which was previously considered in \cite{wang2023efficient}, following the setup in \cite{wilkinson2012stochastic}.

For all four problems, we compare the learned posterior qualitatively and quantitatively against MCMC references where applicable.
Furthermore, we study how the accuracy of the learned posterior depends on factors such as the concentration of the true posterior and how typical the observation is under the prior predictive distribution $\tdist_Y$.
\subsection{Example with a two-dimensional parameter} \label{sec:toy}

We consider a very simple Bayesian inverse problem which allows us to illustrate many of the key concepts introduced in the previous sections.
It concerns the identification of two unknown parameters $\vx = (\xi_0, \xi_1)$ in a two-point boundary value problem for a stationary diffusion equation.
The first parameter $\xi_0 \in \R$ determines the spatially constant log-conductivity, while the second parameter $\xi_1 \in \R$ determines the Dirichlet data at the right-hand boundary.
With $z \in \DD =(0,1)$ denoting the spatial coordinate, the differential equation is given by
\begin{subequations}
\begin{equation} \label{eq:toy_pde}
    - e^{\xi_0} \, u''(z) = 1,
    \quad z \in \DD,
    \qquad u(0) = 0, \quad u(1) = \xi_1,
\end{equation}
where $u: [0,1] \to \R$ denotes the solution, which is easily verified to be
$u(z) = \frac{1}{2} \, e^{-\xi_0} \, z \, (1-z) + \xi_1 \, z$.
The observational data is given by noisy evaluations of the solution $u$ at two locations, i.e.,
\begin{equation}\label{eq:toy_obs}
    \vy
    =
    \begin{bmatrix}
        u(0.25)\\u(0.75)
    \end{bmatrix}
    +
    \vvarepsilon \quad \text{with } \vvarepsilon \sim \NN(0, 0.25 \, I_{2}).
\end{equation}
\end{subequations}
We place independent priors $\xi_0 \sim \NN(0, 1)$ and $\xi_1 \sim \UU(90,110)$ on the unknown parameters.
Then, for 20 observations, equally spaced on the line segment between $\vy^{1}=(20, 68)$ and $\vy^{20}=(27,71)$ in data space, we compute the posterior for each.
Samples from the prior predictive distribution $\tdist_Y$ together with the values of the prior predictive density (evidence) at these 20 observations are shown in \cref{fig:prior-predictive-heat-1d-main}.
The corresponding evidence values are seen to vary over several orders of magnitude.

\begin{figure}
    \centering
    \includegraphics[width=0.5\linewidth]{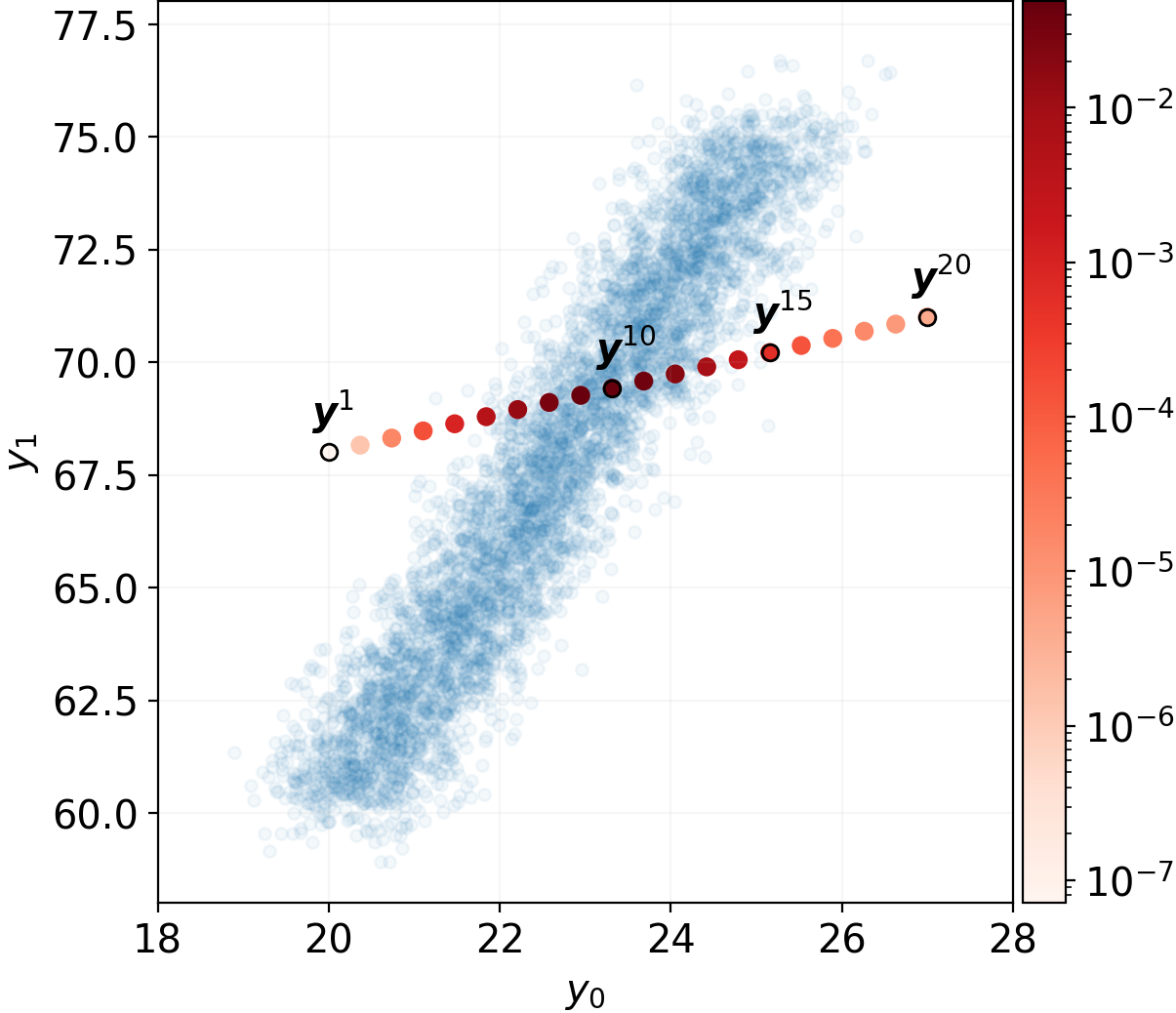}
    \caption{The blue markers show \num{5000}~samples from the prior predictive distribution of the observations \eqref{eq:toy_obs} of the solution of \eqref{eq:toy_pde}. The color intensity of the red markers at $\vy^{1}$ to $\vy^{20}$ indicates the value $\tdens_Y(\vy^{i})$ of the prior predictive density at these locations.}
    \label{fig:prior-predictive-heat-1d-main}
\end{figure}

As a reference, we run RWMH with proposal covariance $\Sigma = I_2$ for each observation, for $10^6$ steps after a burn-in of $10^5$ steps; since the problem is only two-dimensional, we also compute the posterior by numerical quadrature (see below).
The flow matching velocity field is a multilayer perceptron (MLP) with six hidden layers of width $\nhidden = 50$ and swish activations \cite{ramachandran2017searching}.
It is trained on $10^6$ samples from the joint distribution $\tdist$, which together require fewer forward solves than a single one of the 20~RWMH runs, for \num{4000} epochs with batch size \num{10000} and a cosine learning-rate schedule decaying from $10^{-3}$ to $10^{-5}$.

\Cref{fig:2dexample_fm_vs_rw} shows samples generated by RWMH as well as samples generated by the flow matching procedure for the observations $\vy^{10}$, $\vy^{15}$ and $\vy^{20}$.
Visually, for observations that are likely under the prior, such as $\vy^{10}$, both RWMH and FM generate samples that closely follow the true posterior.
Observations that are less likely ($\vy^{15}$) and highly unlikely ({$\vy^{20}$}) under the prior show minor deviations between the FM samples and the true posterior.
As is shown in \cref{fig:2dexample_imh_vs_rw}, using a flow matching posterior as the proposal within IMH can eliminate this small deviation completely.
The same holds true for all other observations $\vy^{i}$, $i=1,\ldots,20$ with IMH acceptance rate of at least $84\%$, as is shown in \cref{fig:toy_acceptance}.
Furthermore, this figure shows TV distances and Kullback--Leibler divergences for the observations $\vy^i$ together with their respective bounds derived in \eqref{eq:tv_from_acceptance} and \eqref{eq:kl_bounds}.
The values for the true TV distances and KL divergences  were approximated by numerical integration of
\begin{equation*}
  D_{\mathrm{KL}}\bigl(\tdist^{\vy^i} \,\|\, \fmdist^{\vy^i}\bigr)
  = \int_\Omega \tdens^{\vy^i}(x) \log\frac{\tdens^{\vy^i}(x)}{\fmdens^{\vy^i}(x)}\,\mathrm{d}x,
  \qquad
  d_{\mathrm{TV}}\bigl(\tdist^{\vy^i}, \fmdist^{\vy^i}\bigr)
  = \frac12 \int_\Omega \bigl\lvert \tdens^{\vy^i}(x) - \fmdens^{\vy^i}(x) \bigr\rvert\,\mathrm{d}x
\end{equation*}
with the midpoint rule on a $200\times200$ grid over $\Omega = [-10,10]\times[90,110]$. Both densities are normalized on the grid.

\begin{figure}
    \centering
    \begin{subfigure}{0.33\textwidth}
    \centering
        \includegraphics[width=0.98\linewidth]{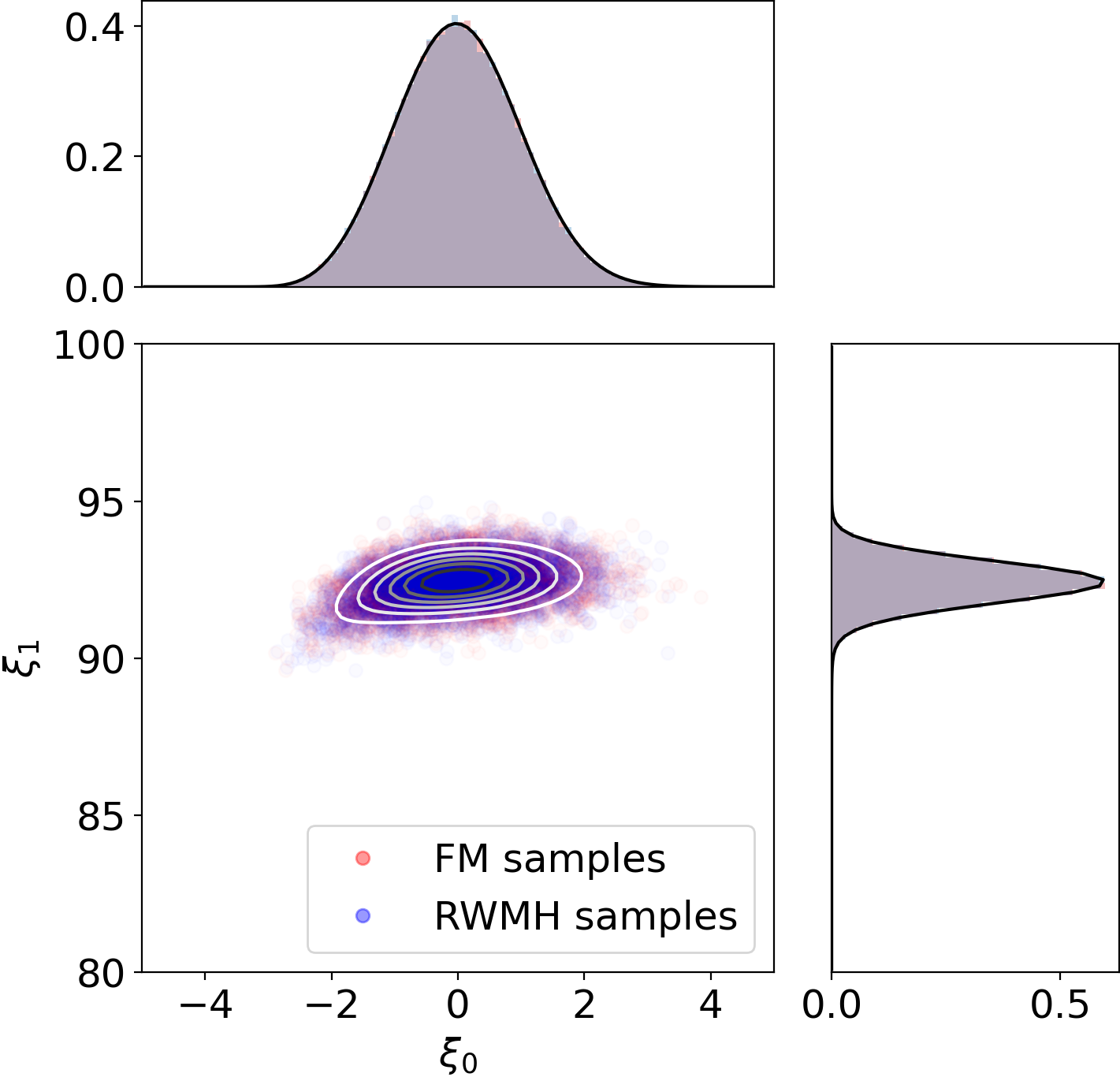}
    \end{subfigure}\hfill%
        \begin{subfigure}{0.33\textwidth}
\includegraphics[width=0.98\linewidth]{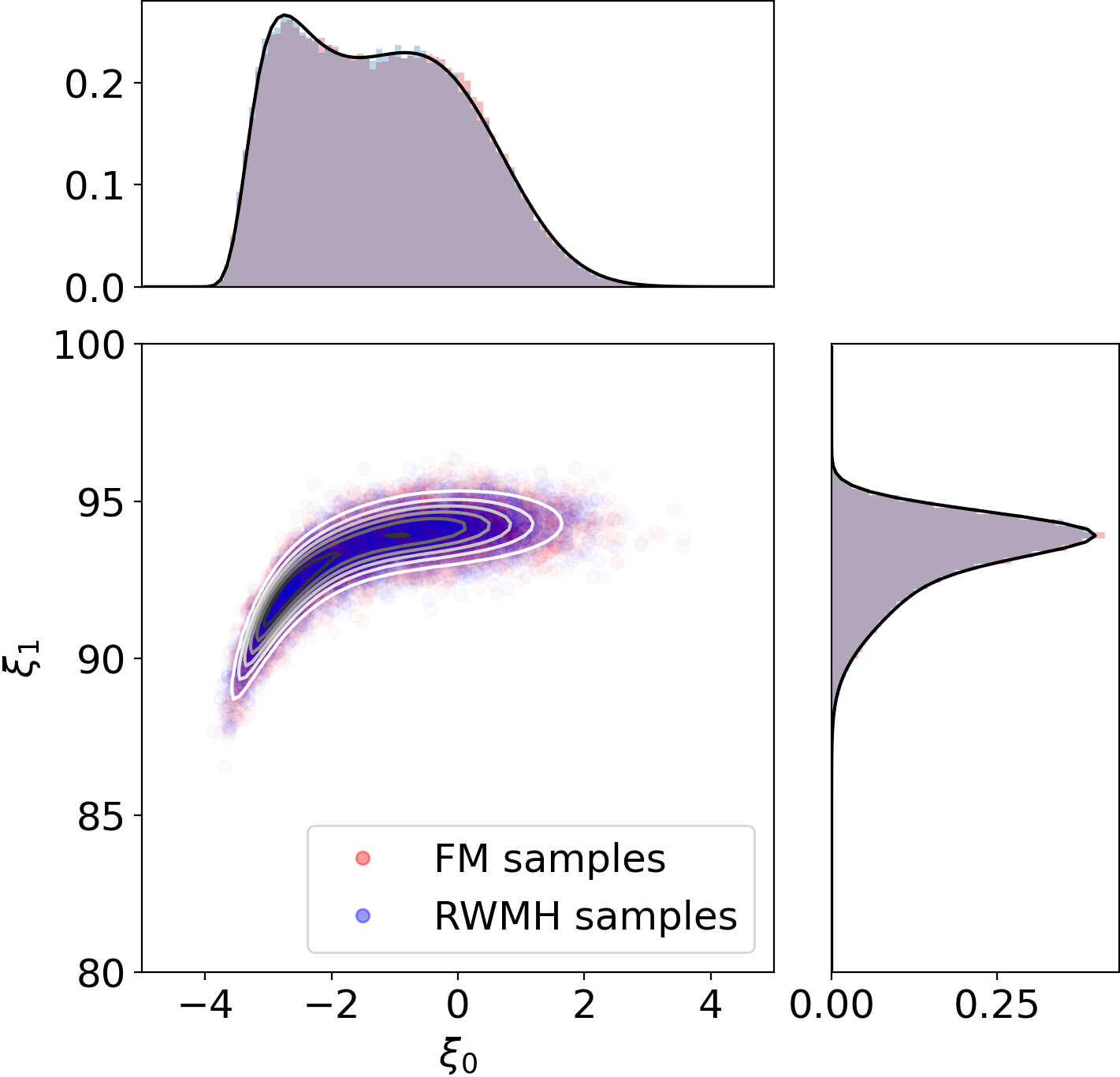}
    \end{subfigure}\hfill%
        \begin{subfigure}{0.33\textwidth}
\includegraphics[width=0.98\linewidth]{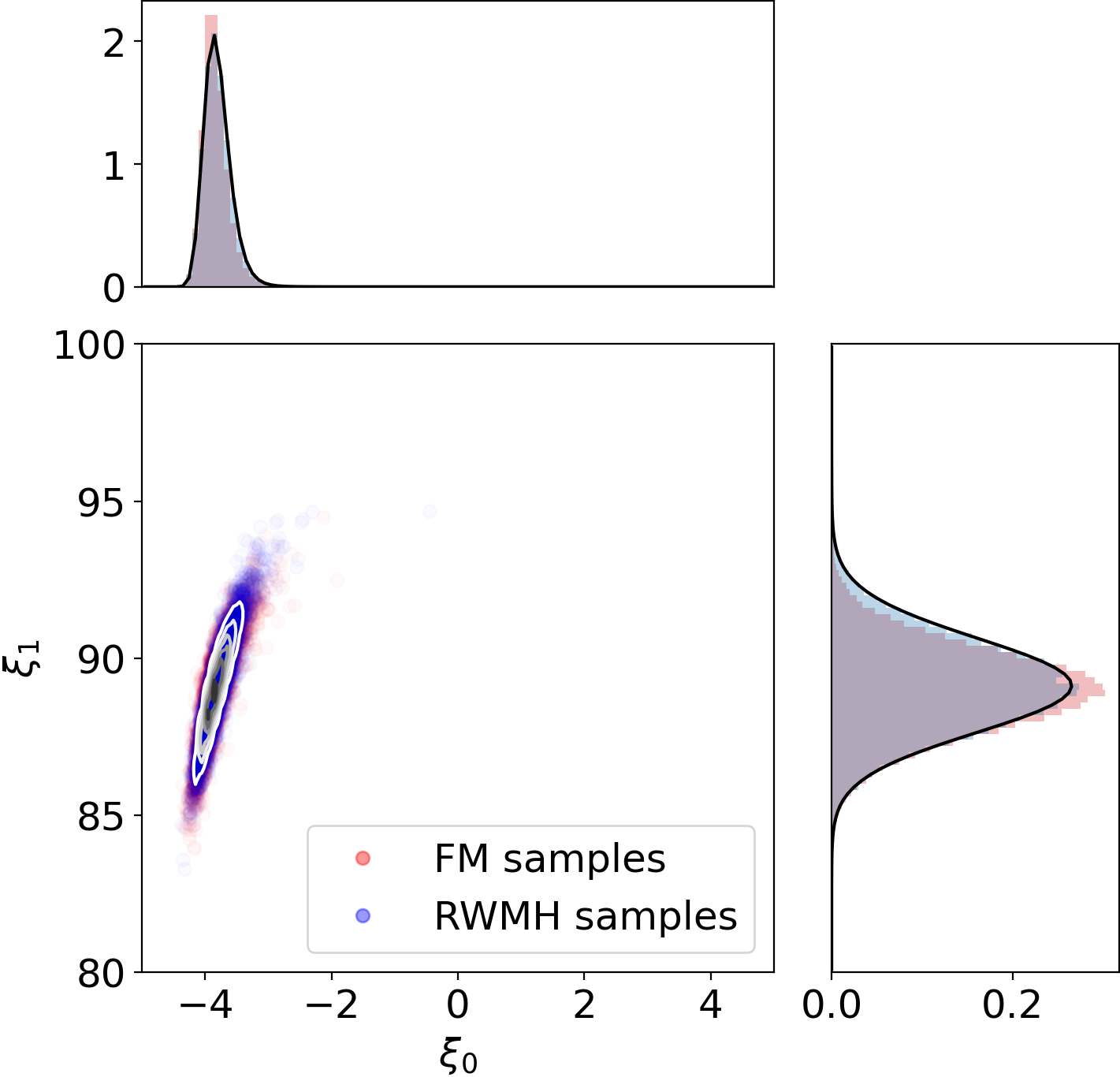}
\end{subfigure}
    \caption{Samples from RWMH and flow matching with their marginal distributions.
    Shown are the observations $\vy^{10}$, $\vy^{15}$ and  $\vy^{20}$.}
    \label{fig:2dexample_fm_vs_rw}
\end{figure}

\begin{figure}
    \centering
        \begin{subfigure}{0.33\textwidth}
    \centering
        \includegraphics[width=0.98\linewidth]{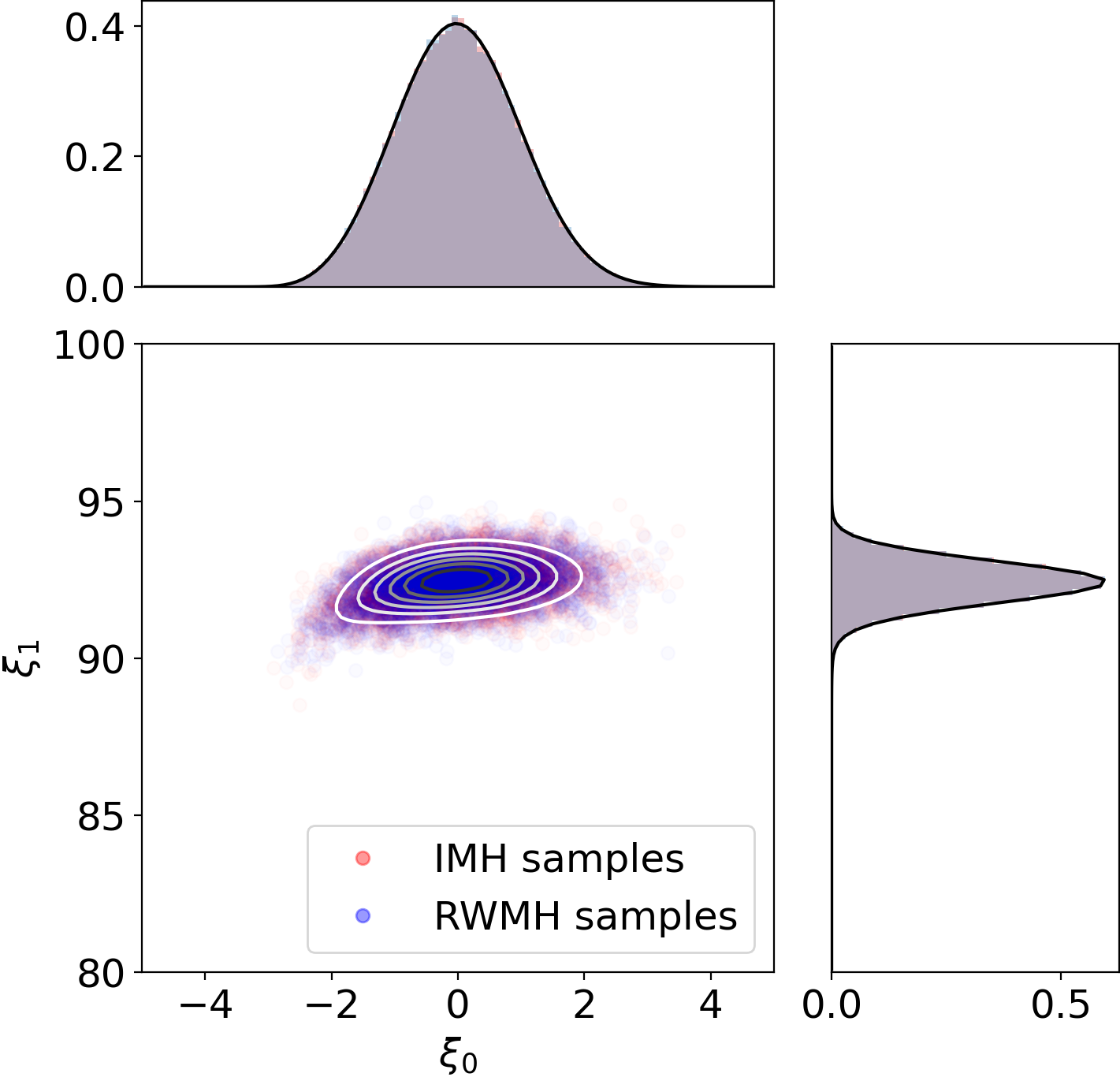}
    \end{subfigure}\hfill%
        \begin{subfigure}{0.33\textwidth}
\includegraphics[width=0.98\linewidth]{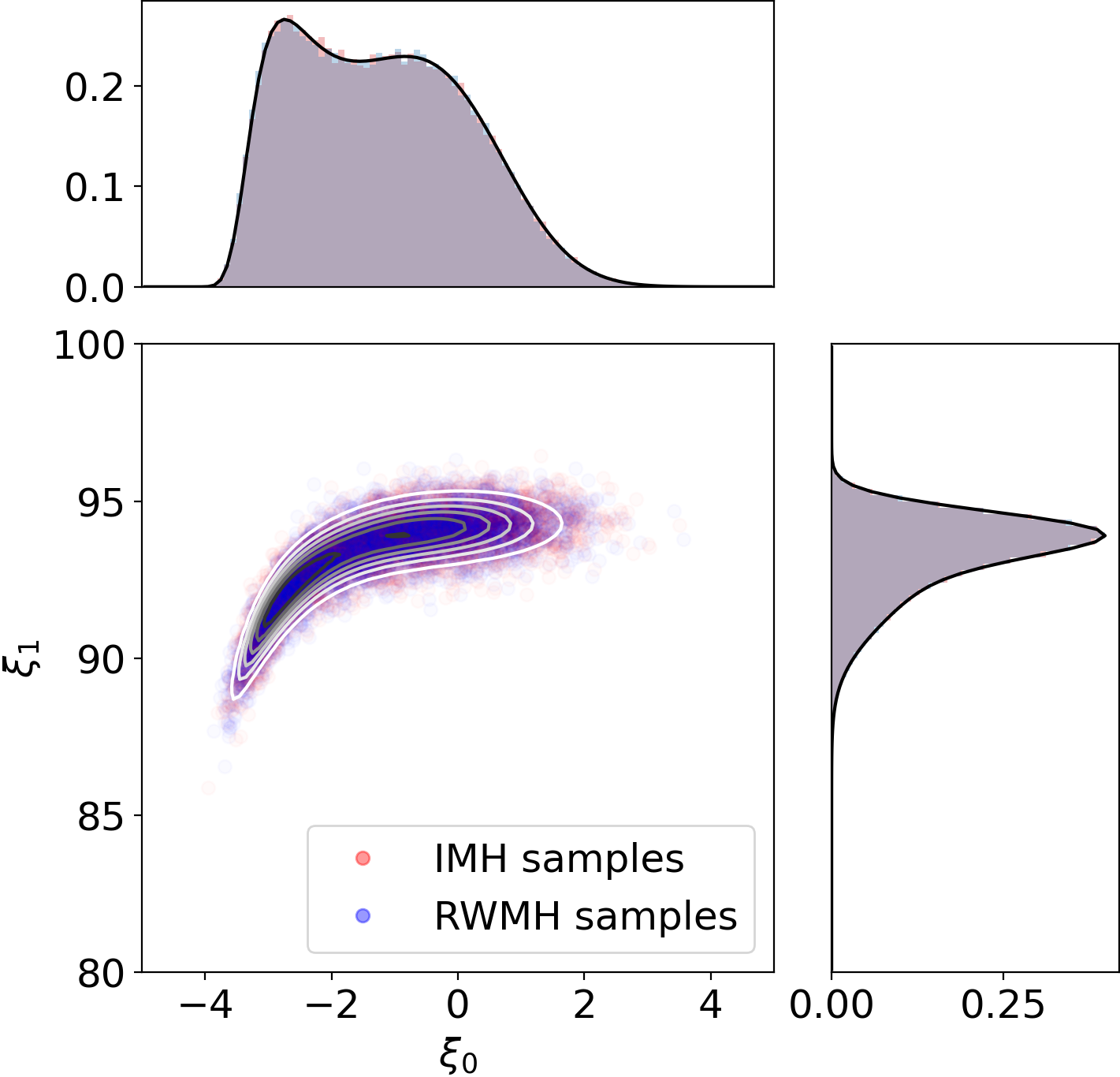}
    \end{subfigure}\hfill%
        \begin{subfigure}{0.33\textwidth}
\includegraphics[width=0.98\linewidth]{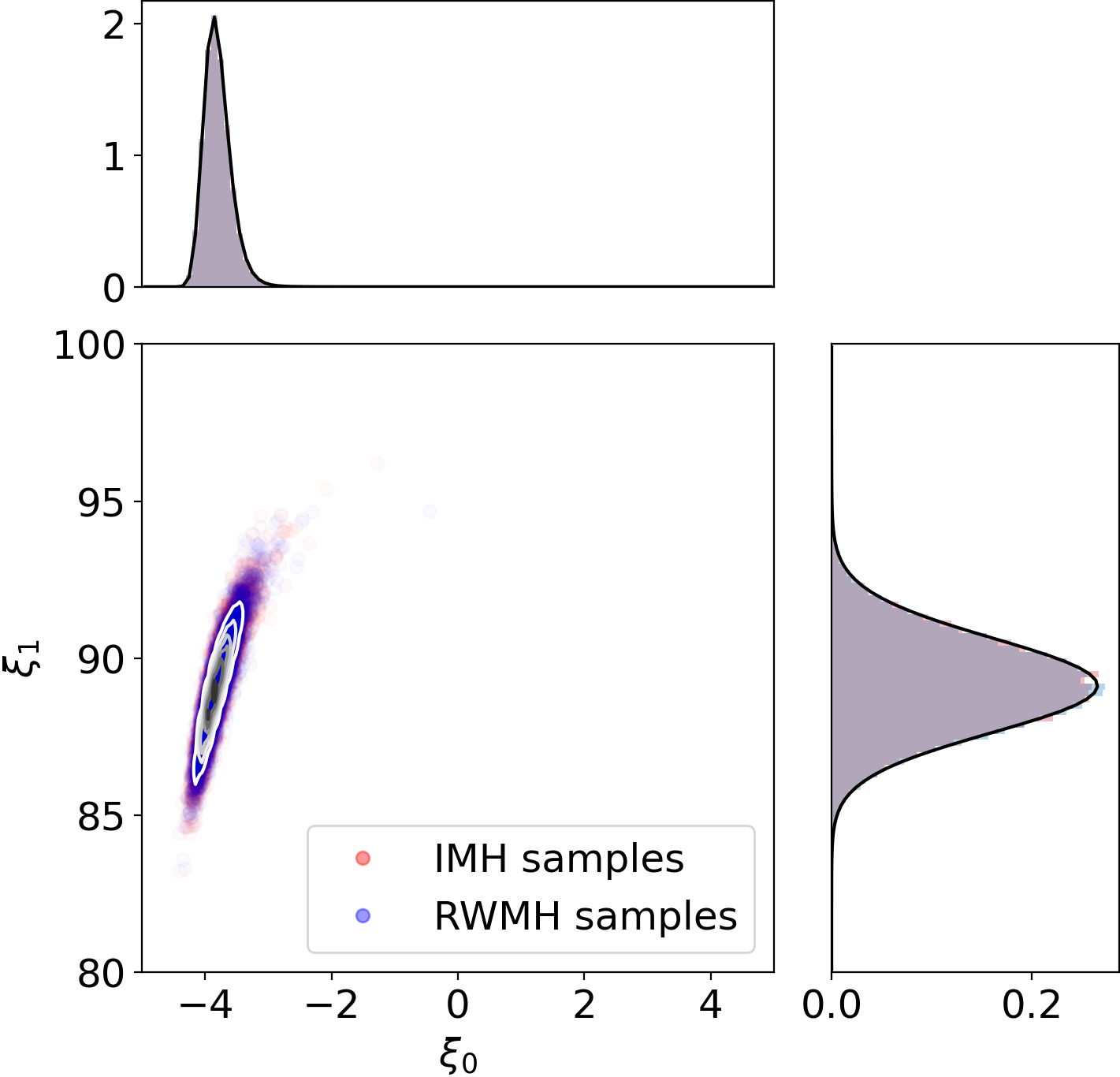}
    \end{subfigure}
    \caption{Samples from RWMH and IMH with flow matching proposal with their marginal distributions.
    Shown are the observations $\vy^{10}$, $\vy^{15}$ and  $\vy^{20}$.
    The corresponding acceptance rates are $\bar \alpha^{10} = 99.48\%$, $\bar \alpha^{15}=96.41\%$ and $\bar \alpha^{20}=84.22\%$.}
    \label{fig:2dexample_imh_vs_rw}
\end{figure}

\begin{figure}
\centering
\begin{subfigure}{0.33\linewidth}
    \includegraphics[width=0.98\linewidth]{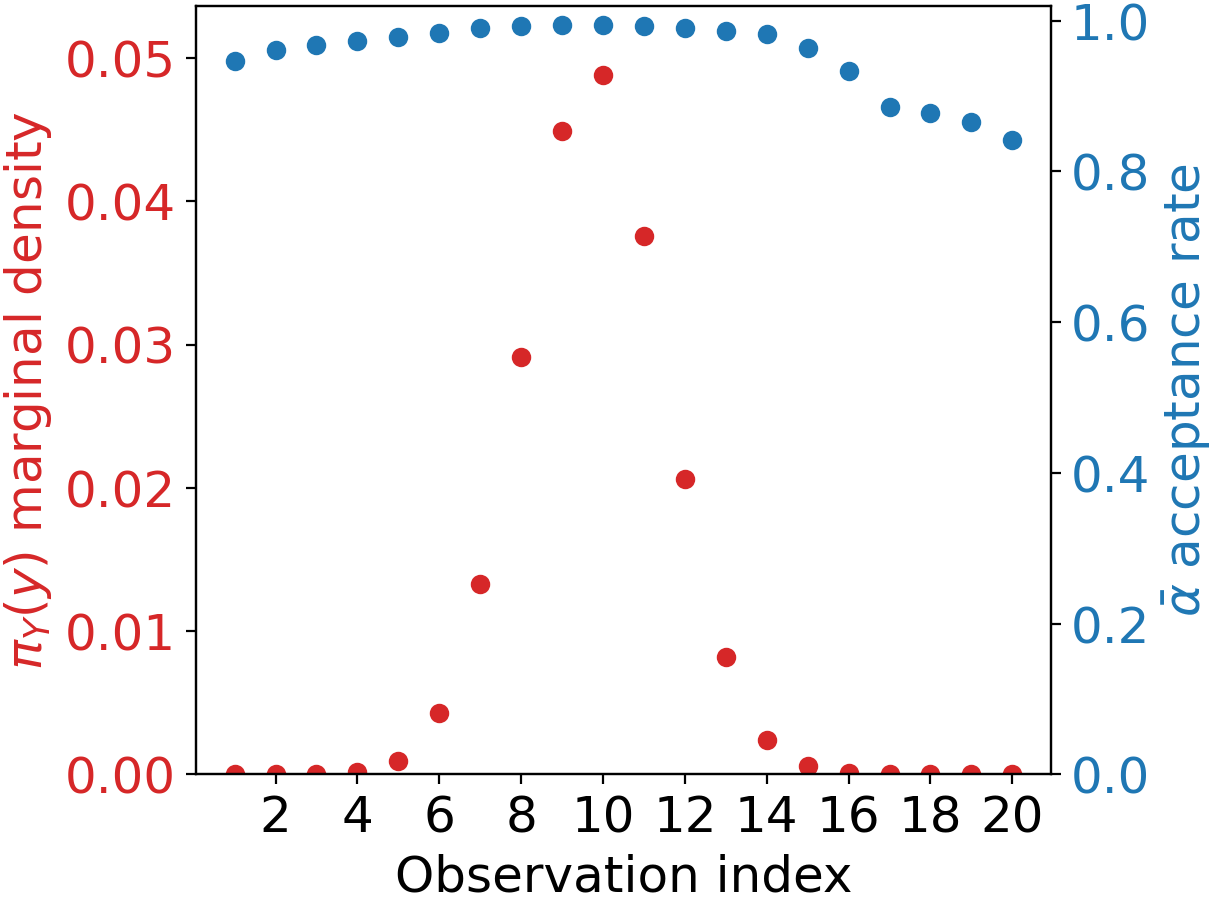}
\end{subfigure}\hfill%
\begin{subfigure}{0.33\linewidth}
    \includegraphics[width=0.98\linewidth]{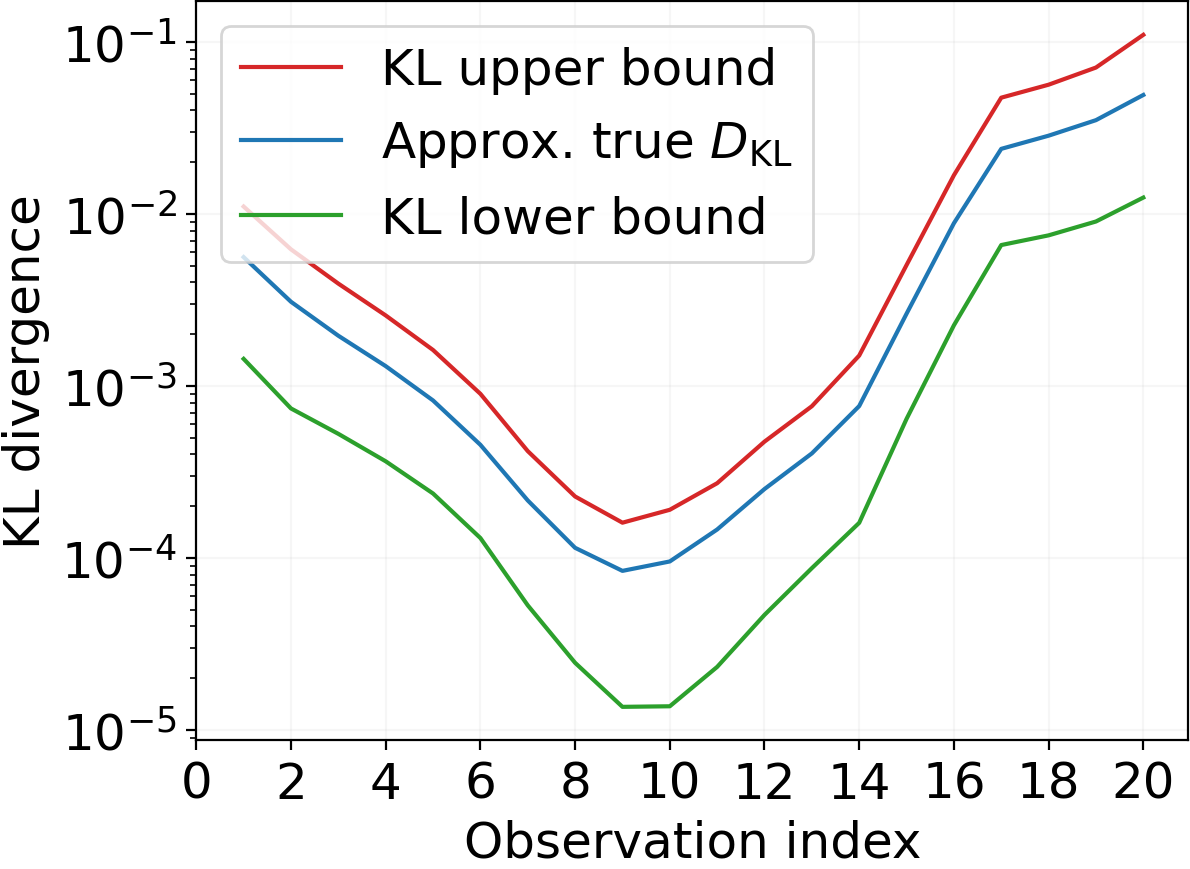}
\end{subfigure}\hfill%
\begin{subfigure}{0.33\linewidth}
    \includegraphics[width=0.98\linewidth]{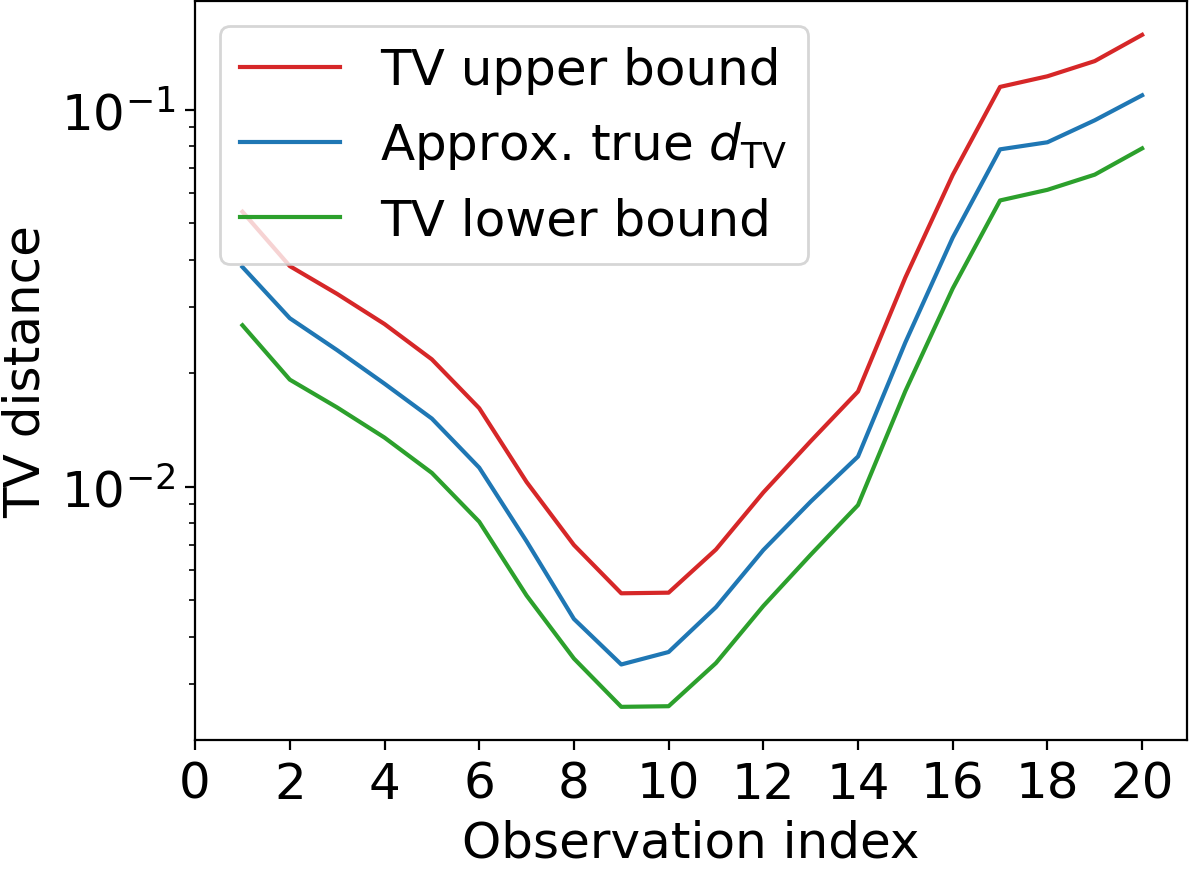}
\end{subfigure}
\caption{%
    Left: prior predictive density $\tdens_Y$ at the 20 observations and acceptance rate of FM-IMH with 500 000 proposals.
    Middle and right: KL divergence and TV distance between posterior and flow matching approximation, computed by quadrature, with the estimated bounds \eqref{eq:kl_bounds} and \eqref{eq:tv_from_acceptance}.}
\label{fig:toy_acceptance}
\end{figure}
\subsection{Diffusion equation with distributed parameter}\label{sec:infinite_heat_example}

As a second example, we consider a diffusion equation in one dimension with an unknown distributed diffusion coefficient $\sigma(z) = \exp(x(z))$.
The difference from \cref{sec:toy} is that the diffusion parameter is no longer assumed to be constant but a function, and that both Dirichlet boundary values are assumed to be known.
Thus, for the spatial variable $z \in \DD =(0,1)$ we consider the boundary value problem
\begin{subequations} \label{bvp}
\begin{equation}\label{eq:bvp}
    - (\sigma(z) \, u'(z) )' = 0, \quad z \in \DD, \qquad  u(0) = 0, \; u(1) = 2,
\end{equation}
where $u: [0,1] \to \R$ denotes the solution.
The prior distribution for the log-coefficient $x:[0, 1] \to \R$ is modeled as a Gaussian random field of Whittle--Mat\'ern type with smoothness $1/2$, i.e.\ with covariance operator
\begin{equation*}
    \mathcal C_0 = 2 \, ( 25 - \Delta )^{-1},
\end{equation*}
where $\Delta$ is the Laplacian on $(0,1)$ with homogeneous Dirichlet boundary conditions.
Away from the boundary its covariance is close to the exponential kernel $\frac15 \exp(-5|z - z'|)$, i.e.\ marginal variance $0.2$ and correlation length $0.2$.
We represent $x$ by its Karhunen--Loève expansion truncated after $\nkl$ terms
\begin{equation*}
    x(z) = \sum_{i=1}^{\nkl} \sqrt{\lambda_i} \, \varphi_i(z) \, \xi_i,
    \qquad \xi_i \sim \mathcal{N}(0,1) \text{ i.i.d.},
\end{equation*}
where $\lambda_i = 2 / (25 + \pi^2 i^2)$ and $\varphi_i(z) = \sqrt2 \sin(i\pi z)$ are the $\nkl$ dominant eigenpairs of $\mathcal C_0$, so that the unknown is effectively parametrized by the  coefficient vector $\vxi = (\xi_1, \dots, \xi_{\nkl}) \in \R^{\nkl}$ endowed with a standard Gaussian prior.
Both $x$ and $\sigma$ are thus random fields depending on $z$ and $\vxi$; we write $x(z) = x(z; \vxi)$ and $\sigma(z) = \sigma(z; \vxi)$ and leave the dependence on $\vxi$ implicit.
For fixed $x$ the flux $\sigma u'$ is constant, so the solution is again available in closed form as
\[
  u(z) = 2 \, \frac{\int_0^z \sigma(s)^{-1} \, \mathrm{d}s}{\int_0^1 \sigma(s)^{-1} \, \mathrm{d}s} ,
\]
and is evaluated numerically by a cumulative Simpson rule on a uniform grid of \num{512} points.
As observations we consider point evaluations of the solution
\begin{equation}\label{eq:bvp_obs}
\vy = (y_i)_{i=1}^{\nobs} + \vvarepsilon \quad \text{with }  \vvarepsilon \sim \NN(0, \gamma^2 \, I_{\nobs}),
\end{equation}
\end{subequations}
where $y_i = u(i \,\Delta z )$ with $\Delta z = 1/(\nobs + 1)$.

In the following, we compare FM against the MCMC algorithms  pCN and DRAM.
For pCN, we tune the parameter $\beta$ to obtain an acceptance rate close to the conventional target \num{23.4}\% \cite{roberts2001optimal}.
For DRAM, the proposal covariance is updated every \num{1000} iterations, and the second-stage proposal is scaled by a factor of \num{0.2}.
After a burn-in of $10^6$ steps, each pCN and DRAM chain runs for $2 \times 10^7$ steps and is thinned by a factor of~500, which leaves \num{40000} samples.
We use this setup in all studies to keep them comparable.
Problem-specific tuning could improve the performance of both samplers, but would not remove the difficulties with
concentrated posteriors.

Unless stated otherwise, the FM velocity field is an MLP with three hidden layers of $\nhidden = 800$ neurons each and swish activations, which takes $(t, x_t, y)$ as input and returns the velocity directly; two modifications of this construction are introduced in \cref{subsubsec:visual}.
It is trained on $2 \times 10^6$ samples from the joint distribution $\tdist$, with a further $2 \times 10^5$ samples for validation, for \num{500} epochs with batch size \num{5000} and a cosine learning-rate schedule decaying from $10^{-3}$ to $10^{-5}$.
Samples are generated by integrating the learned field with a fourth-order Runge--Kutta scheme and 25 steps on the mesh $t_k = 1 - (1 - k/25)^3$, $k = 0, \dots, 25$, graded toward $t = 1$, for the reason given in \cref{subsubsec:visual}; this applies to all results in this section.

\subsubsection{Comparison of FM and MCMC} \label{subsubsec:visual}

\Cref{fig:dram_vs_fm_0,fig:dram_vs_fm_1} display the posterior distributions obtained with flow matching and DRAM.
The four panels correspond to different ground truths $x^\star_j$, $j=1,\dots,4$, i.e., log-coefficient functions $x$ which are used to generate the data:  a sine function, which is highly unlikely under the given prior; a piecewise constant function; and two realizations of the prior truncated at $100$ and $50$ modes, respectively.
Two observation settings are used for each of the four: $\nobs = 32$, $\gamma^2 = 10^{-4}$ (\cref{fig:dram_vs_fm_0}) and $\nobs = 256$, $\gamma^2 = 10^{-5}$ (\cref{fig:dram_vs_fm_1}), both with $\nkl = 128$.
\begin{figure}[!h]
    \centering
    \includegraphics[width=0.9\linewidth]{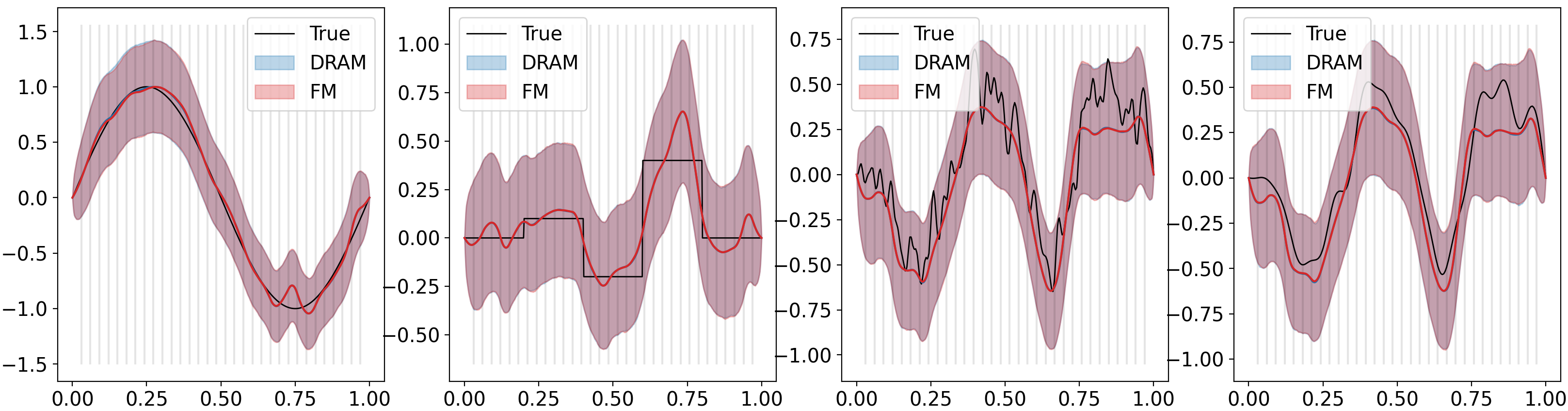}
    \caption{Posterior distributions for problem \eqref{bvp} obtained by DRAM (blue) and FM (red) for four different ground truth coefficients (black solid lines).
    The remaining solid lines indicate the posterior means and the shaded areas the pointwise 95\% credible bands.
    The solution is observed at the 32 locations indicated by the
    vertical gray lines with an observation noise variance
    of $\gamma^2 = 10^{-4}$.}
    \label{fig:dram_vs_fm_0}
\end{figure}
\begin{figure}[!h]
    \centering
    \includegraphics[width=0.9\linewidth]{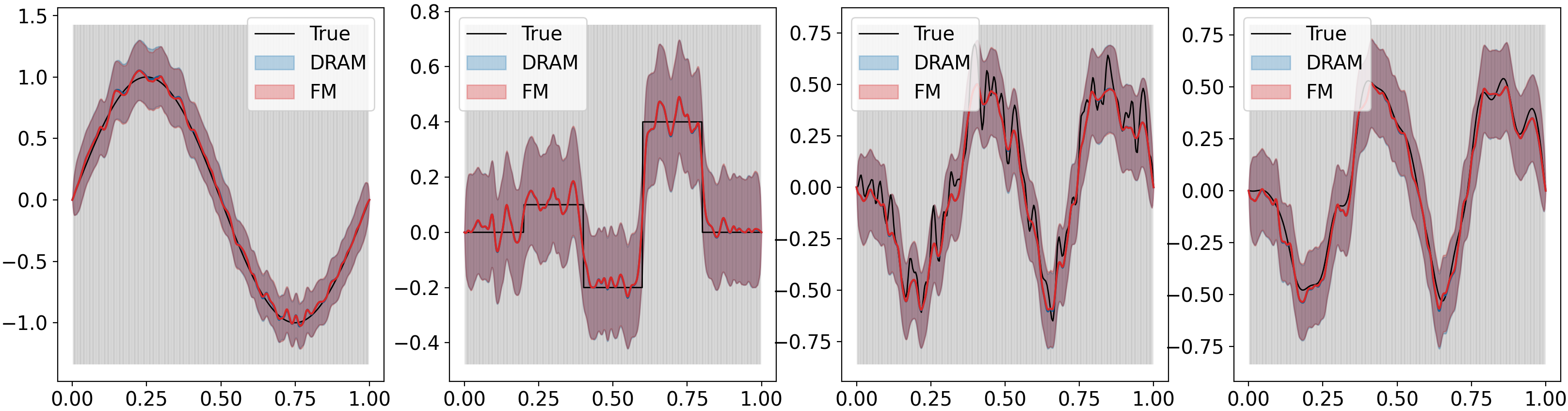}
    \caption{Same as \cref{fig:dram_vs_fm_0} except that the number of observations is increased to 256 and the observation noise variance reduced to $\gamma^2 = 10^{-5}$, resulting in a more concentrated posterior distribution.}
    \label{fig:dram_vs_fm_1}
\end{figure}
In both figures the FM posterior mean and its $95\%$ credible band are visually indistinguishable from those obtained with DRAM, indicating that FM produces an accurate approximation of the true posterior in terms of mean and uncertainty around the mean.
\begin{table}[htbp]
  \centering
  \begin{tabular}{lrrrr}
    \toprule
     & $x^\star_1$ & $x^\star_2$ & $x^\star_3$ & $x^\star_4$ \\
    \midrule
    \multicolumn{5}{l}{\emph{Fig.~\ref{fig:dram_vs_fm_0}:} $\nobs = 32$, $\gamma^2 = 10^{-4}$} \\
    \quad MLP, $\nhidden = 800$
      & \num{51.16} & \num{76.03} & \num{70.13} & \num{69.97}  \\
    \quad MLP + weights, $\nhidden = 800$
      & \textbf{57.01} & \num{79.37} & \textbf{74.69} & \textbf{74.80} \\
    \quad \texttt{gauss}, $\nhidden = 800$
      & \num{31.97} & \num{64.02} & \num{51.18} & \num{51.50} \\
    \quad \texttt{gauss} + weights, $\nhidden = 800$
      & \num{43.20} & \textbf{81.81} & \num{71.40} & \num{71.48} \\
    \quad \texttt{gauss} + weights, $\nhidden = 2400$
      & \num{23.18} & \num{65.55} & \num{55.44} & \num{55.59}  \\
    \addlinespace
    \midrule
    \multicolumn{5}{l}{\emph{Fig.~\ref{fig:dram_vs_fm_1}:} $\nobs = 256$, $\gamma^2 = 10^{-5}$} \\
    \quad MLP, $\nhidden = 800$
      & \num{0.66} & \num{11.21} & \num{5.97} & \num{5.47} \\
    \quad MLP + weights, $\nhidden = 800$
      & \num{1.55} & \num{36.20} & \num{25.47} & \num{27.17} \\
    \quad \texttt{gauss}, $\nhidden = 800$
      & \num{1.07} & \num{22.97} & \num{14.67} & \num{14.93} \\
    \quad \texttt{gauss} + weights, $\nhidden = 800$
      & \num{3.70} & \num{47.49} & \num{33.02} & \num{36.49}  \\
    \quad \texttt{gauss} + weights, $\nhidden = 2400$
      & \textbf{4.65} & \textbf{51.63} & \textbf{39.85} & \textbf{39.89}  \\
    \bottomrule
  \end{tabular}
  \caption{Acceptance rate (\%) of the FM-IMH sampler for the four ground truths
    $x^\star_1,\dots,x^\star_4$ at the two configurations shown in
    Figures~\ref{fig:dram_vs_fm_0} and~\ref{fig:dram_vs_fm_1}; the best rate
    for each ground truth is set in bold.}
  \label{tab:fm_imh_acceptance}
\end{table}

A different picture emerges when the FM posterior is used as the proposal of an independent Metropolis--Hastings sampler (FM-IMH, \cref{subsubsec:imh}).
For the MLP with $\nhidden = 800$ neurons per layer, the acceptance rates in \Cref{tab:fm_imh_acceptance} drop sharply from the setting of \cref{fig:dram_vs_fm_0} to that of \cref{fig:dram_vs_fm_1}.
Hence, as the posterior concentrates, the FM posterior deviates from the true posterior much more than the two figures suggest.

The discrepancy becomes visible in the latent coefficients $\xi_i$ rather than in the field; \cref{fig:heat_marginals} in \cref{app:heat-marginals} shows the marginals of $\xi_1, \dots, \xi_4$ in the setting of \cref{fig:dram_vs_fm_1}.
Three findings guide the modifications below.
First, the data constrain the coefficients very unequally: for $x^\star_2, x^\star_3, x^\star_4$ the posterior standard deviation of $\xi_2$ is about $5.5 \times 10^{-3}$, i.e., roughly $180$ times smaller than its prior standard deviation, whereas, e.g., $\xi_{64}$ and $\xi_{128}$ retain posterior standard deviations between $0.93$ and $1$.
Second, the MLP misrepresents precisely the most strongly constrained coefficients: its marginal of $\xi_2$ is about twice as wide as that of DRAM, and for $x^\star_3$ and $x^\star_4$ its mean of $\xi_4$ is off by $0.7$ to $0.8$ posterior standard deviations.
Measured on the scale of the prior, these errors are of order $10^{-2}$ or smaller, which is why they cannot be seen in \cref{fig:dram_vs_fm_1}.
Third, the FM loss sums the squared velocity errors over all coordinates with equal weight, so that an error which is small on the scale of the prior but large relative to the posterior width is hardly penalized during training.

We therefore consider two modifications of the MLP, which can be applied separately or together.
With \texttt{weights}, the per-coordinate terms of the FM loss are weighted by $w_i \propto 1/s_i$, where $s_i^2$ is the approximated posterior variance of $\xi_i$ resulting from linearizing the parameter-to-observation map at the prior mean --- i.e., the approximation to the posterior covariance matrix is then $(J^\top J/\gamma^2 + I)^{-1}$ with $J$ denoting the Jacobian of the parameter-to-observation map at $\vxi = 0$.
The weights thus require a single Jacobian, but neither a reference chain nor a trained model.
With \texttt{gauss}, the $i$-th component of the velocity field takes the form
\begin{equation}\label{eq:gauss_head}
    v_{\theta,i}(t, x; y) \;=\;
    \underbrace{\omega_i(t)\, x_i + \bigl(1 - t\,\omega_i(t)\bigr) \hat\mu_i(y)}
      _{\text{exact for } \NN\bigl(\hat\mu(y), \operatorname{diag} \psi_i^2\bigr)}
    \;+\; \mathrm{MLP}_i(t, x, y),
    \qquad
    \omega_i(t) \;=\; \frac{t\,\psi_i^2 - (1-t)}{t^2 \psi_i^2 + (1-t)^2},
\end{equation}
where the first term is the exact velocity field of a Gaussian posterior with mean $\hat\mu(y) = W y + b$, linear in $y$, and per-coordinate variances $\psi_i^2$.
The parameters $W$, $b$ and $\psi_i$ are trained together with the MLP, starting from $\psi_i = 1$, so that the MLP (of the default architecture) only has to learn the deviation from a diagonal Gaussian posterior, and narrow marginals no longer have to be generated by the MLP alone.
For small $\psi_i$, $\abs{\omega_i}$ becomes large shortly before $t = 1$, and equally spaced steps would leave the stability region of the Runge--Kutta scheme there; this is the reason for the graded mesh introduced above.
All other settings are those of the MLP.

\Cref{tab:fm_imh_acceptance} reports the acceptance rates of all four combinations with $\nhidden = 800$.
In the concentrated setting of \cref{fig:dram_vs_fm_1}, each modification raises the acceptance rate for every ground truth, the weighting more than the Gauss head, e.g., for $x^\star_3$ from \num{5.97}\% to \num{25.47}\% and \num{14.67}\%, respectively, and the combination of both the most, to \num{33.02}\%.
In the wider setting of \cref{fig:dram_vs_fm_0}, the weighted MLP performs best for three of the four ground truths, whereas the Gauss head without weights lowers the acceptance rate.
The combination of both modifications was additionally trained with $\nhidden = 2400$, which further raises the acceptance rate on the concentrated posterior but lowers it on the wider one.
In the marginals of \cref{fig:heat_marginals}, the combination with $\nhidden = 800$ matches the widths of DRAM for $\xi_2$ and $\xi_4$ to within $3\%$ for $x^\star_2, x^\star_3, x^\star_4$.
The plotted posterior, on the other hand, does not reveal the improvement: in the setting of \cref{fig:dram_vs_fm_1}, the MLP and the combination of both modifications with $\nhidden = 800$ and $\nhidden = 2400$ all reproduce the posterior mean and the bounds of the $95\%$ band of DRAM to between $0.8\%$ and $1.8\%$ in relative $L^2$, the MLP in fact slightly more closely than the other two.
Across all models, the acceptance rate is lowest for $x^\star_1$; we return to this in \cref{subsubsec:prior_mass}.

In \Cref{tab:mcmc_ess}, we report the ESS of \num{1000000} consecutive samples of the chains generated by pCN, DRAM and FM-IMH.
For FM-IMH the proposals are independent, so its ESS measures only the correlation introduced by rejections, whereas for pCN and DRAM it reflects how quickly the chain explores the posterior.
For the purpose of comparison, the ESS of pCN and DRAM is computed on the last \num{1000000} steps of the chains without thinning.
For FM-IMH we draw $10^6$ proposals, generated as for \Cref{tab:fm_imh_acceptance}, and run IMH on them.
The ESS is estimated for each latent coefficient $\xi_i$ \cite{geyer1992practical}, and we report the minimum over all $\nkl = 128$ coefficients.
DRAM yields a noticeably larger ESS than pCN in both observation settings; moreover, the ESS of pCN drops sharply for the more concentrated posterior, while that of DRAM barely changes.
This is expected: DRAM adapts its proposal covariance to the empirical covariance of the chain, whereas the pCN proposal is fixed to a multiple of the prior covariance scaled to achieve a prescribed acceptance rate.
As a consequence, pCN takes steps which are shorter than necessary in directions not informed by the data, thus exploring the posterior less effectively than DRAM.
FM-IMH attains a larger ESS than DRAM for every model and ground truth except $x^\star_1$ with the MLP on the concentrated posterior, by a factor between about \num{1.3} and more than \num{300}.

Measured per forward solve, FM-IMH is more efficient than pCN and DRAM: including the $2\times10^6$ forward solves of its training data, it used $3\times10^6$ solves in total, against $2.1\times10^7$ for each DRAM chain.
Wall-clock times confirm this.
For the setting of \cref{fig:dram_vs_fm_1} and $x^\star_3$, a DRAM chain of \num{2.1e7} states takes about \num{42} minutes on a CPU (AMD Ryzen~7 9700X), and its last \num{1e6} states attain an ESS of \num{1412}, i.e., about \num{85} seconds per \num{1000} effective samples.
For FM with \texttt{gauss} and \texttt{weights}, generating the training data takes about \num{2} minutes and training about \num{14} minutes on a GPU (NVIDIA GeForce RTX~4070~Ti~SUPER).
FM-IMH with \num{1e6} proposals then takes \num{22} minutes, almost entirely spent on evaluating the proposal densities \eqref{eq:backward_likelihood}, and attains an ESS of \num{87703}, i.e., about \num{15} seconds per \num{1000} effective samples, or \num{26} seconds including data generation and training.
Even for this single observation, FM-IMH thus takes \num{38} minutes in total, less than one DRAM chain, and delivers about three times as many effective samples as the \num{2e7} post-burn-in states of that chain, assuming that their ESS grows linearly with the chain length.
Without the Metropolis--Hastings correction, no densities are needed, and \num{1e6} FM samples are drawn in \num{15} seconds.
Since the two methods run on different hardware, these timings indicate orders of magnitude rather than a precise comparison.

\begin{table}[htbp]
  \centering
  \begin{tabular}{lrrrr}
    \toprule
     & $x^\star_1$ & $x^\star_2$ & $x^\star_3$ & $x^\star_4$ \\
    \midrule
    \multicolumn{5}{l}{\emph{Fig.~\ref{fig:dram_vs_fm_0}:} $\nobs = 32$, $\gamma^2 = 10^{-4}$} \\
    \quad pCN & \num{104} & \num{76} & \num{83} & \num{75} \\
    \quad DRAM & \num{1483} & \num{1471} & \num{1589} & \num{1078} \\
    \quad FM-IMH\\
    \quad \quad MLP, $\nhidden = 800$
      & \num{138289} & \num{470585} & \num{419401} & \num{416965} \\
    \quad \quad \texttt{gauss} + weights, $\nhidden = 800$
      & \num{26612} & \num{585327} & \num{403723} & \num{418203} \\
    \quad \quad \texttt{gauss} + weights, $\nhidden = 2400$
      & \num{18243} & \num{351710} & \num{257563} & \num{258677} \\
    \addlinespace
    \midrule
    \multicolumn{5}{l}{\emph{Fig.~\ref{fig:dram_vs_fm_1}:} $\nobs = 256$, $\gamma^2 = 10^{-5}$} \\
    \quad pCN & \num{3.31} & \num{3.07} & \num{3.68} & \num{3.16} \\
    \quad DRAM & \num{1367} & \num{1445} & \num{1412} & \num{1442} \\
    \quad FM-IMH\\
    \quad \quad MLP, $\nhidden = 800$
      & \num{97} & \num{24846} & \num{8995} & \num{4281} \\
    \quad \quad \texttt{gauss} + weights, $\nhidden = 800$
      & \num{1749} & \num{155014} & \num{87703} & \num{91129} \\
    \quad \quad \texttt{gauss} + weights, $\nhidden = 2400$
      & \num{1894} & \num{178779} & \num{136480} & \num{121768} \\
    \bottomrule
  \end{tabular}
  \caption{Effective sample size of \num{1000000} consecutive samples for the
    four ground truths $x^\star_1,\dots,x^\star_4$ at the two configurations
    shown in Figures~\ref{fig:dram_vs_fm_0} and~\ref{fig:dram_vs_fm_1},
    minimum over the \num{128} latent coefficients.}
  \label{tab:mcmc_ess}
\end{table}

\subsubsection{Observations outside the training support} \label{subsubsec:prior_mass}
Across all models in \cref{tab:fm_imh_acceptance}, the acceptance rate is lowest for $x^\star_1$, on the concentrated posterior by an order of magnitude, although its posterior is the least concentrated of the four: the posterior standard deviation of $\xi_2$ is $\num{9.1e-3}$ for $x^\star_1$ against $\num{5.2e-3}$ to $\num{5.6e-3}$ for the other three ground truths.
Concentration therefore does not explain this ordering, and a second, independent mechanism is present.

An amortized model is trained on pairs $(x,\vy)$ drawn from the joint
distribution, so its accuracy at a given $\vy$ depends on how much training data lay near that $\vy$. The relevant statistic is the position of $\vy$ in the distribution of \emph{observations} seen during training, which we measure by the normalized radius
\begin{equation}\label{eq:yradius}
    r(\vy) \;=\; \frac{1}{\sqrt{\nobs}}
    \left\| \frac{\vy - \mu_y}{\sigma_y} \right\|_2 ,
\end{equation}
with $\mu_y, \sigma_y$ the componentwise mean and standard deviation of the training observations, compared against the distribution of $r$ over the training set itself.
In the setting of \cref{fig:dram_vs_fm_1}, the observations of $x^\star_1, \dots, x^\star_4$ have radii $r(\vy) = 2.53$, $0.30$, $0.88$ and $0.87$, against a median training radius of $0.85$.
The radius orders the four ground truths essentially as the acceptance rate does, e.g., \num{4.65}\%, \num{51.63}\%, \num{39.85}\% and \num{39.89}\% for the combination of \texttt{gauss} and \texttt{weights} with $\nhidden = 2400$ in \cref{tab:fm_imh_acceptance}.
The sine observation lies at $2.96$ times the median training radius, with only $0.12\%$ of the $2 \times 10^6$ training observations beyond it: the extreme tail of the prior predictive distribution, where the model has almost no training data.
Note that the criterion is a property of the \emph{observation} and not of the ground-truth field.
The piecewise constant truth is no more a draw of the Mat\'ern prior than the sine is, yet its observation is the most central of the four and it yields the highest acceptance rate; conversely the sine, which is a single Karhunen--Loève mode, produces the most extreme observation.

On the concentrated posterior, the modifications of \cref{subsubsec:visual} raise the acceptance rate for $x^\star_1$ by at most $4$ percentage points, against $9$ to $40$ points for the other three ground truths (\cref{tab:fm_imh_acceptance}).
This is consistent with an absence of training support rather than of capacity, and the same phenomenon recurs for a different forward operator in \cref{subsubsec:eit_support}.
A similar loss of efficiency for out-of-distribution observations has been reported for importance sampling with amortized neural posteriors in gravitational-wave inference \cite{DaxEtAl2023}.

\subsubsection{Parameter studies} \label{subsubsec:heat_studies}

\Cref{fig:infinite_heat_study_nobs} varies the number of observations $\nobs$ and the noise variance $\gamma^2$ at fixed $\nkl = 128$, and \cref{fig:infinite_heat_study_nkl} varies $\nkl$ at fixed $\nobs = 32$, in both cases for the combination of \texttt{gauss} and \texttt{weights} with $\nhidden = 800$ from \cref{subsubsec:visual}, trained separately for each setting.
The solid curves show the mean over $x^\star_2, x^\star_3, x^\star_4$; $x^\star_1$ is drawn separately as a dashed curve since, for the reason given in \cref{subsubsec:prior_mass}, it lies below the range of the other three at nearly every point of both studies.

\begin{figure}[tbp]
    \centering
    \includegraphics[width=0.245\linewidth]{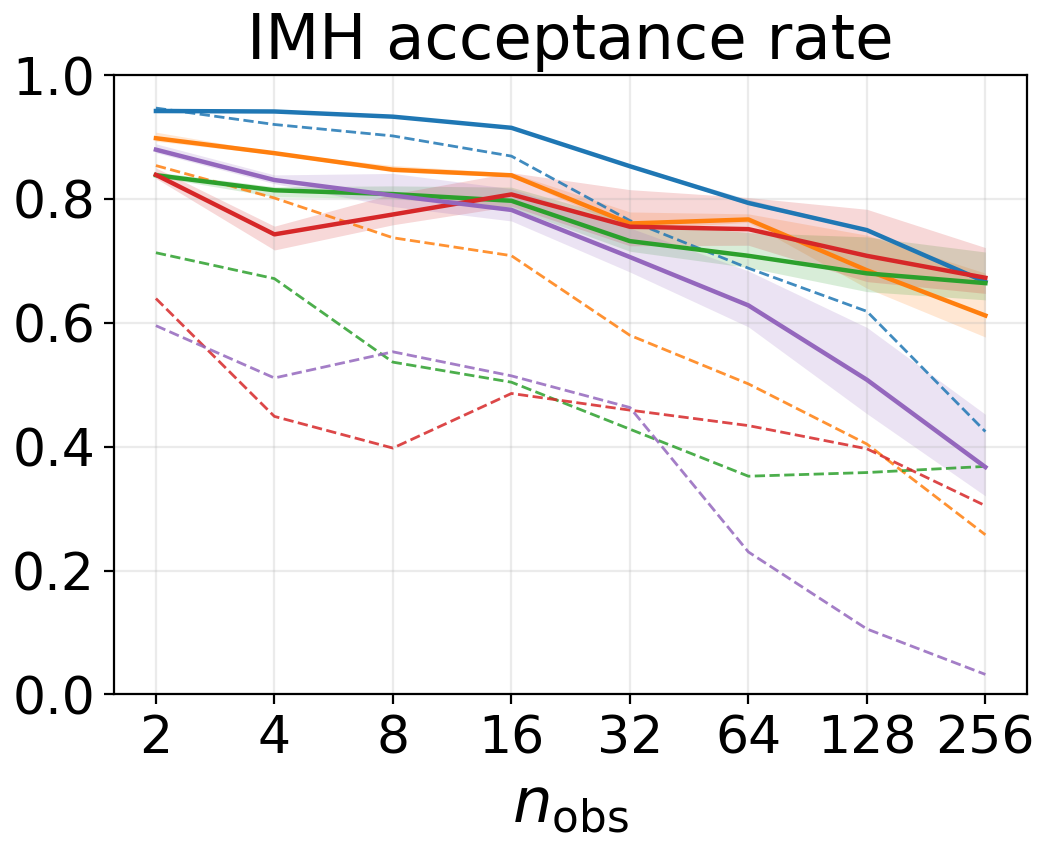}%
    \hfill
    \includegraphics[width=0.245\linewidth]{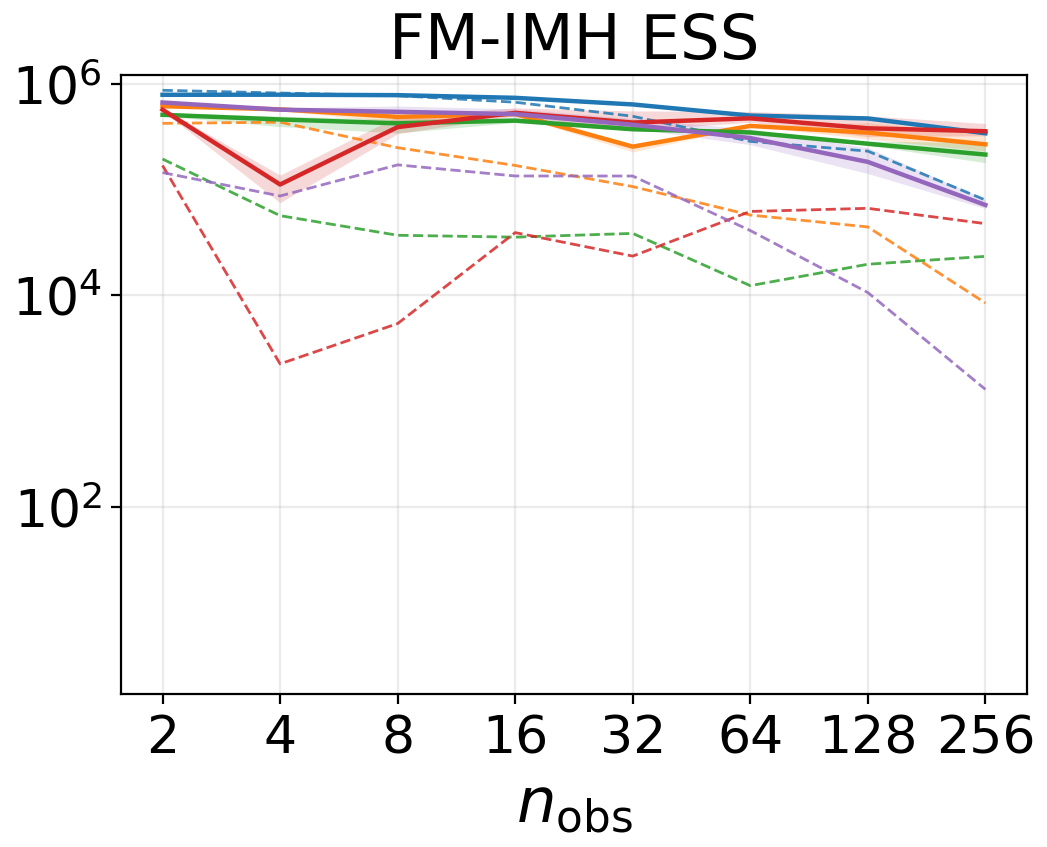}%
    \hfill
    \includegraphics[width=0.245\linewidth]{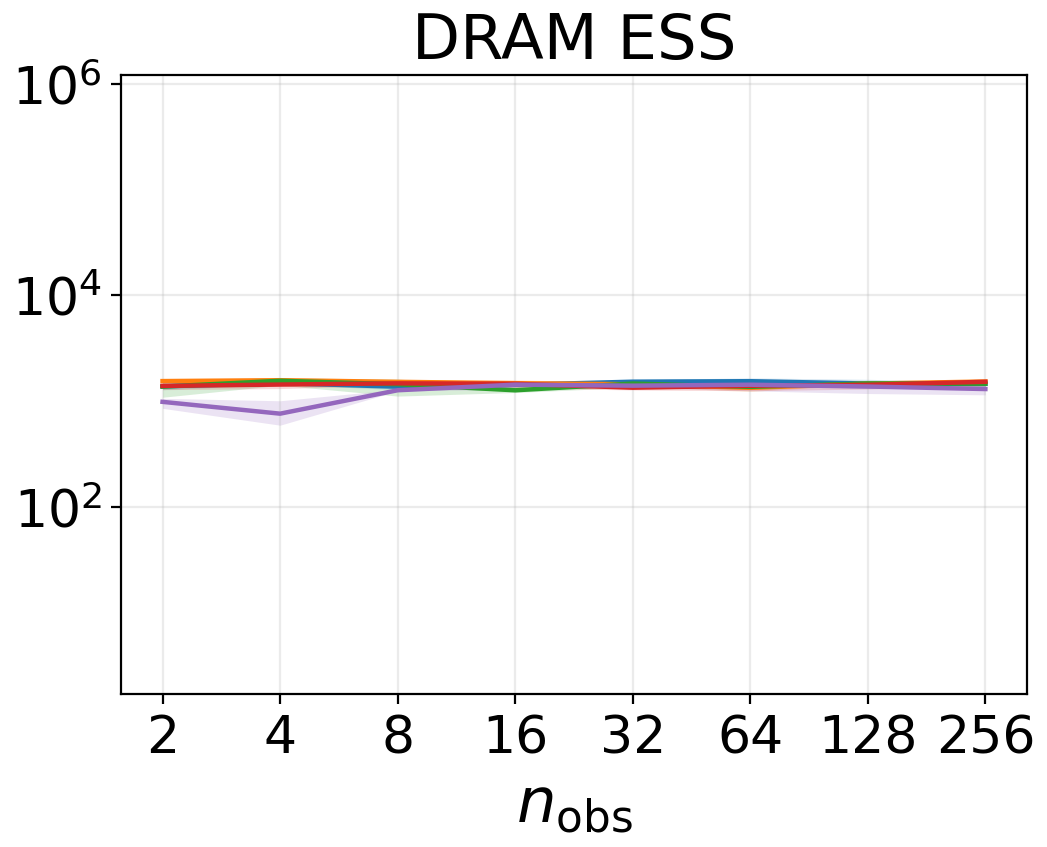}%
    \hfill
    \includegraphics[width=0.245\linewidth]{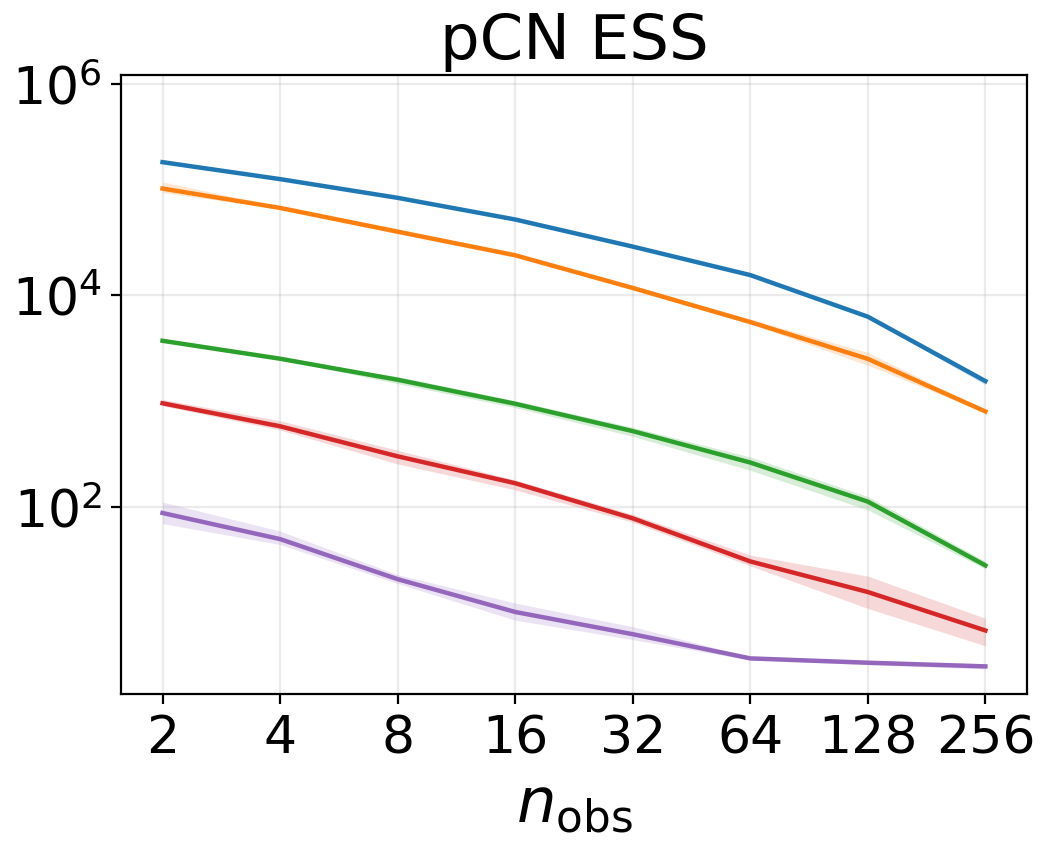}%

    \includegraphics[width=\linewidth]{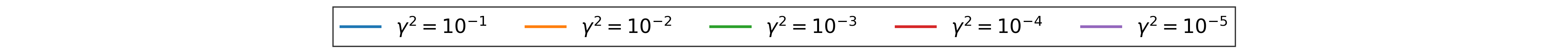}
\caption{Influence of $\nobs$ and the observation noise $\gamma^2$ at
$\nkl = 128$. From left to right: acceptance rate and effective sample size of the
FM-IMH sampler with the combination of \texttt{gauss} and \texttt{weights} of \cref{subsubsec:visual}, $\nhidden = 800$;
effective sample sizes of the DRAM and pCN
reference chains.  Solid curves: mean over ground truths $x^\star_2, x^\star_3, x^\star_4$ and three
training seeds, shaded band: range across these ground truths; dashed curves: ground
truth~$x^\star_1$.  Effective sample sizes are minima over the $\nkl$ latent
coefficients~$\xi_i$ of \num{1000000}~samples, on one common scale in all
panels.}
\label{fig:infinite_heat_study_nobs}

    \vspace{\baselineskip}
    \includegraphics[width=0.245\linewidth]{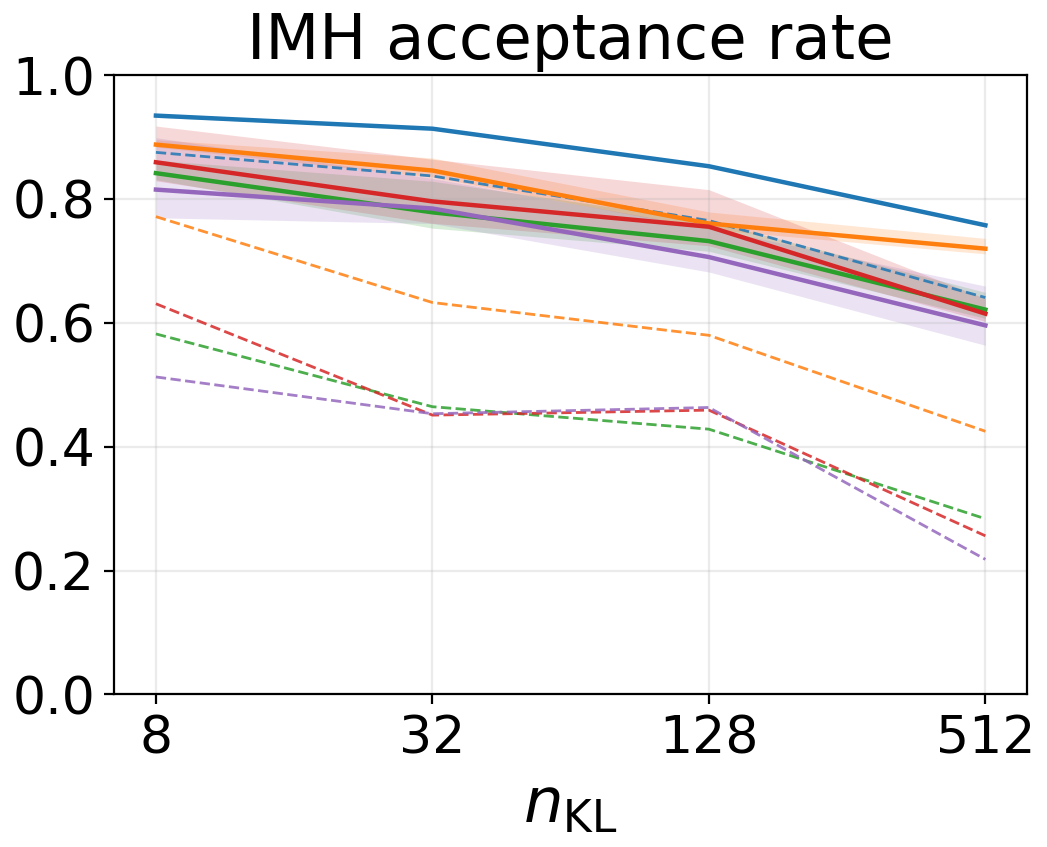}%
    \hfill
    \includegraphics[width=0.245\linewidth]{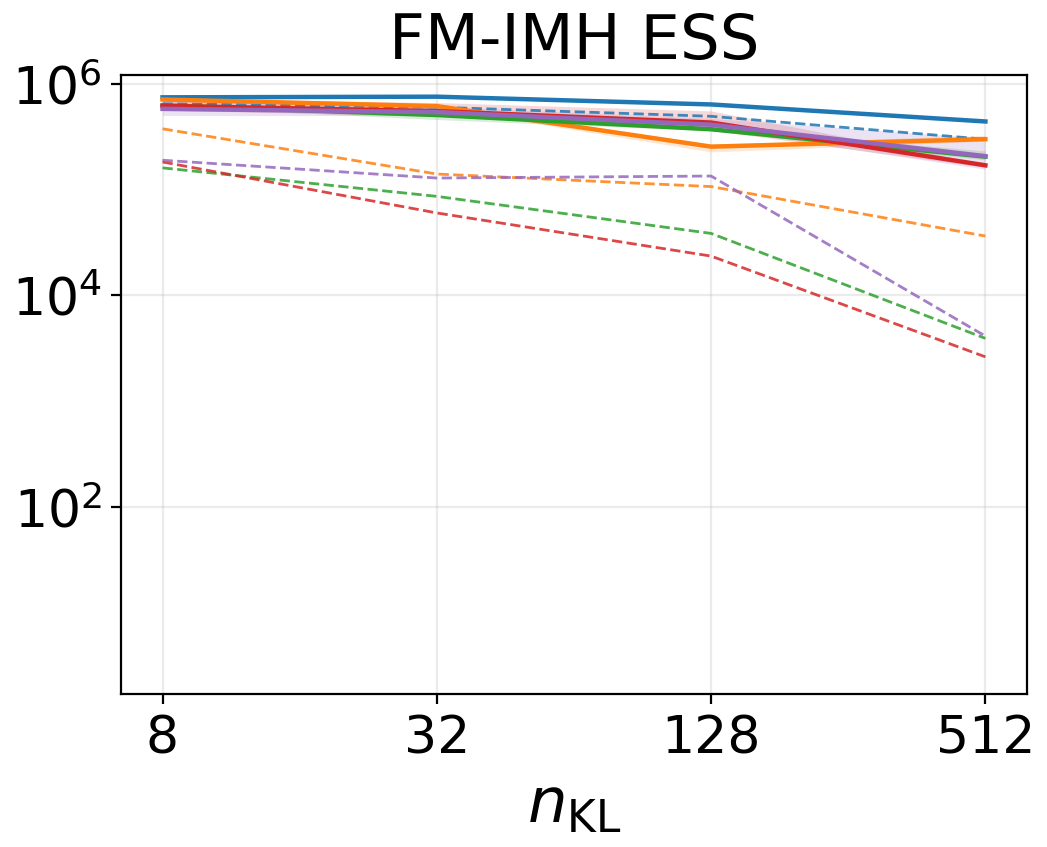}%
    \hfill
    \includegraphics[width=0.245\linewidth]{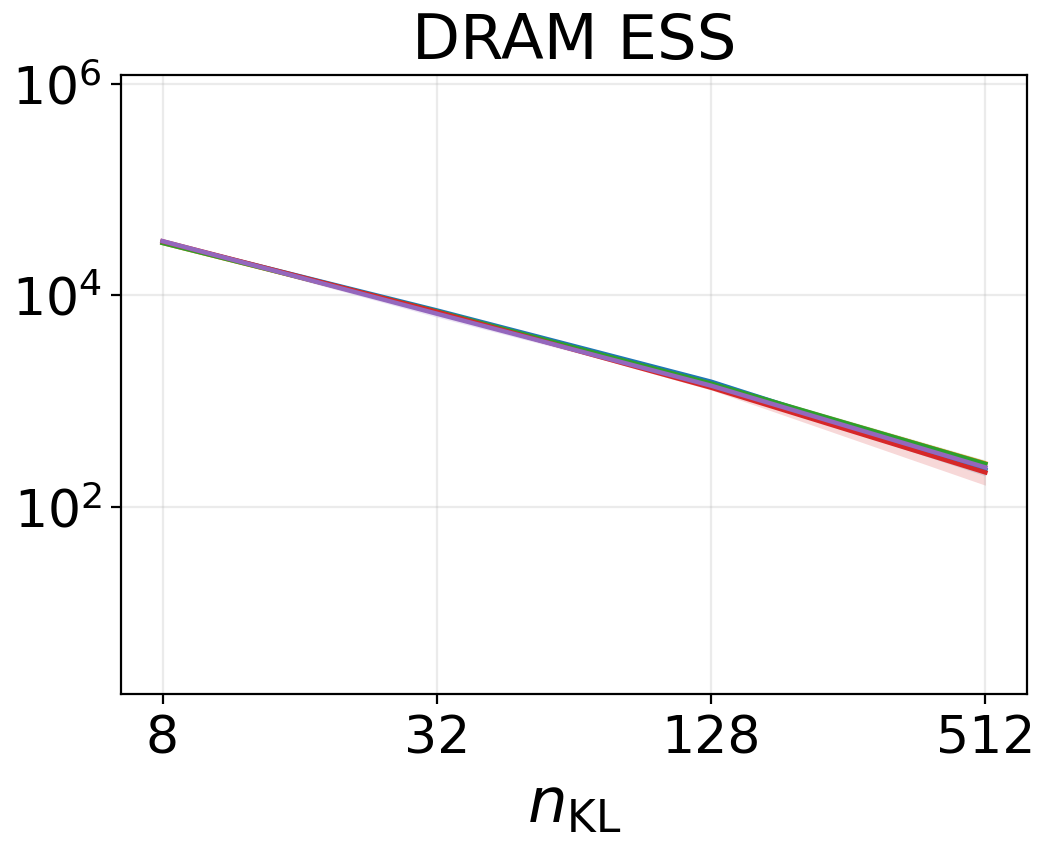}%
    \hfill
    \includegraphics[width=0.245\linewidth]{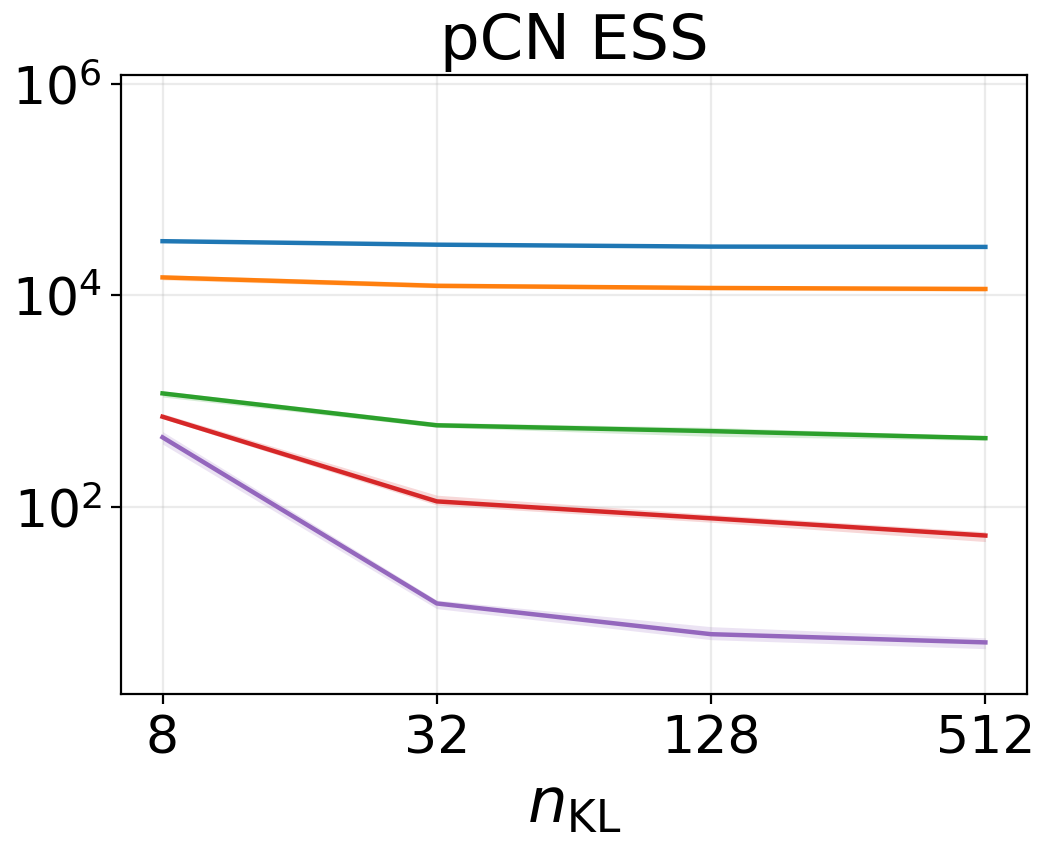}%
    \captionsetup{singlelinecheck=false}
\caption{As \cref{fig:infinite_heat_study_nobs}, but at $\nobs = 32$
with $\nkl$ varying.}
\label{fig:infinite_heat_study_nkl}
\end{figure}

The acceptance rate of FM-IMH decreases as the posterior concentrates.
It falls with $\nobs$ at every noise level, most strongly for the smallest noise variance, from about $88\%$ at $\nobs = 2$ to $37\%$ at $\nobs = 256$ for $\gamma^2 = 10^{-5}$ (left panel of \cref{fig:infinite_heat_study_nobs}).
It also falls with $\nkl$ (left panel of \cref{fig:infinite_heat_study_nkl}).

The ESS of FM-IMH is of order $10^5$ to $10^6$ for $x^\star_2, x^\star_3, x^\star_4$ throughout both studies, and above that of DRAM for all four ground truths.
The ESS of DRAM hardly depends on $\nobs$ and $\gamma^2$, whereas that of pCN drops by several orders of magnitude as the posterior concentrates, for the reason discussed in \cref{subsubsec:visual}.
With respect to $\nkl$ the order is reversed: the ESS of DRAM falls with the chain dimension, while that of pCN levels off beyond $\nkl \approx 32$, which reflects the dimension independence of pCN.
\subsection{Electrical impedance tomography} \label{subsec:eit}

To illustrate the viability of our approach on an inverse problem known to be challenging, we consider electrical impedance tomography (EIT), a severely ill-posed inverse problem in which a spatially distributed electrical conductivity is recovered from a small number of  measurements at boundary electrodes.
We adopt the rigorous Bayesian formulation of \cite{dunlop2016bayesian} and work with all three prior distributions studied there --- log-normal, level set, and star-shaped --- extending the last from a single inclusion to two.
We refer to that paper for the underlying well-posedness theory and give here only the model, the priors, and the parameters needed to reproduce our experiments.

EIT is also the example where both the benefits and limitations of amortization become clear: a single conductivity image is one Bayesian inverse problem, but a clinical EIT device produces a new frame tens of times per second, and each frame is a new observation under the \emph{same} prior and the same forward model.
In this setting, a transport map need be trained only once, whereas a Markov chain must be restarted for each frame.

\subsubsection{The complete electrode model} \label{sec:cem}

Let $\DD \subset \mathbb{R}^2$ be a bounded domain with conductivity $\sigma : \DD \to \mathbb{R}$, and let $L$ electrodes be attached to disjoint arcs $(e_\ell)_{\ell=1}^L$ along the boundary $\partial \DD$ with contact impedances $\kappa_\ell>0$.
Given a current pattern $\vI=(I_\ell)_{\ell=1}^L$ with $\sum_\ell I_\ell=0$, we solve the complete electrode model (CEM) of \cite{somersalo1992existence} for the potential $u$ in $\DD$ and the electrode voltages $\vV=(V_\ell)_{\ell=1}^L$,
\begin{equation} \label{eq:cem}
\begin{cases}
    -\nabla\cdot(\sigma \nabla u) = 0 & \text{in } \DD,\\
    \int_{e_\ell} \sigma \, \partial_n u \, \mathrm{d}S = I_\ell, & \text{for } \ell=1,\dots,L,\\
    \hspace{3.0em}\sigma\, \partial_n u = 0 & \text{on } \partial \DD \setminus {\textstyle\bigcup_\ell} e_\ell, \\
u + \kappa_\ell\, \sigma\, \partial_n u = V_\ell & \text{on } e_\ell,\; \ell=1,\dots,L,
\end{cases}
\end{equation}
which accounts for the shunting effect of the electrodes and their contact impedances and has been found to agree well with experimental EIT data \cite{somersalo1992existence}.
The solution is unique up to an additive constant, which corresponds to the choice of ground potential; following \cite{dunlop2016bayesian} we fix this by taking the representative with $\sum_\ell V_\ell = 0$.

We take $\DD$ to be the unit disk with $L=16$ electrodes of contact impedance $\kappa_\ell = 0.01$, placed equidistantly along $\partial \DD$ and covering 50\% of the boundary.
We inject $J=15$ linearly independent adjacent current patterns of amplitude $0.1$.
We discretize \eqref{eq:cem} with continuous piecewise linear finite elements for the potential and a piecewise constant conductivity, taking on each triangle the value of $\sigma$ at its centroid.
Synthetic data are generated on a fine mesh of $75\,878$ triangular elements that conforms to the true inclusions, and the inversion runs on a coarser mesh of $\DD$ consisting of $12\,144$ elements; both meshes are comparable in size to those used in \cite{dunlop2016bayesian}.
Solving \eqref{eq:cem} once per pattern yields a $J \times L$ array of electrode voltages.
Since our measurement setup departs from \cite{dunlop2016bayesian}, we state it explicitly: each pattern is referenced to the electrode following its sink, and the two injecting electrodes together with the reference are dropped.
This leaves $13$ entries for each of the $J = 15$ patterns, i.e., $\nobs = 13 \cdot 15 = 195$ observations, and
\begin{equation} \label{eq:eit-data}
    \vy = F(\sigma) + \vvarepsilon \in \R^{\nobs},
    \qquad \vvarepsilon \sim \NN(0, \gamma^2 \, I_{\nobs}),
\end{equation}
with $\gamma = 2 \times 10^{-4}$, i.e., an observation noise variance of $4 \times 10^{-8}$.
In the notation of \eqref{eq:forward}, the unknown parameter $x$ is thus the conductivity $\sigma$ itself.
As shown in \cite[Prop.~2.7]{dunlop2016bayesian}, $F$ is continuous with respect to both uniform convergence and convergence in Lebesgue measure, which is what permits a rigorous Bayesian treatment directly on function space and is inherited by any prior pushed forward through a parametrization $\sigma = \sigma(\vxi)$ of the types to be considered below.

Throughout, $\vxi \in \R^{d}$ denotes the latent variable carrying a standard Gaussian prior, as the Karhunen--Lo\`eve coefficients $\xi$ do in \cref{sec:infinite_heat_example}, so that all three priors take the common form
\begin{equation}
\label{eq:eit-latent}
    \sigma = \sigma(\vxi), \qquad \vxi \sim \NN(0, I_d),
    \qquad
    \log \tdens^{\vy}(\vxi) = -\Phi(\vxi) - \tfrac12 \norm{\vxi}^2 + \text{const},
\end{equation}
with misfit $\Phi(\vxi) = \frac{1}{2\gamma^2} \norm{F(\sigma(\vxi)) - \vy}^2$.
Here $\sigma(\vxi)$ is short for the conductivity field $\sigma(\cdot\,; \vxi)\colon \DD \to \R$; where the spatial variable matters we write $\sigma(\vz) = \sigma(\vz; \vxi)$ and leave the dependence on $\vxi$ implicit.
The synthetic ground truth is that of \cite{dunlop2016bayesian}: a background $\sigma = 2$ containing two disjoint circular inclusions of radius $0.3$ at $\pm(0.4,0.4)$ with $\sigma = 1$.

\subsubsection{Priors} \label{sec:eit-priors}

All three priors of \cite{dunlop2016bayesian} are used, each written as a pushforward $\sigma = \sigma(\vxi)$ of the standard Gaussian latent in \eqref{eq:eit-latent}; typical draws are shown in \cref{app:eit-prior-draws}.
We describe only where our choices differ from theirs.

\paragraph{Log-normal}
$\sigma(\vz) = \exp(m_0 + G(\vz; \vxi))$ with a Gaussian random field $G$ with Whittle--Mat\'ern covariance, truncated to the $M\times M$ lowest eigenfunctions $\phi_{k\ell}$ of $-\Delta$ on $[-1,1]^2$ with homogeneous Neumann boundary conditions, i.e., normalized products of cosines $\phi_{k\ell}(\vz) \propto \cos\bigl(\tfrac{k\pi}{2}(z_1+1)\bigr)\cos\bigl(\tfrac{\ell\pi}{2}(z_2+1)\bigr)$, $\vz = (z_1, z_2)$, with eigenvalues $\lambda_{k\ell}$,
\begin{equation} \label{eq:eit-lognormal}
  G(\vz; \vxi) = \sum_{k,\ell=0}^{M-1} \sqrt{\mu_{k\ell}}\,\phi_{k\ell}(\vz)\,\xi_{k\ell},
  \qquad
  \mu_{k\ell} = q\,(\tau^2 + \lambda_{k\ell})^{-\alpha},
  \qquad
  \lambda_{k\ell} = \tfrac{\pi^2}{4}(k^2+\ell^2),
\end{equation}
with $m_0 = \tfrac12\log 2$, $\tau = 40$, $\alpha = 6$ and $q = 10^{16}$ as in \cite{dunlop2016bayesian}, giving a marginal standard deviation of $G$ of about $0.125$.
\textcite{dunlop2016bayesian} define the same prior on the square, truncate at $M = 128$ and restrict to the disk; we truncate at $M = 32$, $d = 1024$, which already captures $99.4\%$ of the pointwise prior variance at $M = 128$.
At $\gamma = 2\times 10^{-4}$ our data have a mean relative error $\ev{\|\vvarepsilon\|}/\abs{F(\sigma_{\mathrm{true}})} = 1.2\%$.\footnote{\textcite{dunlop2016bayesian} report a relative error of $10\%$ but do not state how it is defined. We were unable to reproduce this value, neither with our measurement setup nor with theirs, which uses absolute electrode voltages.}

\paragraph{Level set}
$\sigma(\vz) = 1 + \mathbf 1\{G(\vz; \vxi) > 0\}$ with $G$ as in \eqref{eq:eit-lognormal}, $\tau = 35$, $\alpha = 5$ and $M = 32$.
Since $G(\vz; c\,\vxi) = c\, G(\vz; \vxi)$ and only its sign enters, $\sigma(c\,\vxi) = \sigma(\vxi)$ for every $c > 0$.
The amplitude $q$ is therefore irrelevant, and we normalize $G$ to unit marginal standard deviation.
The same invariance makes the latent posterior scale-invariant.

\paragraph{Star-shaped, two inclusions}
\textcite{dunlop2016bayesian} parametrize the boundary of a single inclusion, which cannot represent their own two-inclusion truth, and name several inclusions as future work.
We use two, each with its own center $c_j$ (a shared center would force nested inclusions), log-size $\rho_{j0}$ and Fourier boundary,
\begin{equation}\label{eq:starshaped}
  \partial\Omega_j = \{\vc_j + r_j(\vartheta)(\cos\vartheta,\sin\vartheta)\}, \quad
  r_j(\vartheta) = \exp\Bigl(\rho_{j0} + \sum_{k=1}^{5} a_{jk}\cos k\vartheta + b_{jk}\sin k\vartheta\Bigr),
\end{equation}
with $a_{jk}, b_{jk} \sim \NN(0, \gamma_k^2)$, $\gamma_k = 0.3\,k^{-2}$, $\vc_j \sim \NN(0, 0.3^2 I_2)$ and $\rho_{j0} \sim \NN(\log 0.3 - \tfrac12\sum_k\gamma_k^2,\ 0.15^2)$.
Inclusions have conductivity $1$ on a background of $2$ and we impose no non-overlap constraint.
The latent dimension is $d = 26$.
This is the only one of the three priors whose support contains the truth.

For both two-valued priors an element changes value only when an interface crosses its centroid, so $\partial_{\vxi}\sigma = 0$ almost everywhere.
Measured on the inversion mesh, a random latent perturbation of size $10^{-3}$ changes no element at all, whereas under the log-normal prior it changes all of them.
Gauss--Newton Hessians, Laplace approximations and likelihood-informed subspaces are therefore unavailable for these two priors.

\subsubsection{Results}
\label{sec:eit-results}\label{sec:eit-ood}\label{subsubsec:eit_support}

For each prior we condition the flow on $\vy$, draw $n = 10\,000$ samples, and compare them with an MCMC reference posterior, whose step sizes were tuned on stationary states; \cref{fig:eit-vobs-grid} shows the posterior means, and \cref{app:eit-vobs} gives the full comparison including standard deviations and the samplers.
The same samples serve as proposals for IMH, for which we report the acceptance rate, the number of accepted moves and the effective sample size (ESS) of the chain (\cref{tab:eit-imh-all}).
We give the ESS for two quantities along the chain: the smallest over the latent coordinates $\xi_i$, as in \cref{sec:infinite_heat_example}, and the mean of the conductivity $\sigma(\vxi)$ over all elements.
For the two-valued priors the mean and the standard deviation carry the same information, $\mathrm{std} = \sqrt{p(1-p)}$ with $p = \prob(\sigma = 1) = 2 - \text{mean}$, so we score both through the inclusion probability $p$:
\begin{itemize}
\item
Under the star-shaped prior the flow reproduces the reference almost exactly, with a correlation of $\operatorname{corr}(p_{\mathrm{FM}}, p_{\mathrm{MCMC}}) = 0.9989$ between the inclusion probabilities
and an inclusion area of $0.182$ for both methods against a true value of $0.178$.
\item
Under the log-normal prior the two posterior means agree to a correlation of $0.995$ and the standard deviations to $0.979$.
Both means are far from the truth, because the prior places almost no mass near it: on the log scale, the inclusion and background values lie 2.8 marginal standard deviations below and above the prior mean, and the smooth, truncated fields of \eqref{eq:eit-lognormal} cannot form sharp interfaces.
This is a property of the posterior, not of either sampler.
The same mismatch places $\vy$ out of distribution for the flow, as discussed below.
\item
Under the level-set prior the flow locates both inclusions with $\operatorname{corr}(p_{\mathrm{FM}}, p_{\mathrm{MCMC}}) = 0.975$ but is over-dispersed: its standard deviation is $1.5$ times that of the reference, and it assigns an inclusion area of $0.222$ against $0.180$.
\end{itemize}
On $\vy$, only the star-shaped flow yields a usable IMH chain: it accepts $7.5\%$ ($753$ moves, ESS $78$ in $\vxi$ and $192$ in the mean conductivity), against $0.20\%$ and $0.13\%$ under the log-normal and the level-set prior (\cref{tab:eit-imh-all}).
\Cref{fig:eit-vobs-grid} shows that the flow reproduces the reference posterior mean under all three priors, including the two under which IMH accepts almost no steps.
The low acceptance there has two causes.
First, $\vy$ is out of distribution for both priors: its normalized radius \eqref{eq:yradius} is $r(\vy) = 3.73$ under the log-normal and $r(\vy) = 1.65$ under the level-set prior, against a training median of about $0.97$, and of the $2\times10^6$ training observations none and $530$ ($0.03\%$), respectively, lie as far out, whereas under the star-shaped prior $r(\vy) = 0.94$ is typical, exceeded by $35\%$ of the training observations.
Under the level-set prior this is visible directly in the fields: a prior draw covers about half of the disk with inclusion (area fraction $0.500 \pm 0.046$), while the true inclusions cover only $0.178$.
The flow has therefore seen almost no data like $\vy$ and has to extrapolate, and neural networks are notoriously unreliable outside the range of their training data \cite{XuEtAl2021}: the training objective does not constrain the learned velocity field there, so its behavior is set by the architecture rather than by the data.
Second, the level-set latent posterior is hard to match for any flow, even in distribution: since $\sigma$ depends on $\vxi$ only through the sign of $G$, $\tdist^{\vy}$ is scale-invariant, and it is thin in every latent direction, because at the interface even modes carrying a few percent of the prior variance decide the sign; small latent errors of the flow therefore change the data misfit strongly, and the likelihood of its proposals varies over many orders of magnitude.
The acceptance rate certifies the flow as a sampler of the latent posterior $\tdist^{\vy}$, not of the conductivity: scored on the identifiable inclusion probability $p$, the level-set flow agrees with the pCN reference at a correlation of $0.975$.

We discuss the out-of-distribution position of $\vy$ separately, because it limits where amortized inference can be applied.
MCMC requires only that the prior have \emph{support} where the posterior concentrates; the likelihood supplies the transport.
An amortized flow never sees the likelihood at training or test time --- the prior \emph{is} its training distribution --- so it requires prior \emph{mass} there.
Support is therefore not enough: a prior can contain the truth exactly and still place almost no mass near it, and the flow then has to extrapolate as under the other two priors.
For the star-shaped prior we therefore choose the size and center priors such that the truth is also a typical draw.

\begin{figure}[htbp]
    \centering
    \includegraphics[width=\linewidth]{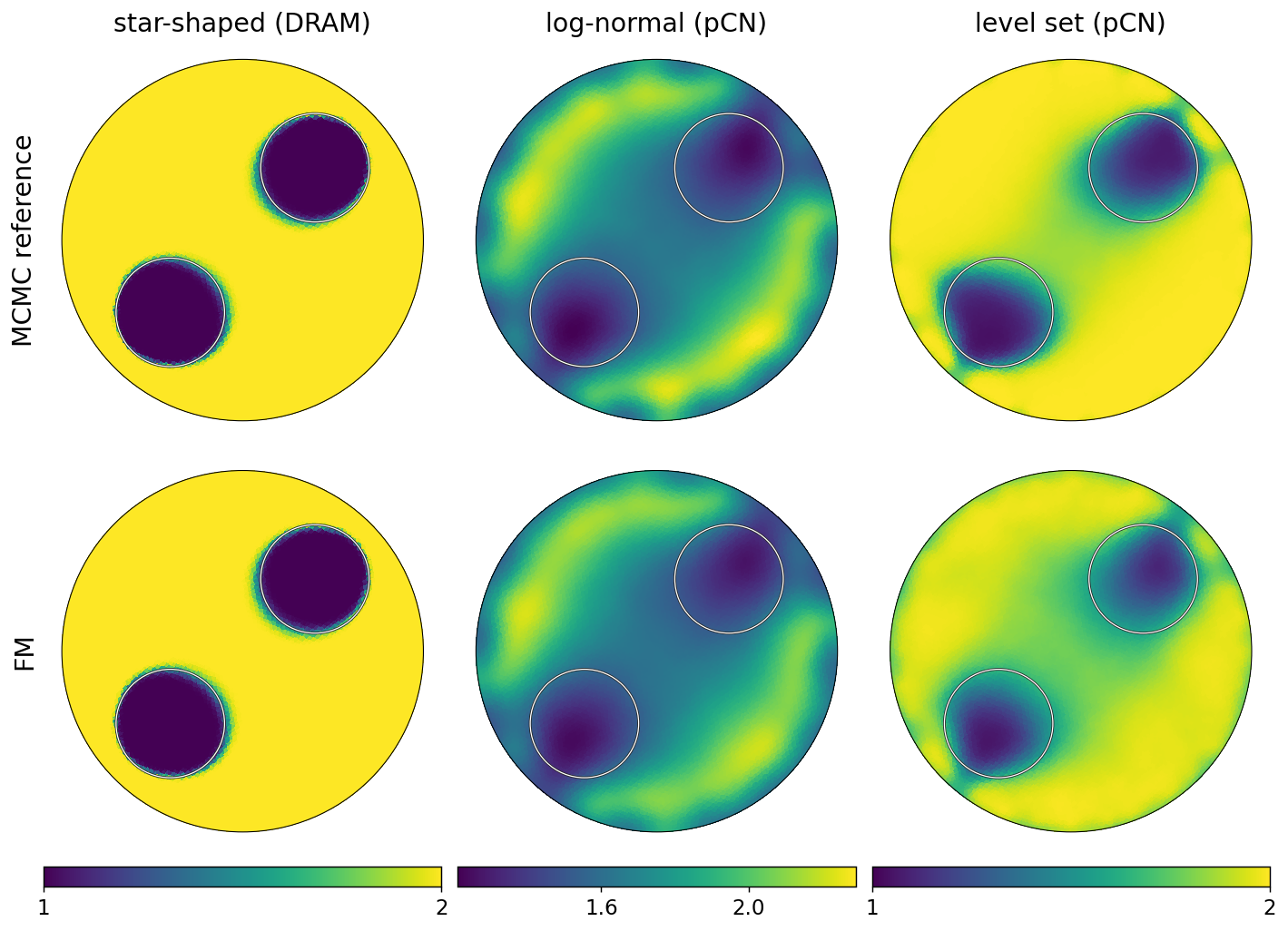}
    \caption{Posterior means of the conductivity conditioned on $\vy$, one column per prior: MCMC reference (top) and flow matching with $10\,000$ samples (bottom), on a color scale shared within each column. Circles mark the true inclusions. Standard deviations and the sampler settings are in \cref{app:eit-vobs}.}
    \label{fig:eit-vobs-grid}
\end{figure}

On observations drawn from the prior, which are in distribution by construction, the picture changes.
Over $20$ prior-draw observations, IMH accepts $36.7\% \pm 2.5\%$ (mean $\pm$ standard deviation) under the log-normal prior and $26.2\% \pm 17.9\%$ under the star-shaped prior, but only $0.11\% \pm 0.02\%$, and at most $0.15\%$, under the level-set prior (\cref{tab:eit-imh-all}).
For the star-shaped prior we exclude draws in which an inclusion crosses the boundary of the disk and use the first $20$ remaining ones.
Under the log-normal prior the low acceptance on $\vy$ is therefore a consequence of its out-of-distribution position and not of the flow, whereas under the level-set prior it persists for the second reason above.
The wide spread under the star-shaped prior, from $0.3\%$ to $52.8\%$, comes from the observations rather than from the flow: acceptance is lowest where a drawn inclusion boundary comes close to the electrodes.
\Cref{fig:eit-priorobs} in \cref{app:eit-priorobs} shows the posterior means for one such observation, the second block of \cref{tab:eit-imh-all}.

Low ESS values should be read as a warning rather than as a measurement: where a chain makes only a few dozen moves, the ESS estimate itself is unreliable.
Such values arise from a drawback of IMH: once the chain reaches a state with a large importance weight~$w$, i.e., where the posterior density is large relative to the flow density, it rejects almost every subsequent proposal until one of comparable weight arrives, and remains stuck at that state in the meantime.
Unlike a local sampler such as pCN, it cannot move away through small steps, since its proposals are drawn independently of the current state.
Interleaving IMH steps with local moves, as in \cite{GabrieEtAl2022}, would mitigate this.

\paragraph{Implementation details}
The architectures were obtained by a small and non-exhaustive hyperparameter search.
All three flows are MLPs with the Gauss head \eqref{eq:gauss_head}, trained on $2\times10^6$ pairs, validated on a further $2\times10^5$, and integrated with RK4 in $25$ steps.
For the star-shaped prior the MLP has $5$ hidden layers of width $800$ with swish activations and is trained for $300$ epochs; for the log-normal prior it has $4$ hidden layers of width $800$ with GELU activations and is trained for $120$ epochs.
For the level-set prior we set $\hat\mu \equiv 0$ and $\psi_i \equiv 1$ in \eqref{eq:gauss_head}, so that the linear part is the exact velocity between two standard Gaussians, and the MLP has $3$ hidden layers of width $800$ with swish activations and is trained for $120$ epochs.

\begin{table}[htbp]
\centering\footnotesize
\caption{IMH with $n = 10\,000$ flow-matching proposals (RK4, $25$ steps) on $\vy$, on the prior draw of \cref{fig:eit-priorobs}, and over $20$ prior-draw observations. The ESS is given for the smallest over the latent coordinates $\xi_i$ and for the mean conductivity over all elements.}
\label{tab:eit-imh-all}
\begin{tabular}{lccc}
\toprule
 & star-shaped & log-normal & level set \\
\midrule
\multicolumn{4}{l}{\emph{Observation $\vy$ (two-inclusion truth), shown in \cref{fig:eit-vobs-grid}}} \\
\quad acceptance rate & $7.5\%$ & $0.20\%$ & $0.13\%$ \\
\quad accepted moves & $753$ & $19$ & $12$ \\
\quad ESS ($\vxi$) & $78$ & $3$ & $6$ \\
\quad ESS (mean $\sigma$) & $192$ & $4$ & $21$ \\
\addlinespace
\multicolumn{4}{l}{\emph{Prior draw (highest star-shaped acceptance of the 20), shown in \cref{fig:eit-priorobs}}} \\
\quad acceptance rate & $52.8\%$ & $38.6\%$ & $0.09\%$ \\
\quad accepted moves & $5\,279$ & $3\,857$ & $8$ \\
\quad ESS ($\vxi$) & $1\,841$ & $352$ & $6$ \\
\quad ESS (mean $\sigma$) & $2\,906$ & $1\,029$ & $9$ \\
\addlinespace
\multicolumn{4}{l}{\emph{Acceptance over $20$ prior draws}} \\
\quad mean $\pm$ sd & $26.2 \pm 17.9\%$ & $36.7 \pm 2.5\%$ & $0.11 \pm 0.02\%$ \\
\quad median & $29.7\%$ & $36.5\%$ & $0.11\%$ \\
\quad range & $0.3$--$52.8\%$ & $30.9$--$42.7\%$ & $0.07$--$0.15\%$ \\
\bottomrule
\end{tabular}
\end{table}

\subsubsection{Dynamic Bayesian flow matching} \label{sec:eit-dynamic}

The static comparison assesses whether the flow is accurate; a clinical EIT device also requires sampling fast enough to image moving structures, and here the amortized model has a clear advantage over MCMC.
The star-shaped flow of \cref{sec:eit-priors} is trained only once and then conditioned in turn on the 195 measurements of each frame of a time-varying phantom; nothing is refitted between frames.
The data are generated on the fine mesh with the measurement protocol and noise level of $\vy$, and every frame is an ordinary observation under the star-shaped prior.
Two inclusions at $(\pm0.42, 0)$ breathe between radii $0.20$ and $0.35$, one cosine cycle per breath, over $720$ frames: two breaths with both inclusions, then four in which only the right one breathes while the left stays at its exhaled radius.
With $1\,000$ samples per frame (RK4, 10 steps) the flow follows both phases (\cref{fig:eit-dynamic-lung-stop}): after the second breath the reconstructed left inclusion stays at its exhaled size while the right one keeps breathing.
The posterior standard deviation appears to capture where the reconstruction is uncertain: it is concentrated in a band along each inclusion boundary, the only place the two-valued field can vary, and this band is wider on the side facing the center of the disk, far from the electrodes, where EIT resolves the boundary least well.
The posterior median deviates from the truth by an RMSE of $0.122$ on average, ranging from $0.056$ to $0.205$ over the frames.
Its inclusion area $\int_{\DD} \mathbf{1}_{\sigma(\vz) = 1}\,\mathrm{d}\vz$, with $\sigma$ the posterior median, differs from the exact area by $0.002$ on average and by at most $0.043$ in any frame, with a correlation of $0.997$ over the frames.
A frame takes $21.7$\,ms on a single consumer GPU (NVIDIA GeForce RTX 4070 Ti SUPER), i.e., $46$ frames per second.
\Cref{app:eit-orbit} shows a second sequence, in which two inclusions orbit the center and then merge.

\begin{figure}[htbp]
    \centering
    \includegraphics[width=\linewidth]{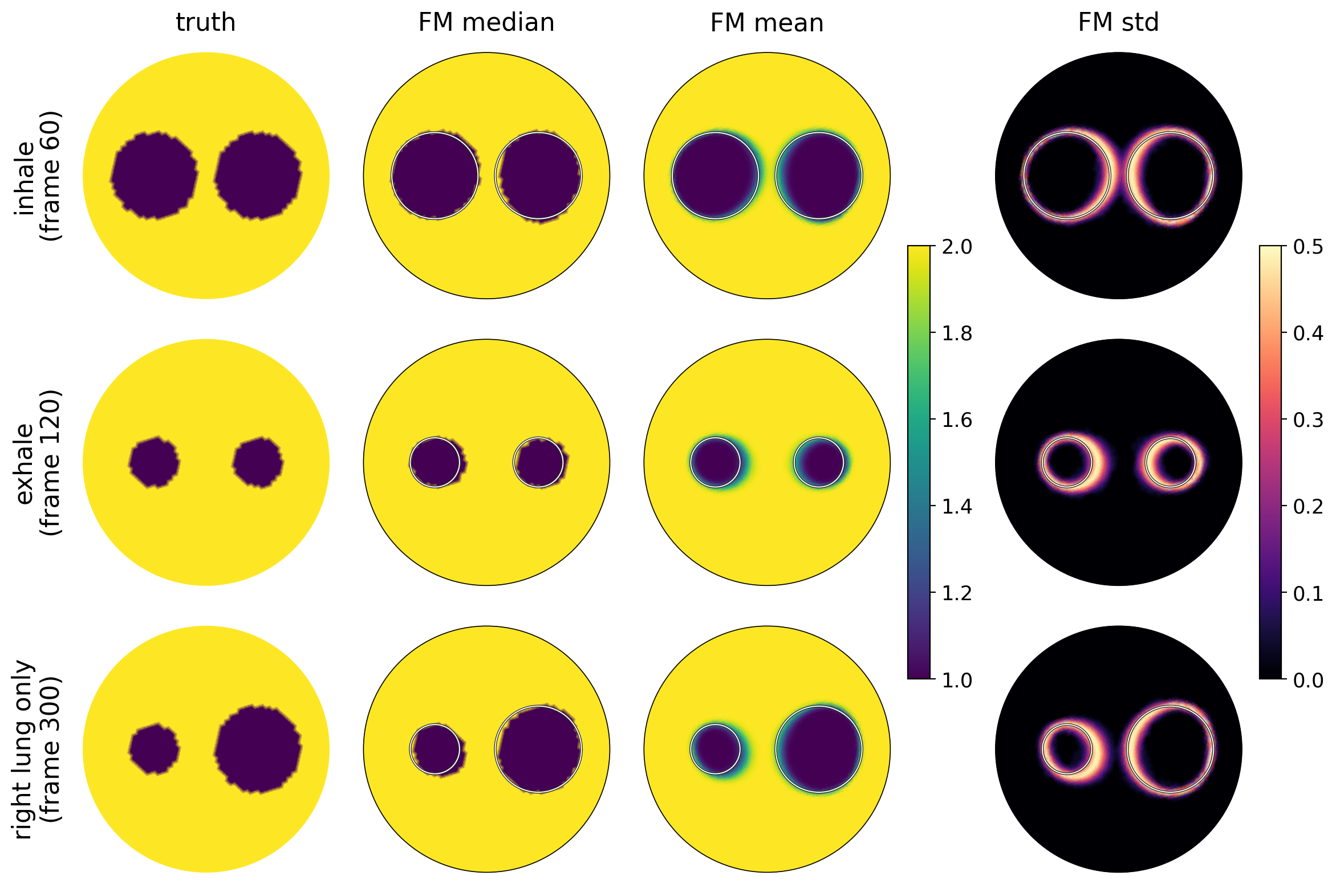}
    \caption{Star-shaped flow on the breathing sequence at three of its $720$ frames: both inclusions inhaled (frame $60$), both exhaled (frame $120$), and the right one inhaled after the left has stopped breathing (frame $300$). Columns: truth, posterior median, posterior mean and posterior standard deviation; $1\,000$ samples per frame, RK4 with $10$ steps, $21.7$\,ms per frame on one RTX 4070 Ti SUPER. White circles mark the true inclusions.}
    \label{fig:eit-dynamic-lung-stop}
\end{figure}

MCMC needs millions of sequential forward solves for every single observation.
All EIT flows, in contrast, are trained only once on $2\times10^6$ samples, with a further $2\times10^5$ for validation; the $2.2\times10^6$ forward solves required to generate these samples are independent and can be run in parallel.
Every subsequent frame is then processed by ODE integration alone, without any forward solve.
\subsection{Stochastic Lotka--Volterra} \label{subsec:lotka_volterra}

We consider the stochastic Lotka--Volterra model introduced in \cite{wilkinson2012stochastic}, which \cite{wang2023efficient} use as one of their computational benchmarks; we follow their setup.
The model consists of a continuous-time Markov jump process tracking interacting predator ($S_1$) and prey ($S_2$) populations.
The state vector $S(t) = (S_1(t), S_2(t))^\top$, starting from a fixed initial condition $S(0) = (50, 100)^\top$, evolves through four discrete transition events defined by the jump vectors $\vzeta_1 = (1, 0)^\top$, $\vzeta_2 = (-1, 0)^\top$, $\vzeta_3 = (0, 1)^\top$, and $\vzeta_4 = (0, -1)^\top$.
The rates of these transitions are governed by the propensity vector $\lambda(S, \vx)$, parametrized by $\vx \in \R^4$:
\begin{equation} \label{eq:lv_propensity}
  \lambda(S, \vx) = (x_1 S_1 S_2,\; x_2 S_1,\; x_3 S_2,\; x_4 S_1 S_2)^\top .
\end{equation}
Under stochastic mass-action kinetics, each propensity component is the product of a stochastic rate constant $x_j$ and a combinatorial function counting the distinct possible reactant interactions (e.g., $S_1S_2$ for binary predation, or $S_i$ for unary events).
The state trajectory $S(t)$ is not governed by a deterministic differential equation but is described exactly by the random time-change (Kurtz) representation
\begin{equation} \label{eq:lv_kurtz}
  S(t) = S(0) + \sum_{j=1}^{4} \vzeta_j\, N_j \left( \int_0^t \lambda_j(S(\tau), \vx) \,\mathrm{d}\tau \right),
\end{equation}
where $N_1, \dots, N_4$ are independent unit-rate Poisson processes.
Exact state trajectories generated by this system are simulated using Gillespie's algorithm, from which an observation vector $\vy \in \R^{\nobs}$, $\nobs = 9$, is constructed using time-series summary statistics (means, log-variances, and temporal auto/cross-correlations).
In a Bayesian context, the exact likelihood $\tdens(\vy \mid \vx)$ is analytically intractable because evaluating it requires marginalizing over the high-dimensional, combinatorial latent space of all possible random jump sequences and event times that could produce $\vy$.
The likelihood is therefore unavailable in closed form, which calls for likelihood-free methods such as approximate Bayesian computation (ABC).
Flow matching is applicable without modification, since its training requires only samples $(\vx, \vy) \sim \tdist$, which the simulator provides.
For $\vx$ we use the same prior as \cite{wang2023efficient}, where further details on the setup can be found.

\textcite{wang2023efficient} compare a CCNF, trained by minimizing the expected negative log-likelihood (NLL) \eqref{eq:nll_loss}, i.e., the KL divergence (see \cref{subsec:bfm}), against an ABC method.
According to the authors, their CCNF model performs comparably to the asymptotically exact Sequential Monte Carlo ABC (SMC-ABC) algorithm, while achieving a speedup factor of 30.

As discussed in \cref{subsec:bfm}, and motivated in \cite{LipmanEtAl2023, liu2022flow} for the unconditional case, flow matching speeds up the training of continuous normalizing flows significantly, as it is simulation-free and requires no ODE solves during training. In our experiments, this decreases the training time from \num{25.5}~minutes, using the exact setup\footnote{\url{https://github.com/EmoryMLIP/COT-Flow}} of \cite{wang2023efficient}, to approximately \num{90}~seconds, using Bayesian flow matching (\cref{subsec:bfm}) and our own architecture for the velocity field. Both training times are measured on our own hardware, an NVIDIA GeForce RTX~4070~Ti~SUPER, and the former is in line with the training time reported in \cite{wang2023efficient}. Both training sets are generated with the data generation script of \cite{wang2023efficient}, the larger one by increasing the number of samples to \num{500000}.
For the velocity field $v_\theta(t, \vx; \vy)$ we employ an MLP with six hidden layers, swish activation and a hidden dimension of $\nhidden=112$.

As the stochastic Lotka--Volterra problem does not have a tractable likelihood, we cannot compare against likelihood-based MCMC methods, and in particular cannot compute the error measures derived in \cref{subsec:diagnostics}.
Our focus is a qualitative comparison against the results provided by \cite{wang2023efficient} with the same amount of training data, i.e., \num{50000} and \num{500000} samples, respectively.
Our results are shown in \cref{fig:lv-large}.

In particular, when comparing our model trained on the large training set, displayed in \cref{fig:lv-large}, against the SMC-ABC result from \cite[Figures 2 and 3]{wang2023efficient}, the agreement is good.
Comparing against COT-Flow, which is the CCNF trained with the NLL/KL-divergence objective in \cite{wang2023efficient}, our model matches the asymptotically exact SMC-ABC more closely.

Using the smaller training set, the model does not learn the posterior distribution sufficiently well, as can be seen, e.g., by comparing the $x_3$-marginal for the second set of observations.
This, however, can also be observed in \cite[Figure 3]{wang2023efficient} and is due to the low number of training samples.

\begin{figure}[H]
    \centering
    \begin{subfigure}{.48\textwidth}
\includegraphics[width=0.9\linewidth]{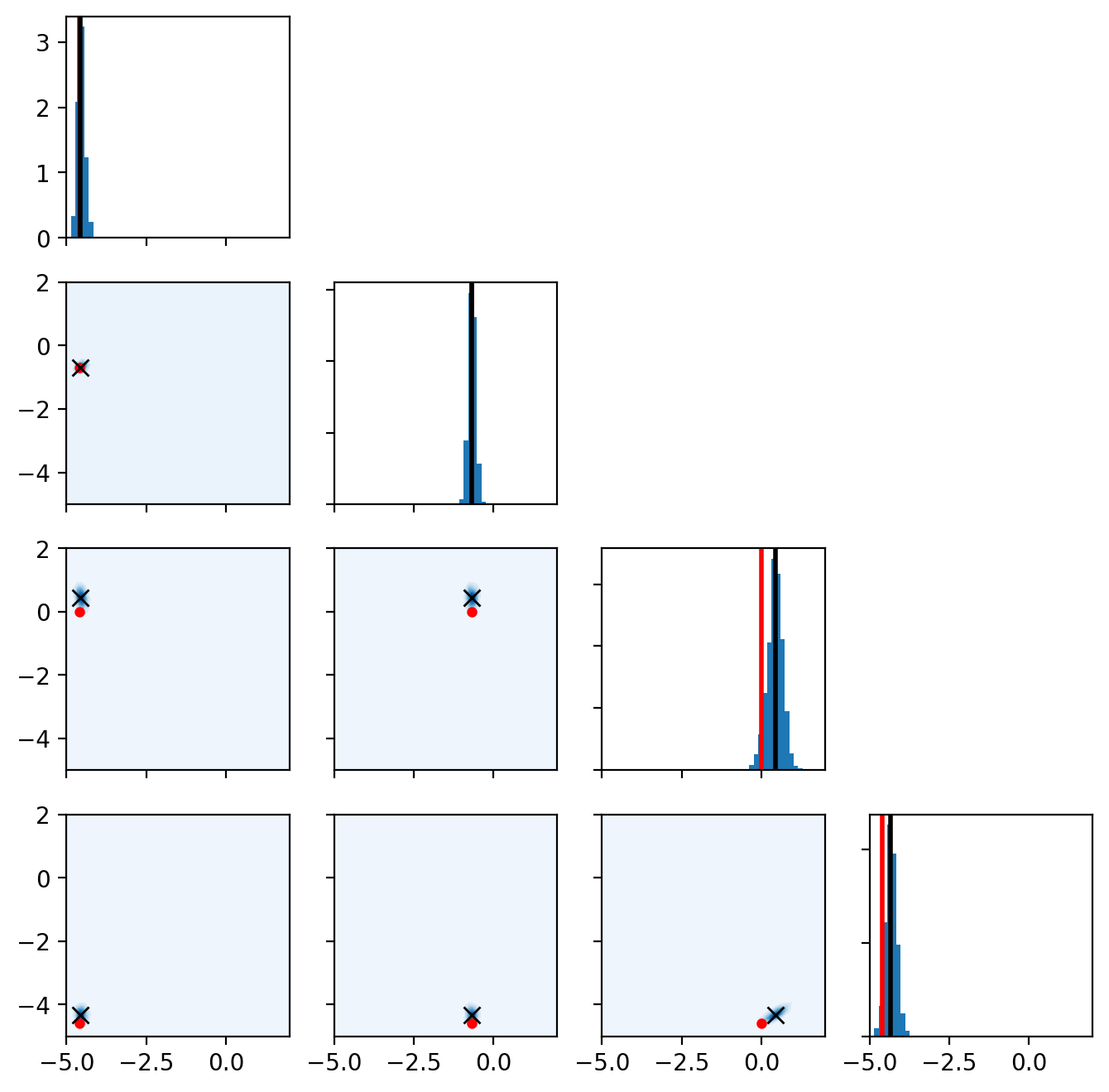}
\end{subfigure}\hfill%
    \begin{subfigure}{.48\textwidth}
       \includegraphics[width=0.9\linewidth]{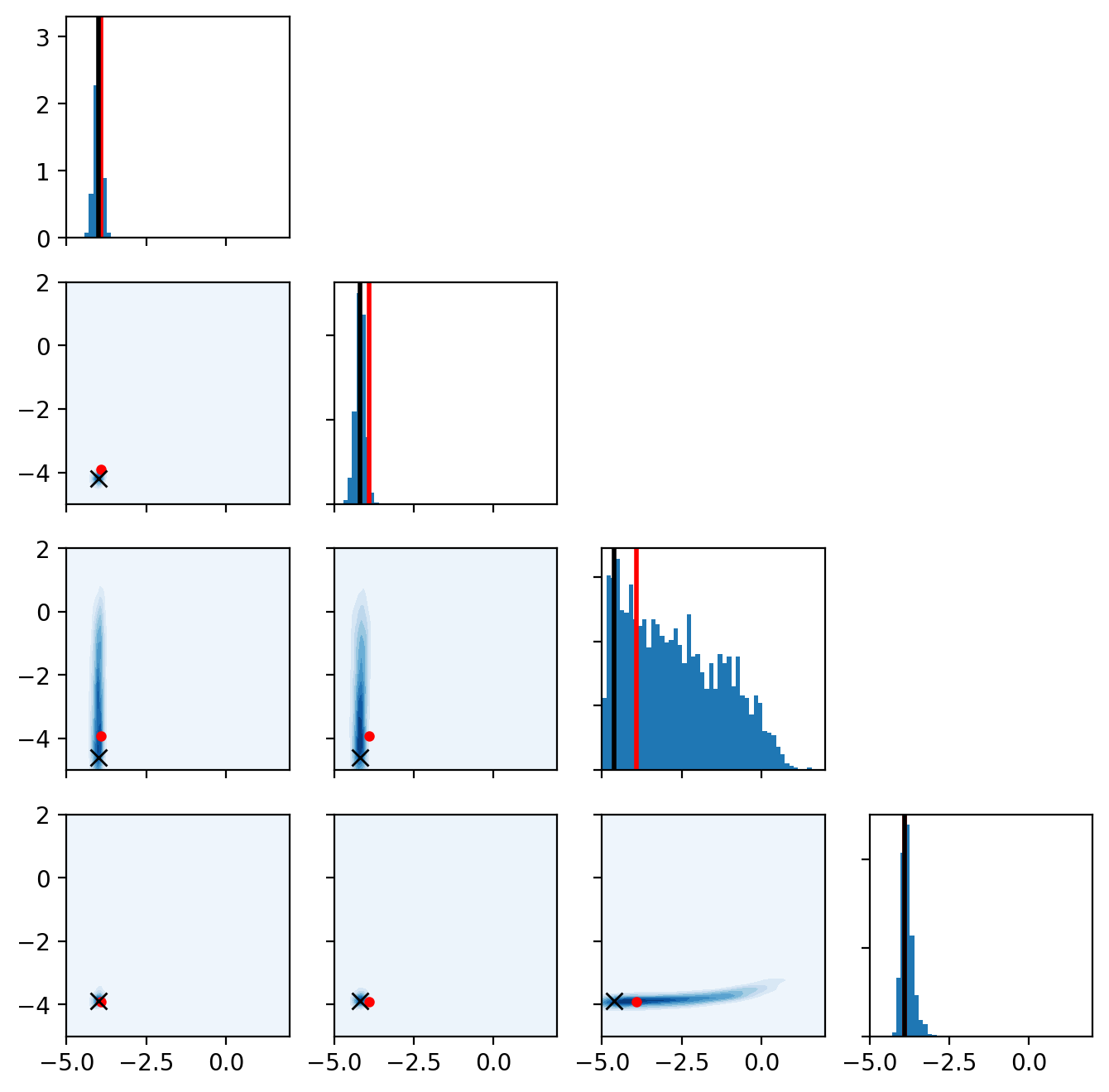}
    \end{subfigure}\\
    \begin{subfigure}{.48\textwidth}
\includegraphics[width=0.9\linewidth]{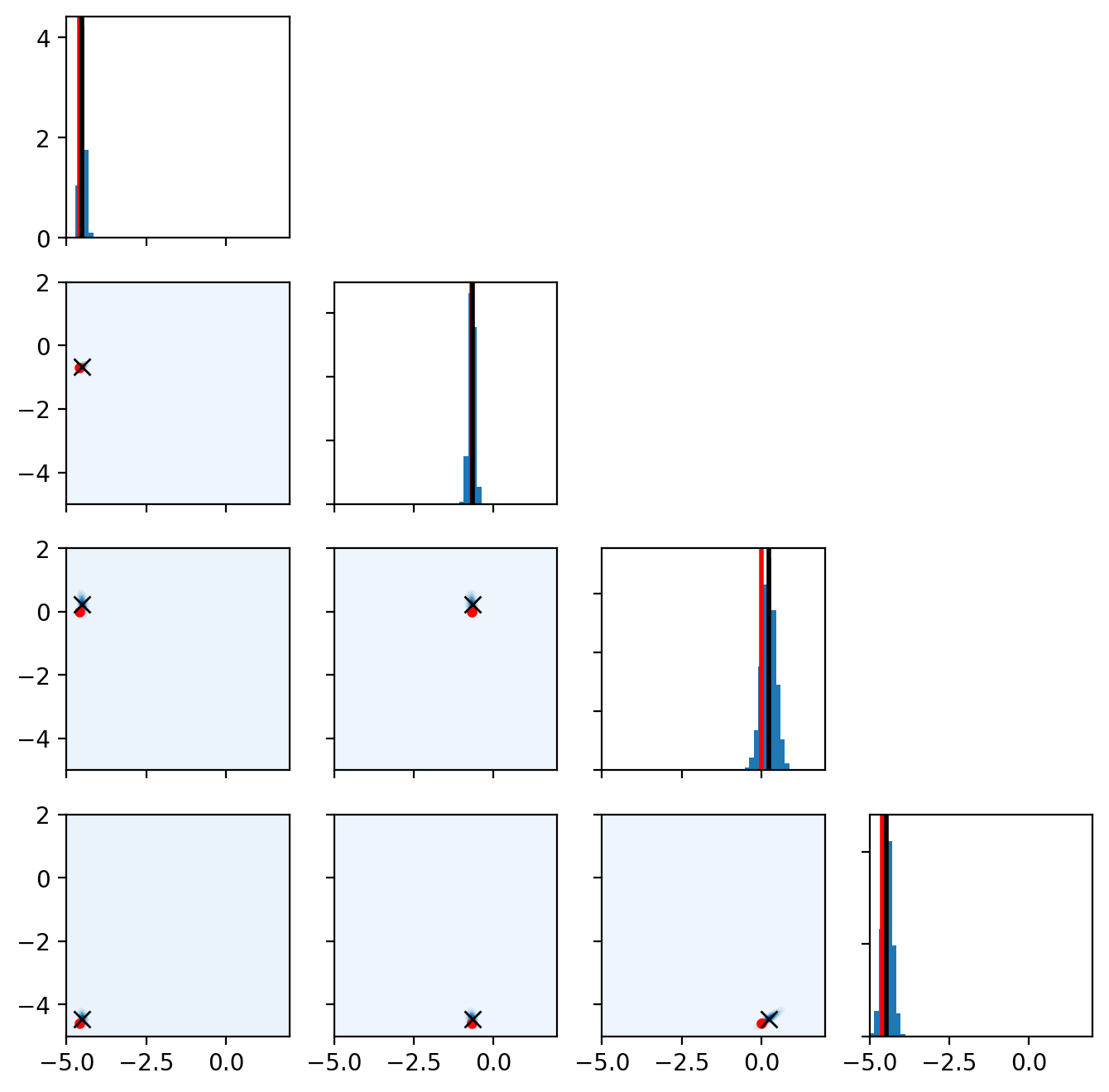}
\end{subfigure}\hfill%
    \begin{subfigure}{.48\textwidth}
       \includegraphics[width=0.9\linewidth]{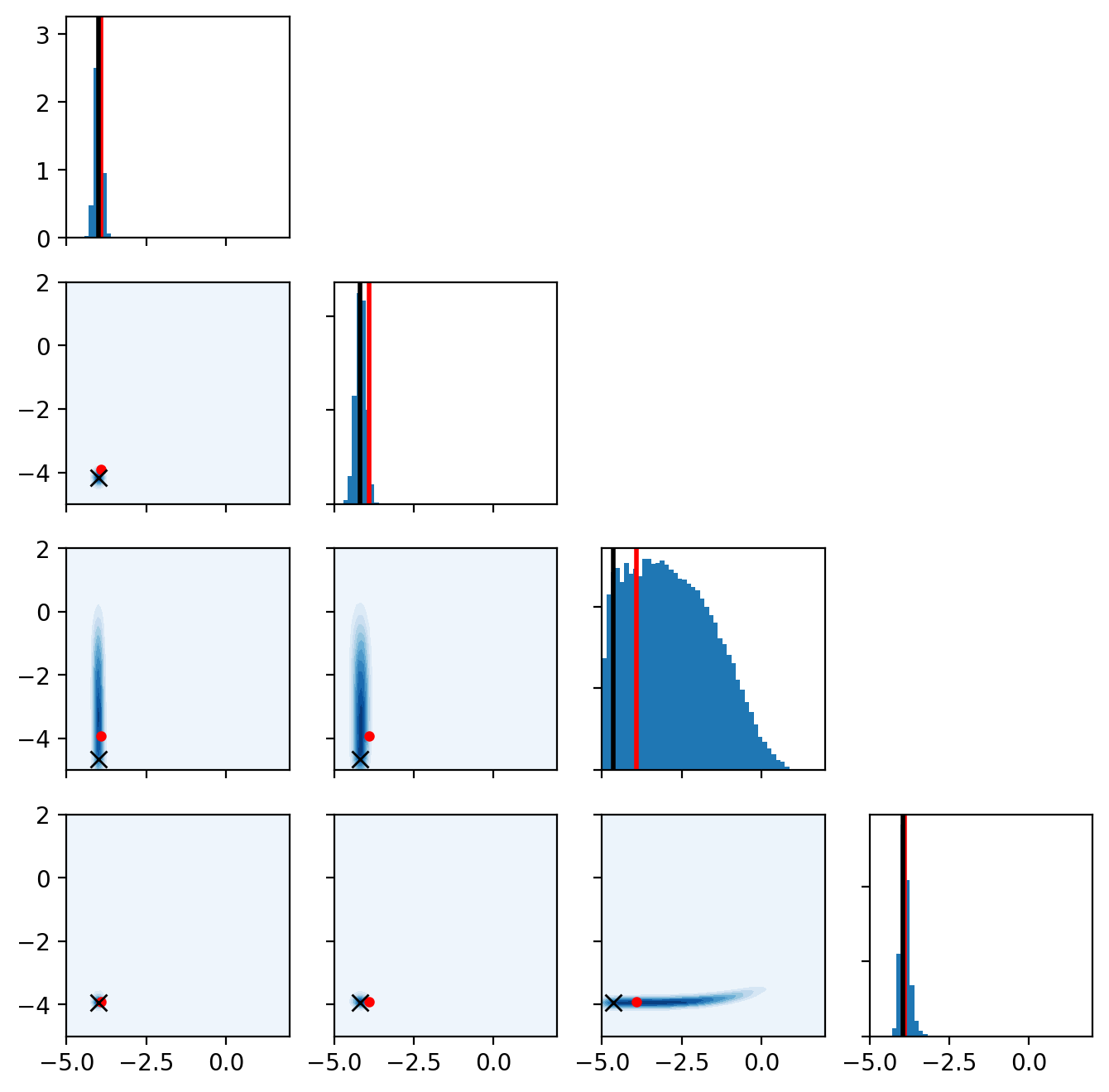}
    \end{subfigure}
    \caption{Flow matching posterior for two different observations, following the setup in \cite{wang2023efficient}. The model is trained using flow matching and a training set of size \num{50000} (top) and  \num{500000} (bottom). The red point/bars represent the true parameter $\vx^\star =(0.01, 0.5, 1, 0.01)^\top$ (left) and $\vx^\star =(0.02, 0.02, 0.02, 0.02)^\top$ (right), the black crosses/bars denote the mode of the learned density $\fmdens^{\vy}$, as in \cite{wang2023efficient}.}
    \label{fig:lv-large}
\end{figure}
\section{Discussion} \label{sec:discussion}

In this paper we have discussed and evaluated Bayesian flow matching for approximate sampling from posteriors occurring in PDE-based Bayesian inverse problems.
We have derived computable estimates of lower and upper bounds on the TV distance and KL divergence between true posterior and flow matching approximation which are based mainly on the acceptance rate of a hybrid sampler --- an IMH using the flow matching approximation as proposal.
The validity and accuracy of these bounds have been verified for the example of Section \ref{sec:toy}, where reference values can be computed by quadrature.

Moreover, we have compared flow matching against well-established MCMC methods in various settings in Sections \ref{sec:infinite_heat_example} and \ref{subsec:eit}.
Across these examples, whenever the true observations were covered by the spread of observations within the training data, flow matching agreed closely with MCMC-generated samples in terms of posterior means, standard deviations and quantiles, at a fraction of the forward solves required by MCMC.
However, the fit in TV distance can sometimes be  disappointing: for instance, in the setting of \cref{fig:dram_vs_fm_1}, flow matching reproduces the posterior mean and the $95\%$ band of DRAM samples to about $1\%$ in relative $L^2$-norm, yet FM-IMH accepts only 5\% to 11\% of its proposals (\cref{tab:fm_imh_acceptance}).
The explanation for these seemingly contradictory findings is that flow matching can fit the marginals of the highly informed parameter directions rather poorly, which, due to their comparatively small posterior variance, do not significantly contribute to posterior mean and spread, but differences in the marginals are strongly penalized in TV distance.
In general, we noticed in our experiments that flow matching may suffer from many or highly accurate observations.

We have proposed two remedies that account for the very different posterior spread across coordinates: weighting the loss by the inverse of the linearized posterior standard deviation, and adding to the velocity field an affine term that reproduces a Gaussian approximation of the posterior.

Another limitation of flow matching found in our studies is the loss of accuracy when the observation lies in the tail of the prior predictive distribution, a phenomenon known in machine learning as out-of-distribution error.
In these cases MCMC is to be preferred.
We note that, in principle, it can be checked whether the observations $y$ at hand are sufficiently covered by the prior predictive distribution by comparing the normalized radius $r(y)$ of \cref{eq:yradius} with that of the training observations.

However, our experiments also confirmed two important advantages of Bayesian flow matching over MCMC: amortization, demonstrated for EIT in Section~\ref{subsec:eit}, and likelihood-free inference, demonstrated for a stochastic Lotka--Volterra model in Section~\ref{subsec:lotka_volterra}.
In both problems MCMC would have been either infeasible (\Cref{subsec:lotka_volterra}) or too slow for a stream of observations (\Cref{subsec:eit}).
Moreover, our hybrid method --- IMH with flow matching proposal --- is another promising method combining, in a way, the best of both worlds: due to Metropolization, asymptotic exactness is restored, and since the forward solves for the proposals are independent and can be carried out in parallel, it can still be significantly faster than a single run of a classical MCMC method.
In our experiments the IMH chain mixed well whenever the observation was typical under the prior (Tables~\ref{tab:mcmc_ess} and \ref{tab:eit-imh-all}), but poorly for observations in the tails of the prior predictive distribution.

A theoretical foundation for flow matching-based posterior approximation in inverse problems is still largely missing.
The distances used here to judge a learned posterior, the total variation distance bounded by the acceptance rate and the Kullback--Leibler divergence, are common choices in finite dimensions but ill-suited to high- and infinite-dimensional settings: two measures on a function space are easily mutually singular, which makes both distances degenerate even when the measures are close in every observable sense.
Further directions therefore include an error analysis of flow matching in function space based on the Wasserstein distance, which does not have this defect;
a systematic comparison of the total cost of both approaches at matched computational budgets, including the forward solves for the training data and the training of the network, which we have not attempted since its outcome depends strongly on the available hardware;
and, on the practical side, addressing missing training support by combining the flow with adaptive error models \cite{kaipiosomersalo2005,CalvettiDunlopSomersaloStuart2018,CuiEtAl2019}: just as these replace error statistics computed over the prior by statistics updated where the posterior concentrates, the flow could be refined iteratively with training samples drawn near the observations actually encountered, trading part of the amortization for accuracy.

\ifiop
\data{\datastatement}
\else
\section*{Data availability}
\datastatement
\fi

\printbibliography[title=References]

\appendix
\crefalias{section}{appendix}

\section{The Metropolis--Hastings algorithm}
\label{app:mh_alg}
\Cref{alg:mh} states the generic Metropolis--Hastings iteration of \cref{subsec:mh}.
The samplers used in this paper differ only in the proposal density $q$: random walk Metropolis, pCN and DRAM propose locally around the current state (\cref{subsec:mcmc_variants}), whereas IMH draws its proposals independently of it, here from the flow (\cref{subsubsec:imh}).

\begin{algorithm}[ht]
\caption{Generic Metropolis--Hastings MCMC}
\label{alg:mh}
\begin{algorithmic}[1]
\Require unnormalized target density $\tdens^y$, proposal density
  $q(\cdot \mid \cdot)$, initial state $x^{(0)}$, chain length $N$
\For{$k = 1, \dots, N$}
  \State sample candidate $\tilde{x} \sim q(\,\cdot \mid x^{(k-1)})$
  \State compute the acceptance probability
    $\displaystyle
     \alpha = \min\Bigl\{ 1,\;
       \frac{\tdens^y(\tilde{x})}{\tdens^y(x^{(k-1)})}\,
       \frac{q(x^{(k-1)} \mid \tilde{x})}{q(\tilde{x} \mid x^{(k-1)})}
     \Bigr\}$
  \State sample $u \sim \mathcal{U}[0, 1]$
  \If{$u \le \alpha$}
    \State $x^{(k)} \gets \tilde{x}$ \Comment{accept}
  \Else
    \State $x^{(k)} \gets x^{(k-1)}$ \Comment{reject}
  \EndIf
\EndFor
\Ensure states $x^{(1)}, \dots, x^{(N)}$
\end{algorithmic}
\end{algorithm}

\section{Proof of Proposition~\ref{prop:imh_tv}} \label{app:imh_tv_proof}

Write $f \coloneqq \tdens^y$ and $g \coloneqq q$ for the two densities. By \eqref{eq:mh_acceptance} with $q(\tilde x \mid x) = g(\tilde x)$, we have $f(x) g(\tilde x)\, \alpha(x, \tilde x) = \min\{ f(x) g(\tilde{x}),\, f(\tilde{x}) g(x) \}$. Using
$\min\{a, b\} = \tfrac{1}{2}(a + b) - \tfrac{1}{2}|a - b|$ together with
$\iint f(x) g(\tilde{x})\, \mathrm{d}x\, \mathrm{d}\tilde{x} = 1$ gives
\begin{equation*}
  \bar{\alpha}
  = \iint \min\bigl\{ f(x) g(\tilde{x}),\, f(\tilde{x}) g(x) \bigr\}
    \, \mathrm{d}x\, \mathrm{d}\tilde{x}
  = 1 - D,
  \qquad
  D \coloneqq \tfrac{1}{2} \iint
    \bigl| f(x) g(\tilde{x}) - f(\tilde{x}) g(x) \bigr|
    \, \mathrm{d}x\, \mathrm{d}\tilde{x}.
\end{equation*}
For the lower bound in \eqref{eq:imh_sandwich}, insert
$\pm\, g(x) g(\tilde{x})$ and apply the triangle inequality,
\begin{equation*}
  \bigl| f(x) g(\tilde{x}) - f(\tilde{x}) g(x) \bigr|
  \le g(\tilde{x})\, |f(x) - g(x)| + g(x)\, |f(\tilde{x}) - g(\tilde{x})|,
\end{equation*}
and integrate, using $\|f - g\|_{L^1} = 2\, d_{\mathrm{TV}}(q, \tdist^y)$, to obtain
$D \le 2\, d_{\mathrm{TV}}(q, \tdist^y)$. For the upper bound, set
$A \coloneqq \{ f > g \}$ and
$\varphi(x, \tilde{x}) \coloneqq
 \mathbf{1}_A(\tilde{x}) - \mathbf{1}_A(x) \in [-1, 1]$. A direct
computation gives
\begin{equation*}
  \iint \bigl( f(x) g(\tilde{x}) - f(\tilde{x}) g(x) \bigr)\,
    \varphi(x, \tilde{x})\, \mathrm{d}x\, \mathrm{d}\tilde{x}
  = 2 \Bigl( \int_A g\, \mathrm{d}x - \int_A f \,\mathrm{d}x \Bigr)
  = -2\, d_{\mathrm{TV}}(q, \tdist^y),
\end{equation*}
since $\int_A (f - g)\, \mathrm{d}x = d_{\mathrm{TV}}(q, \tdist^y)$ for this choice of $A$. As
$|\varphi| \le 1$, the left-hand side is bounded in modulus by $2D$, so
$D \ge d_{\mathrm{TV}}(q, \tdist^y)$, which completes the proof. \hfill$\square$

\section{Flow matching in practice} \label{app:fm_alg}

\Cref{alg:fm_data,alg:fm_train,alg:fm_sample} state the three stages used
throughout \cref{sec:experiments}: generating the training set from the prior and
the forward map, fitting the velocity field, and drawing samples.

\begin{algorithm}[H]
\caption{Training-set generation}
\label{alg:fm_data}
\begin{algorithmic}[1]
\Require prior $\tdist_X$, forward map $F$, noise covariance $\Gamma$, size $N$
\For{$j = 1, \dots, N$}
  \State $x^j \sim \tdist_X$
  \State $\varepsilon^j \sim \NN(0, \Gamma)$
  \State $y^j \gets F(x^j) + \varepsilon^j$
\EndFor
\Ensure pairs $\{(x^j, y^j)\}_{j=1}^N$ from the joint distribution
\end{algorithmic}
\end{algorithm}

\begin{algorithm}[H]
\caption{Flow matching training}
\label{alg:fm_train}
\begin{algorithmic}[1]
\Require pairs $\{(x^j, y^j)\}_{j=1}^N$ from \cref{alg:fm_data}, initial
  parameters $\theta_0$ of the velocity field $v_\theta$, batch size $B$
\State $\theta \gets \theta_0$
\While{not converged}
  \State draw a minibatch of $B$ pairs, $(X_1^\ell, Y^\ell)$ for $\ell = 1, \dots, B$
  \State $X_0^\ell \sim \NN(0, I_d)$,
    \quad $t^\ell \sim \UU[0,1]$,
    \quad $X_t^\ell \gets (1 - t^\ell) X_0^\ell +  t^\ell X_1^\ell
      \vphantom{\smash[b]{\Big(}}$
  \State $\displaystyle
    \mathcal L(\theta) \gets \frac{1}{B} \sum_{\ell=1}^B
      \bigl\| v_{\theta}\bigl(t^\ell, X_t^\ell; Y^\ell\bigr)
             - \bigl(X_1^\ell - X_0^\ell\bigr) \bigr\|_2^2$
  \State $\theta \gets \textsc{AdamStep}(\theta, \nabla_\theta \mathcal L)$
\EndWhile
\Ensure $\theta$ approximately minimizing \eqref{eq:fm_loss}, i.e.,
  $\min_\theta \ev{\bigl\| v_\theta(t, X_t; Y) - (X_1 - X_0) \bigr\|_2^2}$
\end{algorithmic}
\end{algorithm}

\begin{algorithm}[H]
\caption{Approximate posterior sampling}
\label{alg:fm_sample}
\begin{algorithmic}[1]
\Require trained $v_\theta$, observation $y$, number of samples $n$
\For{$j = 1, \dots, n$} \Comment{independent; done in one batch in practice}
  \State $X_0^j \sim \NN(0, I_d)$
  \State $X_1^j \gets$ solve
    $\;\odv{}{t} X(t) = v_\theta\bigl(t, X(t); y\bigr)$, $t \in (0,1)$,
    from $X(0) = X_0^j$
\EndFor
\Ensure approximate posterior samples $\{X_1^j\}_{j=1}^n$
\end{algorithmic}
\end{algorithm}

\section{Marginals of the leading Karhunen--Loève coefficients} \label{app:heat-marginals}

\Cref{fig:heat_marginals} complements \cref{fig:dram_vs_fm_1} by showing the posterior in the latent coefficients $\xi_1, \dots, \xi_4$ instead of the field, for the MLP and for the combination of \texttt{gauss} and \texttt{weights} of \cref{subsubsec:visual}, both with $\nhidden = 800$.
The widely spread marginals of $\xi_1$ and $\xi_3$ are reproduced by both flows, whereas for the strongly constrained coefficients $\xi_2$ and $\xi_4$ the MLP is too wide or shifted.
Apart from $x^\star_1$, the combination of both modifications matches DRAM in all four coefficients.

\begin{figure}[htbp]
    \centering
    \includegraphics[width=\linewidth]{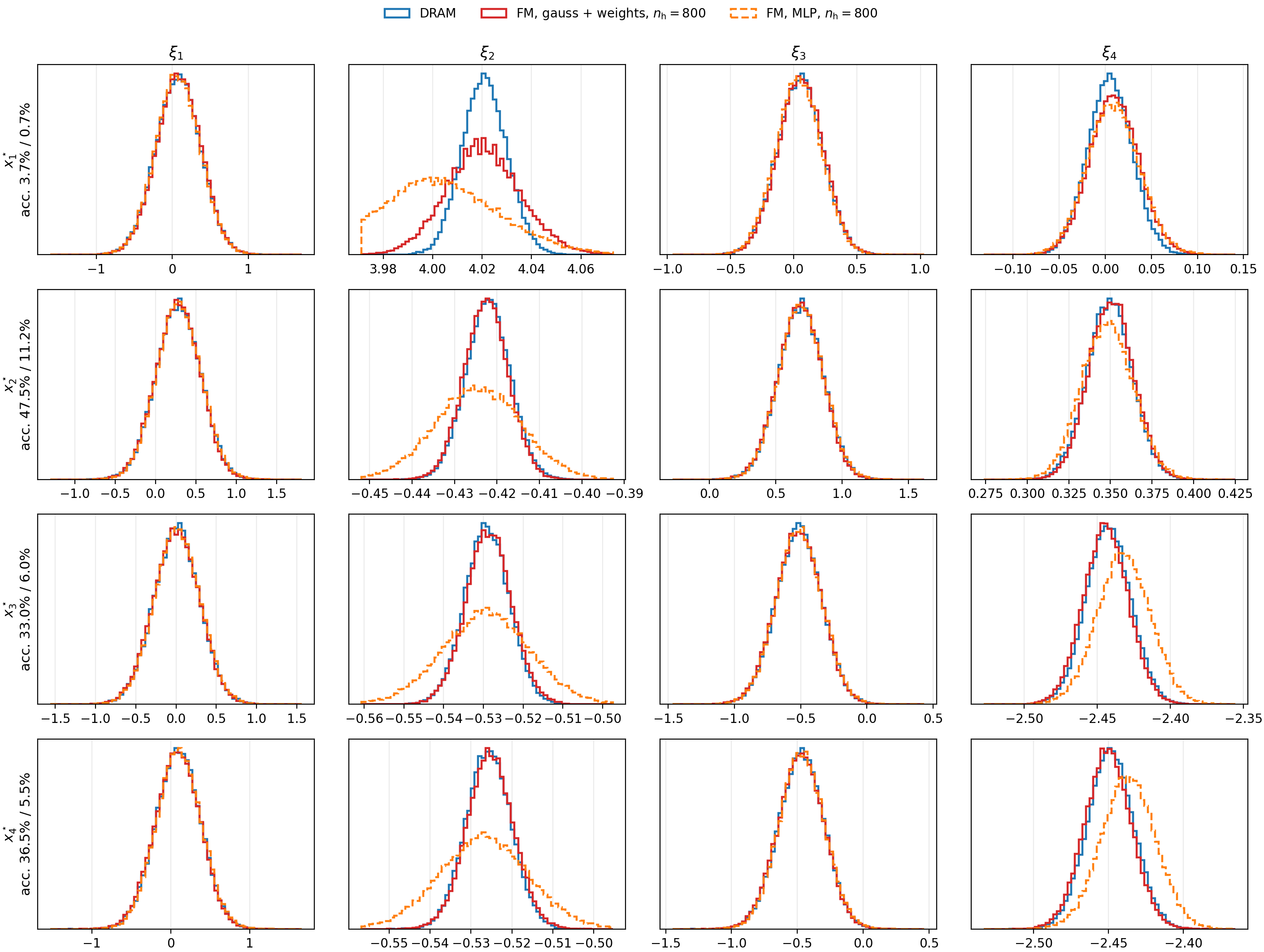}
    \caption{Marginal posterior densities of the Karhunen--Loève coefficients $\xi_1, \dots, \xi_4$ (columns) for the four ground truths $x^\star_1, \dots, x^\star_4$ (rows) in the setting of \cref{fig:dram_vs_fm_1} ($\nkl = 128$, $\nobs = 256$, $\gamma^2 = 10^{-5}$): DRAM (blue), FM with \texttt{gauss} and \texttt{weights} (red) and the MLP (orange, dashed), both with $\nhidden = 800$.
    Histograms of \num{40000} samples each; the row labels give the FM-IMH acceptance rates of the two flows from \cref{tab:fm_imh_acceptance}.}
    \label{fig:heat_marginals}
\end{figure}

\section{Draws from the EIT priors} \label{app:eit-prior-draws}
\Cref{fig:eit-prior-draws} shows four draws from each of the three priors of \cref{sec:eit-priors}, one row per prior; each column uses the same latent seed for all three.
The log-normal draws are smooth, low-contrast fluctuations around $e^{m_0} = \sqrt2$ without any preferred geometry; over the four draws $\sigma$ ranges from $0.85$ to $2.33$.
The level-set draws are two-valued, $\sigma \in \{1, 2\}$, and since $G$ is a centered Gaussian field every point lies in an inclusion with probability $\tfrac12$: a draw covers about half of the disk with inclusion (area fraction $0.500 \pm 0.046$), split into many irregular components, several of which touch the boundary.
The star-shaped draws, also two-valued, contain exactly two inclusions of radius about $0.3$ near the center of the disk.
Since the centers are drawn with standard deviation $0.3$ and no non-overlap constraint is imposed, the two inclusions frequently overlap, as in the first and third draws, and can extend beyond the disk, as in the second.
Of the three priors, only the star-shaped one produces fields that resemble the ground truth of \cref{sec:cem}, two separated disks in a constant background.

\begin{figure}[htbp]
  \centering
  \includegraphics[width=0.9\linewidth]{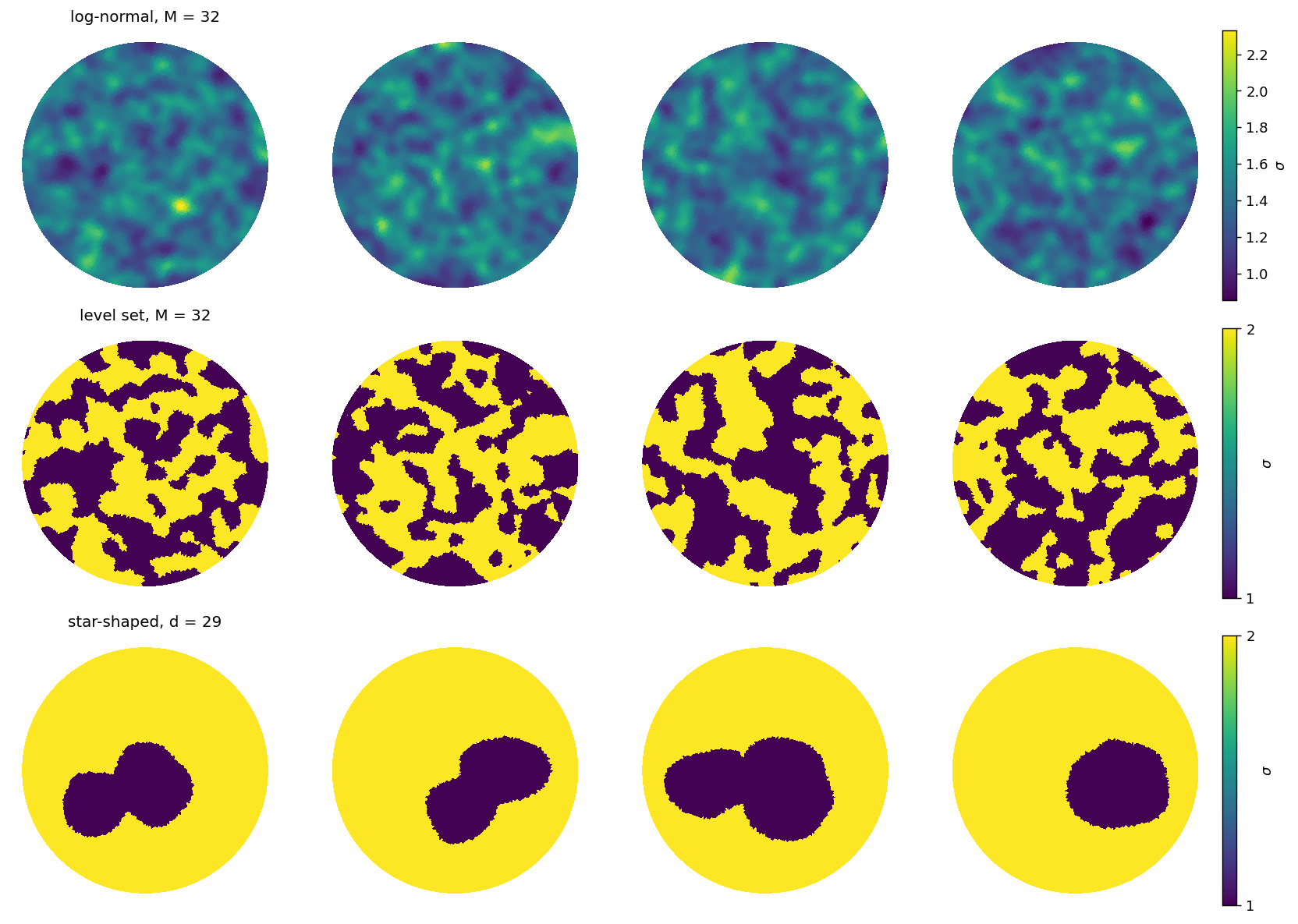}
  \caption{Four draws from each EIT prior on the inversion mesh, one row per prior and the same latent seed in each column.}
  \label{fig:eit-prior-draws}
\end{figure}

\section{Flow matching against MCMC on the canonical observation} \label{app:eit-vobs}

\Cref{fig:eit-vobs-starshaped,fig:eit-vobs-lognormal,fig:eit-vobs-levelset} extend \cref{fig:eit-vobs-grid} by the posterior standard deviations and the difference between the two posterior means, one figure per prior.
Each shows the truth and the reference and flow-matching posterior means in the top row, and the difference of the means and the two standard deviations in the bottom row; means and standard deviations each share one color scale.

\paragraph{Reference samplers}
For the star-shaped prior ($d = 26$) the reference is DRAM with $100$ chains of $2\times10^5$ steps, thinned by $200$, at an acceptance rate of $42.7\%$.
For the two priors with $d = 1024$ the reference is pCN.
Under the log-normal prior we run $4$ chains of $2\times10^6$ steps, thinned by $1000$, with $\beta = 0.045$ and an acceptance rate of $22.8\%$.
Under the level-set prior we run $100$ chains of $6\times10^5$ steps, thinned by $300$, with $\beta = 0.03$ and an acceptance rate of $19.8\%$; on the $5\,382$ elements whose value is uncertain, the effective sample size is at least $2\,039$.

\paragraph{Effective sample size}
\Cref{tab:eit-ref-ess} compares the effective sample size of these chains with that of IMH using the flow as proposal, for the smallest over the latent coordinates $\xi_i$ and for the mean conductivity over all elements, as in \cref{tab:eit-imh-all}.
The ESS of the reference is computed per chain on the stored, thinned draws, summed over the chains and normalized by the total number of samples after burn-in.
For pCN a sample costs one forward solve; the delayed-rejection stage of DRAM solves a second time after a first-stage rejection, so its cost per sample is higher.
Per forward solve, IMH is ahead under the star-shaped prior, by a factor of about $35$ in $\vxi$ and $11$ in the mean conductivity.
Under the log-normal and the level-set prior IMH makes only $19$ and $12$ moves, so its ESS there is a rough estimate.
Not included is the one-time cost of the flow's $2.2\times10^6$ training and validation samples.

\begin{table}[htbp]
\centering\footnotesize
\caption{Effective sample size of the MCMC reference chains on $\vy$, against IMH with $10\,000$ flow-matching proposals (\cref{tab:eit-imh-all}).}
\label{tab:eit-ref-ess}
\begin{tabular}{lccc}
\toprule
 & star-shaped & log-normal & level set \\
\midrule
sampler & DRAM & pCN & pCN \\
chains & $100$ & $4$ & $100$ \\
burn-in per chain & $2\times10^3$ & $2\times10^5$ & $4\times10^5$ \\
samples per chain & $2\times10^5$ & $2\times10^6$ & $6\times10^5$ \\
thinning & $200$ & $1000$ & $300$ \\
ESS ($\vxi$) & $4\,397$ & $641$ & $2\,781$ \\
ESS (mean $\sigma$) & $33\,830$ & $1\,741$ & $15\,961$ \\
ESS ($\vxi$) per $10^4$ samples & $2.2$ & $0.80$ & $0.46$ \\
ESS (mean $\sigma$) per $10^4$ samples & $16.9$ & $2.2$ & $2.7$ \\
\addlinespace
\multicolumn{4}{l}{\emph{IMH, $10^4$ flow-matching proposals}} \\
\quad ESS ($\vxi$) & $78$ & $3$ & $6$ \\
\quad ESS (mean $\sigma$) & $192$ & $4$ & $21$ \\
\bottomrule
\end{tabular}
\end{table}

\paragraph{Star-shaped}
In \cref{fig:eit-vobs-starshaped} both methods place all of the uncertainty in a thin annulus around each interface, the only place a two-valued posterior can be uncertain, and in both the annulus is wider on the side facing the center of the disk, farthest from the electrodes.
The difference of the two means is confined to the same annulus, in line with the agreement of the inclusion probabilities reported in \cref{sec:eit-results}.

\paragraph{Log-normal}
In \cref{fig:eit-vobs-lognormal} both posterior means show the two inclusions only as blurred depressions, and both place a band of high conductivity along parts of the boundary where the truth is constant; the two methods are far from the truth in the same way.
This reflects the mismatch between the smooth log-normal prior and the piecewise constant truth discussed in \cref{sec:eit-ood}, and is not an error of the flow.
The standard deviations are of similar size across the disk and lower inside the inclusions and at the electrodes.
The standard deviation of the flow is the smoother of the two, since the reference still carries Monte Carlo noise.

\paragraph{Level set}
In \cref{fig:eit-vobs-levelset} the flow locates both inclusions, and both methods leave the region between them uncertain.
The flow is over-dispersed, however: its standard deviation is larger throughout the background, and the difference of the means is positive inside the true inclusions and negative outside, i.e., the flow moves inclusion probability from the inclusions into the background.
This is the larger inclusion area, $0.222$ against $0.180$, reported in \cref{sec:eit-results}.

\begin{figure}[htbp]
    \centering
    \includegraphics[width=0.98\linewidth]{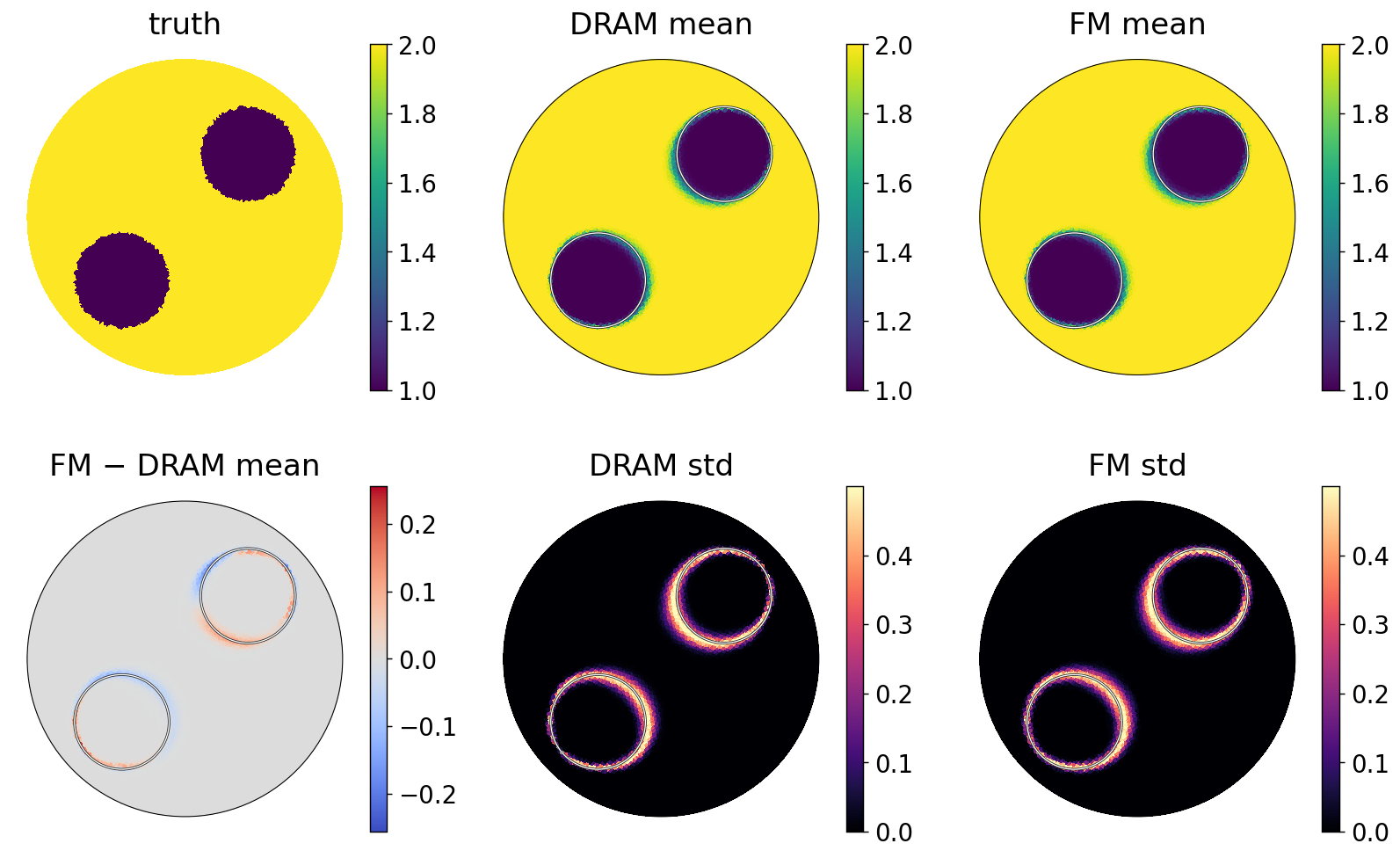}
    \caption{Star-shaped prior ($d=26$), conditioned on $\vy$, against DRAM.
    Top: truth and posterior means. Bottom: FM minus DRAM mean, and posterior standard deviations. Circles: true inclusions.}
    \label{fig:eit-vobs-starshaped}
\end{figure}

\begin{figure}[htbp]
    \centering
    \includegraphics[width=0.98\linewidth]{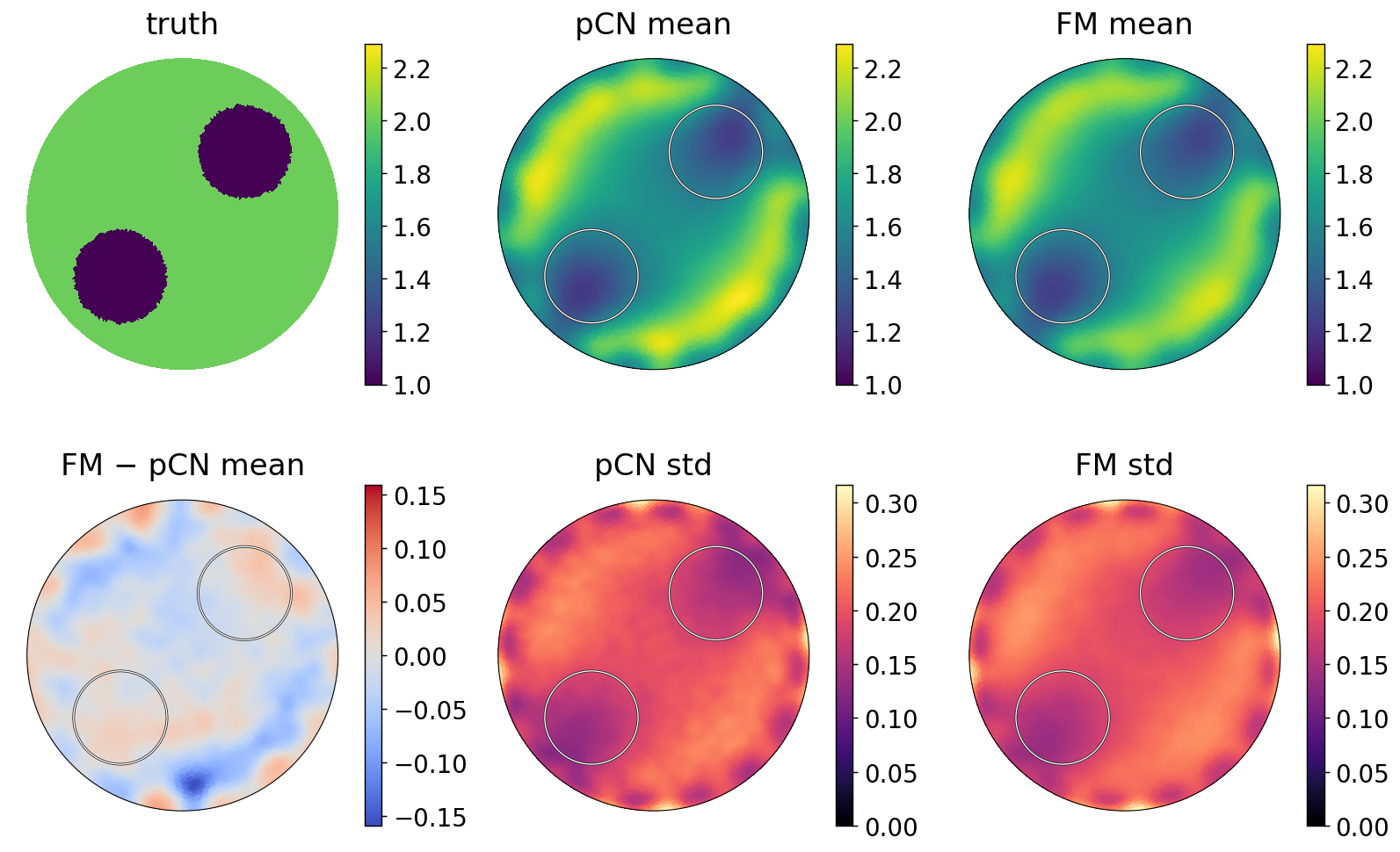}
    \caption{Log-normal prior ($M=32$), conditioned on $\vy$, against pCN; layout as in \cref{fig:eit-vobs-starshaped}.}
    \label{fig:eit-vobs-lognormal}
\end{figure}

\begin{figure}[htbp]
    \centering
    \includegraphics[width=0.98\linewidth]{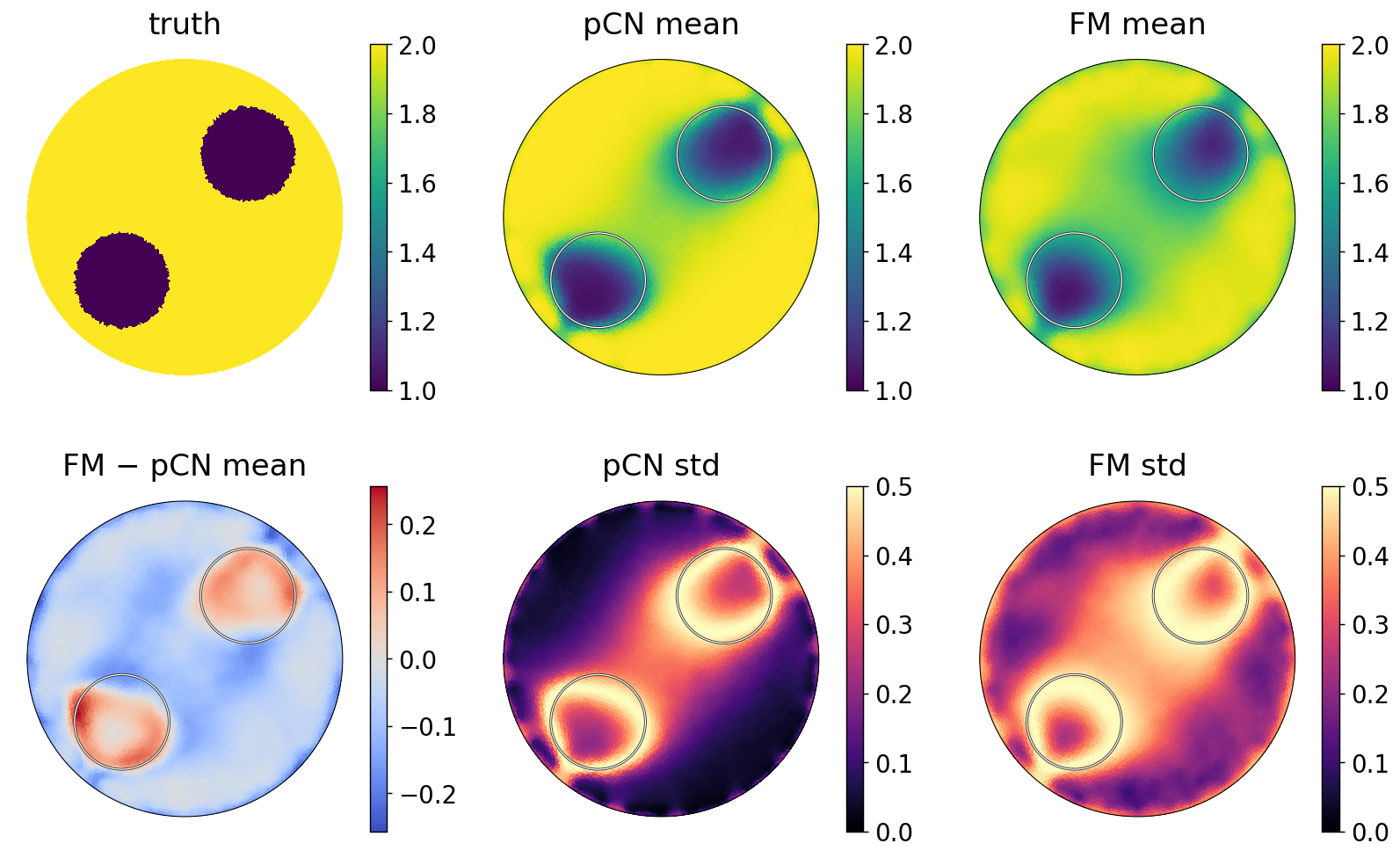}
    \caption{Level-set prior ($M=32$), conditioned on $\vy$, against pCN; layout as in \cref{fig:eit-vobs-starshaped}.}
    \label{fig:eit-vobs-levelset}
\end{figure}

\section{A prior draw as truth} \label{app:eit-priorobs}

\Cref{fig:eit-priorobs} repeats the comparison of \cref{app:eit-vobs} on an observation that is in distribution by construction.
For each prior we take one draw, which is not part of the training set, as ground truth and simulate its observation exactly as the training data were generated.

For the star-shaped and the log-normal prior the flow is accurate enough on this observation that IMH with the flow as proposal serves as the reference: with $10\,000$ proposals it accepts $52.8\%$ ($5\,279$ moves, ESS $1\,841$ in $\vxi$ and $2\,906$ in the mean conductivity) and $38.6\%$ ($3\,857$ moves, ESS $352$ and $1\,029$), respectively (\cref{tab:eit-imh-all}).
For the level-set prior the IMH chain visits only $9$ distinct states in $10\,000$ proposals, so its mean would show little more than those states and is not the posterior mean.
The reference there is pCN with $60$ chains of $1.8\times10^6$ steps each at $\beta = 0.01$ (acceptance $17.7\%$), which still yields an effective sample size of only about $360$ per element (median over the uncertain elements) and $285$ for the inclusion area.
The remaining roughness of that panel is its Monte Carlo error.

Under the star-shaped prior both means recover the two overlapping inclusions of the draw, blurred only at the interface.
Under the log-normal prior both means are nearly featureless: the fine-scale fluctuations of the draw lie in directions the $195$ measurements do not inform (\cref{sec:eit-priors}), and the posterior mean stays close to the prior mean except near the boundary.
Under the level-set prior both means resolve structure only close to the boundary, and flow and reference agree on where they place it; the interior, farthest from the electrodes, remains uncertain.

\begin{figure}[htbp]
    \centering
    \includegraphics[width=\linewidth]{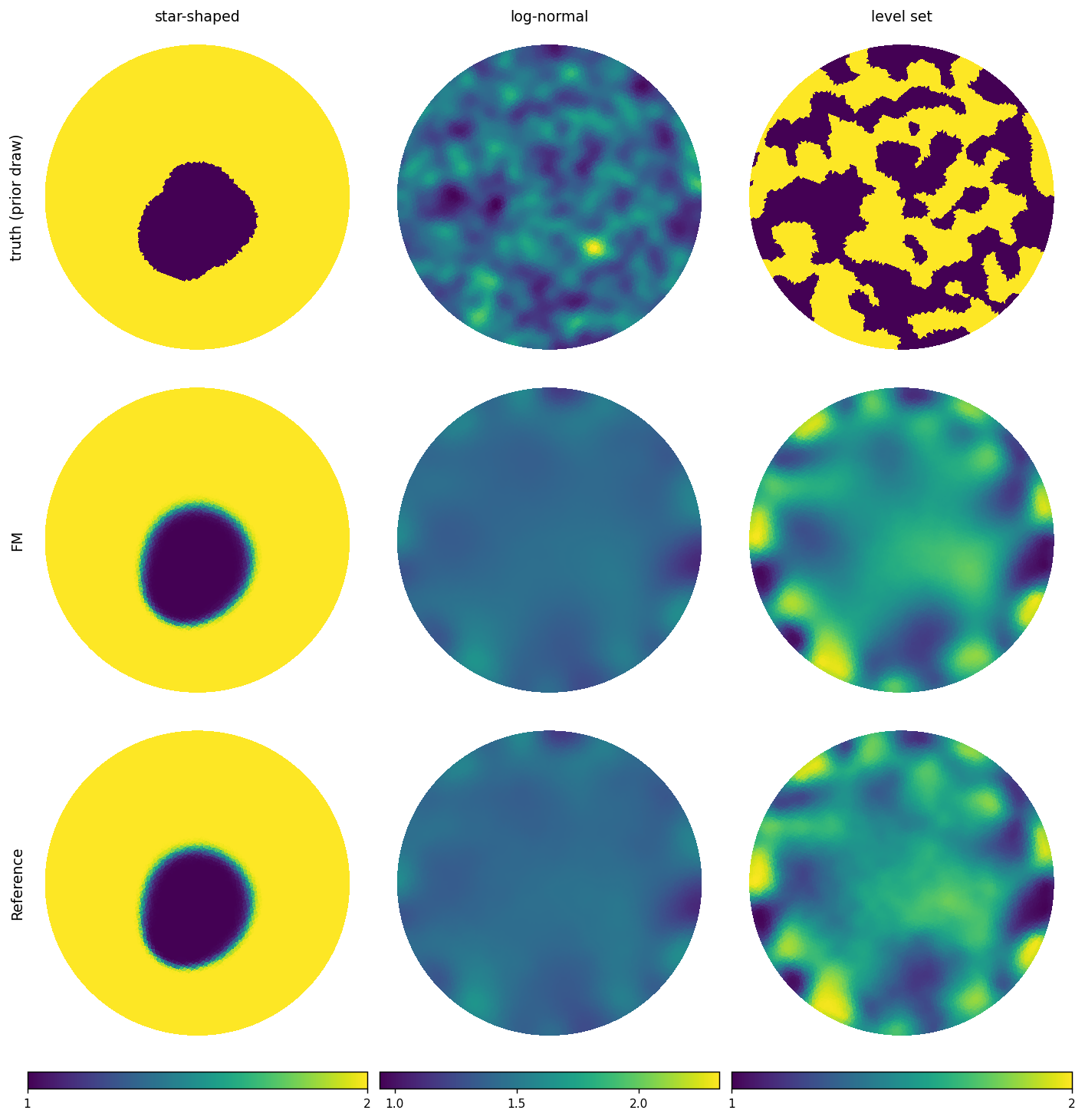}
    \caption{A prior draw as ground truth, one column per prior: truth (top), flow-matching posterior mean (middle) and reference posterior mean (bottom), with IMH as the reference for the star-shaped and log-normal priors and pCN for the level-set prior. Color scales are shared within each column.}
    \label{fig:eit-priorobs}
\end{figure}

\section{Moving and merging inclusions} \label{app:eit-orbit}
A second sequence, set up as in \cref{sec:eit-dynamic}, tests whether the flow also follows inclusions that move and merge.
Over $72$ frames two inclusions of radius $0.25$ first orbit the center of the disk once clockwise at radius $0.5$ ($48$ frames) and then move straight to the center ($24$ frames), where they overlap and finally coincide in a single inclusion, halving the inclusion area.
With $1\,000$ samples per frame, the flow follows every stage (\cref{fig:eit-dynamic-orbit}), including the merge, which this prior can only represent by placing its two inclusions on top of each other.
The posterior median deviates from the truth by an RMSE of $0.109$ on average, ranging from $0.061$ to $0.157$ over the frames.
Its inclusion area $\int_{\DD} \mathbf{1}_{\sigma(\vz) = 1}\,\mathrm d\vz$ differs from the exact area by $0.002$ on average and by at most $0.029$ in any frame, with a correlation of $0.961$ over the frames.
In the top row of \cref{fig:eit-dynamic-orbit} the two posterior means rotate with the true inclusions; in the bottom row they approach each other, cover the union of the two overlapping disks in frames $62$ and $67$, and form a single disk in the last frame.
Integrating with RK4 at $10$ or at $25$ steps gives the same posterior, with a per-frame RMS difference of at most $0.002$ in the mean, at $21.9$ against $34.2$\,ms per frame, i.e., $46$ against $29$ frames per second on one RTX 4070 Ti SUPER; the figure uses $10$ steps.

\begin{figure}[htbp]
    \centering
    \includegraphics[width=\linewidth]{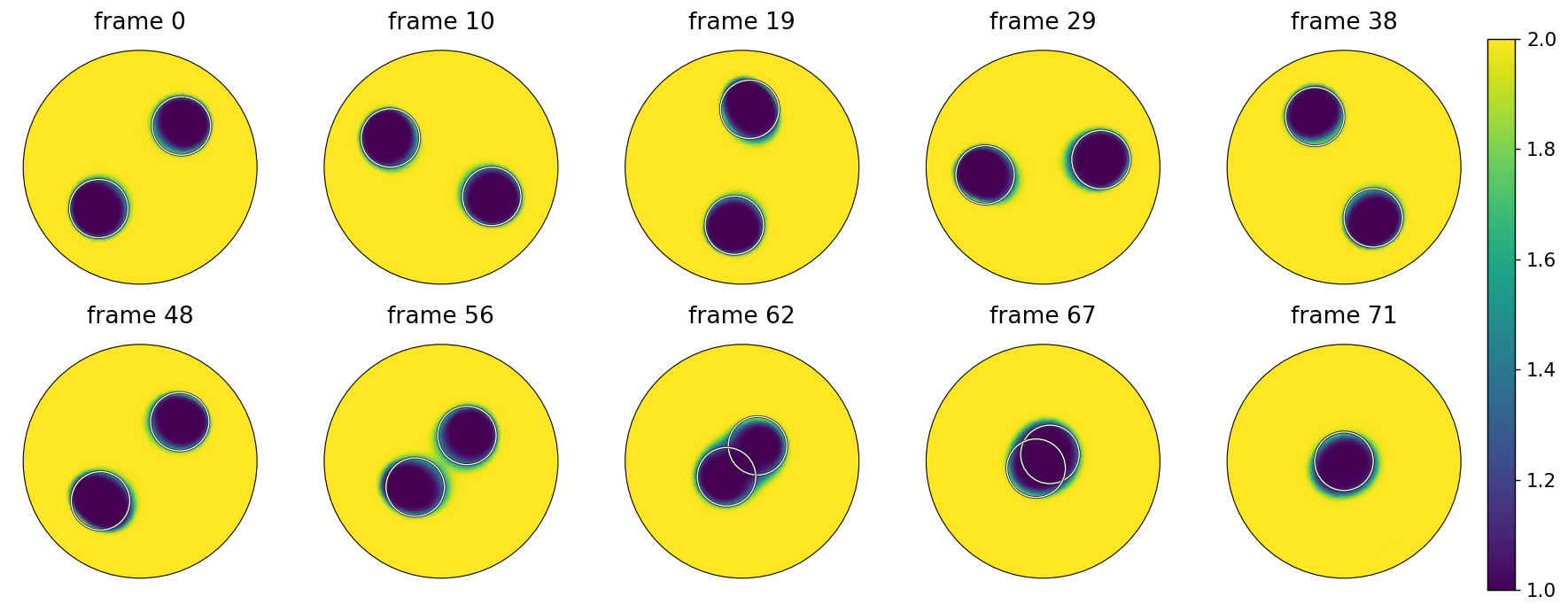}
    \caption{Posterior means of the star-shaped flow for ten of the $72$ frames of the orbit-and-merge sequence. Top: the clockwise orbit; bottom: the move to the center and the merge. White circles mark the true inclusions.}
    \label{fig:eit-dynamic-orbit}
\end{figure}

\end{document}